\documentclass[11pt, a4paper]{amsart}
\usepackage{amsmath,amssymb,amsthm,mathtools,wasysym,calc,verbatim,tikz,url,hyperref,mathrsfs,cite,fullpage,bbm,tikz,tikz-cd}
\usepackage[noabbrev,capitalize,nameinlink]{cleveref}
\usepackage[shortlabels]{enumitem}

\usepackage{comment}

\numberwithin{equation}{section}
\newtheorem{theorem}{Theorem}[section]
\newtheorem{lemma}[theorem]{Lemma}

\newtheorem{prop}[theorem]{Proposition}

\newtheorem{notation}[theorem]{Notation}

\newtheorem{claim}[theorem]{Claim}

\newtheorem{definition}[theorem]{Definition}

\newtheorem{problem}[theorem]{Problem}

\theoremstyle{definition}

\theoremstyle{remark}
\newtheorem{remark}[theorem]{Remark}
\newtheorem*{remark*}{Remark}

\newcommand{\snorm}[1]{\lVert#1\rVert}
\newcommand{\norm}[1]{\bigg\lVert#1\bigg\rVert}

\newcommand{\imod}[1]{~\mathrm{mod}~#1}
\newcommand{\eps}{\varepsilon}

\newcommand{\mb}{\mathbb}

\newcommand{\mbm}{\mathbbm}
\newcommand{\mc}{\mathcal}

\newcommand{\mr}{\mathrm}

\newcommand{\ol}{\overline}
\newcommand{\on}{\operatorname}

\newcommand{\wh}{\widehat}
\newcommand{\wt}{\widetilde}

\begin{document}

\title{Structured extensions and multi-correlation sequences}
\author{James Leng}
\address{Department of Mathematics, UCLA, Los Angeles, CA 90095, USA}
\email{jamesleng@math.ucla.edu}

\maketitle

\begin{abstract}
We show that every multi-correlation sequence is the sum of a generalized nilsequence and a null-sequence. This proves a conjecture of N. Frantzikinakis. A key ingredient is the reduction of ergodic multidimensional inverse theorems to analogous finitary inverse theorems, offering a new approach to the structure theory of multidimensional Host-Kra factors. This reduction is proven by combining the methods of Tao (2015) with the Furstenberg correspondence principle. We also prove the analogous multidimensional finitary inverse theorem with quasi-polynomial bounds.
\end{abstract}

\section{Introduction}
Let $\mathbf{X} = (X, \mathcal{X}, \mu)$ be a probability space with $\mathcal{X}$ countably generated, $T_1, \dots, T_k: X \to X$ be commuting measure preserving transformations, and $f_0, f_1, \dots, f_k \in L^\infty(\mathbf{X})$. A \emph{$k$-fold multi-correlation} sequence is a sequence of the form
$$c(n) = \int f_0\cdot T_1^nf_1 \cdots T_k^n f_k d\mu.$$
A result of Herglotz states that if $k = 1$, then there exists a measure $\lambda$ on $\mathbb{T}$ such that
$$c(n) = \int_{\mathbb{T}} z^n d\lambda(z).$$
A question of Frantzikinakis \cite[Problem 1, Problem 2]{Fra16} asks whether a generalization of this formula exists for multi-correlation sequences. Frantzikinakis describes finding an analogous formula as being of ``fundamental importance which has been in the mind of experts for several years. A satisfactory solution is going to give us new insights and significantly improve our ability to deal with difficult questions involving multiple ergodic averages which at the moment seem out of reach." We do not claim to solve this problem, but a related one, as we shall explain later.

Concretely, as remarked by Briet and Green \cite{GB21}, one potential application of such a formula is that it gives an avenue for attack for the \emph{random Szemer\'edi theorem}. In addition, the study of such sequences has led to various multiplicative number theoretic results in \cite{FH16, TT18, Sh23}.
\subsection{Previous results}
There are several partial and negative results in this direction. Before stating these results, we require the following definitions of a nilsequence and a null-sequence.
\begin{definition}[Nilsequence]
A \textbf{nilsequence} is a sequence of the form $n \mapsto c(n)$ where $c(n) = F(g^n\Gamma)$ with $F: G/\Gamma \to \mathbb{C}$ is a smooth function on a nilmanifold $G/\Gamma$ and $g \in G$. We say that $n \mapsto c(n)$ is \textbf{degree $k$} if $G$ is at most $k$-step nilpotent, that is, the $k$-fold iterated commutator $[G, [G, \dots, [G, G]]]$ vanishes. (So $G$ being two-step implies that $[G, [G, G]] = \mathrm{Id}_G$.)
\end{definition}
\begin{definition}[Generalized nilsequence]
We define a \textbf{generalized nilsequence} as a uniform limit of nilsequences. We specify that a \text{degree $k$ generalized nilsequence} is the uniform limit of $k$-step nilsequences. 
\end{definition}
\begin{remark}
The terminology \emph{generalized nilsequence} and \emph{nilsequence} are used slightly differently across different papers in the literature. This holds similarly with the term \emph{multi-correlation sequence}. For instance, \cite{Fra16} uses ``\emph{basic nilsequence}" for what we use as a ``nilsequence" if $F$ were weakened to be merely continuous. They also define a notion of ``\emph{basic generalized nilsequence}" if $F$ were further weakened to be Riemann integrable. Also, \cite{KLMR21} permits multi-correlation (without the hyphen) sequences to be of the form
$$c(n) = \int f_0 \cdot T^{\lfloor q_1(n)\rfloor}f_1 \cdots T^{\lfloor q_k(n)\rfloor}f_k d\mu$$
where $q_1, \dots, q_k: \mathbb{N} \to \mathbb{R}$ are polynomials. Additionally, several forms of these definitions are equivalent. For instance, while it is more common to define the underlying function $F$ of a nilsequence as being \emph{continuous} rather than \emph{smooth}, they in fact lead to the same definition of \emph{generalized nilsequence} because a continuous function may be uniformly approximated by a smooth function. We thus caution the reader to carefully check these definitions and their equivalences before applying results of this paper.
\end{remark}
\begin{definition}[Null-sequence]
We define a \textbf{null-sequence} as a sequence $a:\mathbb{Z} \to \mathbb{C}$ such that
$$\limsup_{N \in \infty} \frac{1}{2N + 1} \sum_{n = -N}^N |a(n)|^2 = 0.$$
\end{definition}
A famous result of Bergelson, Host, and Kra \cite{BHK05} shows that if $T_i = T_1^i$, then one can write $c(n)$ as the sum of a generalized nilsequence and a null-sequence. Inspired by this, Frantzikinakis \cite{Fra15} shows that given $\epsilon > 0$, every multi-correlation sequence can be written as the sum of a nil-sequence and a sequence $a(n)$ such that
$$\limsup_{N \to \infty} \frac{1}{2N + 1} \sum_{n = -N}^N |a(n)|^2 \le \epsilon.$$
We henceforth refer to such a decomposition statement as a \emph{weak decomposition statement} and a version with $a$ a null-sequence a \emph{strong decomposition statement}. Frantzikinakis poses the following problem in \cite[Problem 20]{Fra16}.
\begin{problem}\label{pro:Frantzikinakisproblem}
Is it true that every multi-correlation sequence is the sum of a nilsequence and a null-sequence?
\end{problem}
This problem has been studied and restated multiple times, e.g., \cite{FH16, MR17, Mo20, LMR21, Le20, KF22a, KLMR21, DMKS21}. In \cite{Mo20}, Moragues solves the problem under the assumption that $T_i$ and $T_iT_j^{-1}$ are simultaneously ergodic. In addition, Koutsogiannis, Le, Moreira, and Richter \cite{KLMR21} (and also work by the latter three authors \cite{LMR21}), Le \cite{Le20}, Frantzikinakis-Host \cite{FH16}, and Kuca-Frantzikinakis \cite{KF22a} prove decomposition statements for sequences of the form
$$\int f_0 \cdot T_1^{\varphi_1(n)}f_1 \cdots T_k^{\varphi_k(n)}f_k d\mu$$
for various functions $\varphi_i: \mathbb{N} \to \mathbb{N}$. \cite{LMR21, FH16} proves a weak decomposition statement for $\varphi_i$ various arithmetic functions (such as the primes), \cite{KF22a} proves a strong decomposition statement for $\varphi_i$ pairwise independent polynomials, \cite{KLMR21} proves a weak decomposition theorem for $\varphi_i$ polynomial orbits of primes.

There are also negative results in this direction. Frantzikinakis, Lesigne, and Weirdl \cite{FLW12} shows that \emph{any} bounded sequence $a(n)$ can be written as
$$a(n) = \int T^n g \cdot S^n h d\mu$$
for $S$ and $T$ not necessarily commuting. Furthermore, Briet and Green \cite{GB21} constructs a sequence for $k = 2$ and $T_2 = T_1^2$ that cannot be written as the integral combination of (continuous) nilsequences. This is moreso a negative result for \cite[Problem 2]{Fra16}, which conjectures that all multi-correlation sequences can be writtten as an integral combination of generalized nilsequences, where a different definition of ``nilsequences" is used: one which permits the underlying function of the nilsequence to be Riemann integrable. We remark, though, that \cite[Problem 2]{Fra16} is still an open problem, as Briet and Green's example is explicitly an integral combination of (Riemann integrable) nilsequences.

We finally mention interesting work of Shalom \cite{Sh23} which proves via an inequality of Grothendieck, a related decomposition (and an application to partition regularity) for $k = 2$ of not-necessarily commuting transformations.
\subsection{Main results}
Our main result is an affirmative answer to Problem~\ref{pro:Frantzikinakisproblem}.
\begin{theorem}\label{thm:maintheorem}
Every $k$-fold multi-correlation sequence is the sum of a degree $k$ generalized nilsequence and a null-sequence.
\end{theorem}
\begin{remark}
Via the argument of \cite[Proposition 4.5]{TT18}, one can upgrade this theorem to the statement that if $c(n)$ is the multi-correlation sequence and $c(n) = c_{\mathrm{nil}}(n) + c_{\mathrm{null}}(n)$ with $c_{\mathrm{nil}}(n)$ a generalized nilsequence and $c_{\mathrm{null}}(n)$ a null-sequence, then
$$\lim_{x \to \infty} \mathbb{E}_{p \le \mathbb{P} \cap [\pm x]} c_{\mathrm{null}}(p) = 0$$
where $\mathbb{P}$ denotes the primes.\footnote{We are indebted to Borys Kuca and Nikos Frantzikinakis for informing us of this application.}
\end{remark}
\begin{remark}
As this proof is rather technical, we have isolated the case of $k = 2$ in another paper \cite{Len24}.
\end{remark}
As Bergelson-Host-Kra's theorem was deduced from a structure theorem on (single dimensional) Host-Kra factors, we offer in the following result a structure theorem for multidimensional Host-Kra factors (see also Lemma~\ref{lem:joiningconstruction}), which is a key ingredient for the proof of Theorem~\ref{thm:maintheorem}. We remark that (multidimensional) Host-Kra factors are defined in Definition~\ref{def:HostKra}.
\begin{theorem}[Ergodic inverse theorem]\label{thm:ergodicinversetheorem}
Fix integers $k, j, j'$ with $0 < j + 1, j' + 1, k$ and $j \le k$. Let $(X, \mathcal{X}, \mu, \vec{T})$ be an ergodic $\mathbb{Z}^k$ system and let $\mathbf{Z} = \mathbf{Z}_{T_1, T_2, \dots, T_j, \vec{T}, \dots, \vec{T}}$ be the Host-Kra factor for the seminorm $\|\cdot \|_{T_1, T_2, \dots, T_j, \vec{T}, \dots, \vec{T}}$ with $j' + 1$ copies of $\vec{T}$. Then $\mathbf{Z}$ admits an extension to an ergodic system $(\widetilde{X}, \widetilde{\mathcal{X}}, \widetilde{\mu}, \widetilde{\vec{T}})$ with
$$\widetilde{\mathcal{X}} = I(\widetilde{T_1}) \vee I(\widetilde{T_2}) \vee \cdots \vee I(\widetilde{T_j}) \vee \Xi_{j + j', \mathrm{pronil}}$$
where $\Xi_{j + j', \mathrm{pronil}}$ is an inverse limit of $j + j'$-step $\mathbb{Z}^k$-nilfactors and if $S \in \mathrm{Aut}(\mathbf{X})$, $I(S)$ denotes the sigma algebra of $S$-invariant sets in $\wt{\mathcal{X}}$.
\end{theorem}
\begin{remark}
This theorem resembles several theorems of Austin in \cite{A09, A10b, A10, A15, A15b}, particularly that of \cite[Theorem 1.3]{A10}, where the ``directional CL factor" considered there roughly corresponds to the case of $j' = 1$, $j = 2$, and $k = 2$ in our theorem (though we stress that Austin arrives at the ``directional CL factors" considered there by considering different directional seminorms).
\end{remark}
\begin{remark}
Our proof of this theorem uses the norm convergence of multiple ergodic averages \cite{A09, Tao08}. This is perhaps not necessary; it is plausible with a bit more work that one can obtain a new proof of the norm convergence multiple ergodic averages using our method.
\end{remark}
We shall deduce the above theorem from the following finitary inverse theorem. See Section 2 below for notation.
\begin{theorem}\label{thm:maininversetheorem2}
Let $\delta > 0$, $k, \ell, \ell', K$ be integers with $1 \le \ell' + 1, k, \ell - 1 \le K$, and $f: [N]^k \to S^1$ be such that 
$$\|f\|_{U([N]^k, \dots, [N]^k, e_1[N], \dots, e_\ell[N])} \ge \delta$$
with $\ell' + 1$ copies of $[N]^k$. Then there exist one-bounded functions $f_1, \dots, f_\ell$ with $f_i$ not depending on coordinate $i$, and a degree $\ell + \ell'$ nilsequence $\chi$ of dimension $m$ and complexity $M$ such that
$$|\mathbb{E}_{n \in [N]^k} f(n)\chi(n)f_1(n)\cdots f_\ell(n)| \ge \epsilon$$
where
$$\epsilon^{-1}, M \le \exp(\log(1/\delta)^{O_K(1)}), m \le \log(1/\delta)^{O_K(1)}.$$
\end{theorem}
\begin{remark}
Mili\'cevi\'c proved a slightly more general inverse theorem on the setting of $\mathbb{F}_p^n$ with iterated exponential bounds in \cite{Mi21b} with $e_i[N]$ are replaced with disjoint subgroups. (As Mili\'cevi\'c works in the setting of $\mathbb{F}_p^n$, the obstructions obtained there are different, namely are phase polynomials rather than nilsequences.) A small modification of our methods would give a similar result also with quasi-polynomial bounds. 
\end{remark} 
The proof of Theorem~\ref{thm:maininversetheorem2} extensively uses techniques developed in \cite{GTZ11, GTZ12, Len23, LSS24b} and many ingredients we require follow proofs in \cite{LSS24b} rather closely. As such, many intermediate lemmas we prove will follow proofs of lemmas in \cite{LSS24b} essentially verbatim. 
\subsection{Historical remarks about the structure of characteristic factors}
We will now briefly explain some historical context behind Theorem~\ref{thm:ergodicinversetheorem}. We particularly wish to highlight the work of Austin \cite{A09, A10, A10b, A15, A15b}, which unfortunately has not received much follow-up. 

Let $T, T_1, T_2: X \to X$ be measure preserving transformations on a probability space $X$ with $T$ and $\vec{T} = (T_1, T_2)$ ergodic. In 1977, Furstenberg \cite{Fur77} proved that the \emph{Kronecker factor} $\mathcal{K}$, that is the maximal compact group factor, is \emph{characteristic} for the average
$$\mathbb{E}_{n \in [N]} T^n f_1 \cdot T^{2n} f_2$$
meaning that
$$\lim_{N \to \infty} \|\mathbb{E}_{n \in [N]} T^n f_1 \cdot T^{2n} f_2 - T^n \mathbb{E}(f_1|\mathcal{K}) \cdot T^{2n} \mathbb{E}(f_2 |\mathcal{K})\|_{L^2} = 0.$$
A natural followup question is to give a description of the characteristic factors of more complicated ergodic averages such as
\begin{equation}\label{eq:ergodicaverage1}
\mathbb{E}_{n \in [N]} T^n f_1 \cdot T^{2n} f_2\cdots T^{kn}f_k  
\end{equation}
\begin{equation}\label{eq:ergodicaverage2}
\mathbb{E}_{n \in [N]} T_1^n f_1 \cdot T_2^n f_2.    
\end{equation}
The former was done in work of Conze-Lesigne \cite{CL84, CL87, CL88}, Furstenberg-Weiss \cite{FW96}, Host-Kra \cite{HK05}, and Ziegler \cite{Zie07}, that described the characteristic factors as inverse limits of $k - 1$-step nilsystems. This description of these factors is analogous to the inverse theory of Gowers norms \cite{GTZ12} which states that a function has large Gowers norm if and only if it correlates with a nilsequence. Indeed, one can show that \eqref{eq:ergodicaverage1} is controlled by the \emph{Host-Kra-Gowers norm} (the ergodic cousin of the Gowers norm) in the sense that if any of the $f_i$ have zero Host-Kra-Gowers norm, then the $L^2$ limit as $N$ goes to infinity of \eqref{eq:ergodicaverage1} is zero.  

One can ask if the characteristic factor of \eqref{eq:ergodicaverage2} admits a similar description as one obtains from the \emph{box norm} $\|\cdot\|_{U(\mathbb{Z}_Nv_1, \mathbb{Z}_Nv_2)}$ with
$$\|f\|_{U(\mathbb{Z}_Nv_1, \mathbb{Z}_Nv_2)}^4 = \mathbb{E}_{x \in \mathbb{Z}_N^2} \mathbb{E}_{h_1, h_2 \in \mathbb{Z}_N} f(x)\overline{f(x + h_1v_1)f(x + h_2v_2)}f(x + h_1v_1 + h_2v_2)$$
for $v_1, v_2$ two vectors in $\mathbb{Z}^2$. By the pigeonhole principle, one see that if $\|f\|_{U(\mathbb{Z}_Nv_1, \mathbb{Z}_Nv_2)} \ge \delta$, then $f$ correlates with a product of a function invariant under translation by $v_1$ and a function invariant under translation by $v_2$. Since \eqref{eq:ergodicaverage2} is controlled by the analogous Host-Kra-Gowers box norm with $v_1 = e_1$ and $v_2 = e_2 - e_1$ at position $f_1$ and $v_1 = e_2$ and $v_2 = e_2 - e_1$ at position $f_2$, one can ask if the characteristic factor at $f_1$ is $I(\vec{T}^{v_1}) \vee I(\vec{T}^{v_2})$.

Unfortunately, this turns out not to hold; rather, the characteristic factor can be described as a (direct integral of) compact extension(s) of $I(\vec{T}^{v_1}) \vee I(\vec{T}^{v_2})$. This had been folklore since the aforementioned work of Conze and Lesigne, but was formally proven by Austin in \cite{A10b}. (As evidenced by the length of that work, especially when compared to the few pages of \cite[IV]{CL84} that prove the convergence of \eqref{eq:ergodicaverage2}, the problem of determining that characteristic factor is much more subtle than one might expect.) The precise description of Austin's characteristic factor is rather technical and involves a direct integral of homogeneous systems over the sigma algebra generated by the Vietoris topology of a compact group. Thus, we will not state the characteristic factor they obtained here.

In breakthrough work, Austin \cite{A09} proved that by taking an \emph{extension} of the system, one can ensure that the characteristic factors for \eqref{eq:ergodicaverage2} agree with $I(\vec{T}^{e_1}) \vee I(\vec{T}^{e_2 - e_1})$ at position $f_1$ and $I(\vec{T}^{e_2}) \vee I(\vec{T}^{e_2 - e_1})$ at position $f_2$. (Austin also proved an analogue of this result for $n$ commuting transformations.) As box norms are a special case of multidimensional Gowers norms, Austin's vision is clear: one can perhaps get characteristic factors which are amenable with (conjectured) inverse theorems for finitary multidimensional Gowers norms if one took an extension of the system. Indeed, \cite{A10, A15, A15b} carries out this vision in the specific case of the average
\[\mathbb{E}_{n \in [N]} T^{np_1}f_1 \cdot T^{np_2} f_2 \cdot T^{np_3} f_3\]
with $p_1, p_2, p_3$ integer vectors in $\mathbb{Z}^2$. 

It should be noted that the idea of taking an extension of the system leading to more structured characteristic factors (along with the term \emph{characteristic factor}) first appeared in the aforementioned work of Furstenberg and Weiss, where they showed that by taking a suitable extension of the system, one can take the homogeneous space $G/H$ coming from a compact extension (i.e., Furstenberg-Mackey-Zimmer theory) of the Kronecker factor to simply be a compact group $G'$. Their construction was later generalized to the non-ergodic setting by Austin in \cite{A10, A15} and is an essential component of the ``FIS extensions" considered there.

We conclude this subsection with one further technical point about the \emph{very remarkable} \cite{A15, A10}. The work of Austin involves significant cohomological difficulties in addition to dealing with fiberwise homogeneous spaces. It is likely that if Austin's methods were adapted to deal with the ergodic inverse theorems considered in this paper, this would similarly involve dealing with direct integrals of fiberwise spaces as well as the cohomological difficulties involved solving ``directional Conze-Lesigne-type" equations. 
\subsection{Discussion of Theorem~\ref{thm:ergodicinversetheorem}}
The purposes of Theorem~\ref{thm:ergodicinversetheorem} are twofold. First, we emphasize results for multidimensional Host-Kra factors rather than directly prove results for characteristic factors of ergodic averages. The reason is that multidimensional Host-Kra factors encompass characteristic factors for a large number of ergodic averages since one can control a large number of ergodic averages by multidimensional Host-Kra-Gowers norms. As such, Theorem~\ref{thm:ergodicinversetheorem} is more amenable to use directly than Austin's theorems in \cite{A09, A10, A15, A15b}. Our hope is that the formulation of Theorem~\ref{thm:ergodicinversetheorem} will lead to a more streamlined structure theory of characteristic factors for multidimensional ergodic averages.

Second, we give a different approach to obtaining structural information of characteristic factors; instead of taking an extension of the original system, we take an extension of the Host-Kra factor. Our approach is inspired by \cite{Tao15} and is proved by transferring a sufficient finitary inverse theorem to the ergodic setting. This avoids much of the potential difficulty in fiberwise homogeneous spaces and measurable group cohomology when proving our ergodic inverse theorem. However, we do not claim an ergodic inverse theorem analogous to \cite[Theorem 1.2]{A15b}. Instead, we suggest that finitary methods could be used to alleviate difficulties in the infinitary setting, especially when used in conjunction with infinitary methods. We hope to revisit this in future work.

The proof of Theorem~\ref{thm:ergodicinversetheorem} is a synthesis of Tao's argument in \cite{Tao15} (which to our knowledge is the first work deducing an ergodic inverse theorem from a finitary inverse theorem) and the Furstenberg correspondence principle. A key ingredient is the use various maximal inequalities such as the ergodic maximal inequality and Hardy-Littlewood maximal inequality to transfer to finitary functions. We offer a brief sketch here. One can show that the Host-Kra factors are generated by \emph{dual functions} of the form
$$\mathcal{D}(f) = \lim_{N \to \infty} \mathbb{E}_{n_1, \dots, n_{\ell'+ 1} \in [N]^k} \mathbb{E}_{h_1, \dots, h_\ell \in [N]} \prod_{\omega \in \{0, 1\}^{\ell' + \ell + 1} \setminus \{0\}} C^{|\omega|}\vec{T}^{\omega \cdot (n_1, \dots, n_{\ell' + 1}, h_1e_1, \dots, h_\ell e_\ell)}f$$
where $C$ denotes the conjugation operator and $f \in L^\infty(\mathbf{X})$. Since $\mathcal{X}$ is countably generated, we can restrict our attention to countably many $\mathcal{D}(f_n)$. The problem then becomes trying to find a system $(\mathbf{Y}, \vec{S})$ and functions $\widetilde{f}_i$ in $L^\infty(\mathbf{Y})$ such that which ``model" $\mathcal{D}(f_n)$, that is, for any multivariate polynomial $P$ and shifts $h_1, \dots, h_n$,
\begin{equation}
\int P(\vec{T}^{h_1}\mathcal{D}f_{k_1}, \dots, \vec{T}^{h_n}\mathcal{D}f_{k_n}) d\mu_{\mathbf{Z}} = \int P(\vec{S}^{h_1}\tilde{f}_{k_1}, \dots, \vec{S}^{h_n}\tilde{f}_{k_n}) d\mu_{\mathbf{Y}}.    
\end{equation}
(Here, $\mathbf{Z}$ denotes the Host-Kra factor.) Now consider the function for fixed generic $x$,
$$\mathcal{D}_N(f_n)(x) := \mathbb{E}_{n_1, \dots, n_{\ell'+ 1} \in [N]^k} \mathbb{E}_{h_1, \dots, h_\ell \in [N]} \prod_{\omega \in \{0, 1\}^{\ell' + \ell + 1} \setminus \{0\}} C^{|\omega|}\vec{T}^{\omega \cdot (n_1, \dots, n_{\ell' + 1}, h_1e_1, \dots, h_\ell e_\ell)}f_n(x).$$
This is anti-uniform in the $U([N]^k, \dots, [N]^k, e_1[N], \dots, e_\ell [N])$ norm, in that any function which correlates with $h \mapsto \vec{T}^h\mathcal{D}_N(f_n)(x)$ has large $U([N]^k, \dots, [N]^k, e_1[N], \dots, e_\ell [N])$ norm. Hence, we can approximate $h \mapsto \vec{T}^h\mathcal{D}_N(f_n)(x)$ via the arithmetic regularity lemma by a sum of structured functions coming from Theorem~\ref{thm:maininversetheorem2}, that is, functions of the form
$$\chi_{x}(h) = \sum_{i = 1}^{D} (f_1)^i_{x}(h)\cdots (f_\ell)^i_{x}(h)F^i_{x}(g_{x}^i(h)\Gamma^i_x)$$
such that $F^i_{m, n, \ell, x}(g_{m, n, \ell, x}(h)\Gamma)$ is a nilsequence and $f_j$ does not depend on the $j$th component. Now, through the use of various pointwise estimates (i.e., maximal inequalities), there exists $x$ such that one can approximately model the $L^2$ limit $\mathcal{D}(f_n)$ by a family of $h \mapsto \vec{T}^h\mathcal{D}_N(f_n)(x)$ (with the approximations getting better as $N$ increases) and thus also model them by $\chi_x$ (again, with better approximations as $N$ tends to infinity). It then suffices to consider what happens to $\chi_x$ as $N$ goes to infinity. By taking successive products of the underlying nilmanifolds, the nilsequence part $F^i_x(g_x^i(h)\Gamma^i_x)$ becomes a ``pro-nilsequence" as $N$ to infinity, and an application of the Furstenberg correspondence principle models $(f_j)^i_x$ on functions invariant under $S_j$; taking a limit as $N$ goes to infinity does not affect our ability to use the Furstenberg correspondence principle as we can simply take an infinite product of the Furstenberg system obtained at each finitary scale. This is our chief innovation in the proof of Theorem~\ref{thm:ergodicinversetheorem} given Theorem~\ref{thm:maininversetheorem2}.

Using our methods, it is possible to deduce a similar ergodic inverse theorem for any Host-Kra factor, given an appropriate finitary inverse theorem. Since we are able to deduce a large class of finitary inverse theorems via Theorem~\ref{thm:maininversetheorem2}, we are able to deduce a large class of new ergodic inverse theorems not covered in the work of Austin.\footnote{It should also be noted that Austin's results are not stated in terms of the Host-Kra factors, but rather, factors arising from Furstenberg self-joinings and hence the inverse theorems for Host-Kra factors relevant to Austin's work \cite{A10, A15, A15b} are open though the relevant structural extension result for the Host-Kra analogue of \cite{A09} was proven by Host in \cite{Ho09}.}

\subsection{Discussion of the proof of Theorem~\ref{thm:maininversetheorem2}}
Our proof of Theorem~\ref{thm:maininversetheorem2} proceeds inductively. First, we prove Theorem~\ref{thm:maininversetheorem} and then using a similar approach to \cite{GTZ12} to upgrade this to Theorem~\ref{thm:maininversetheorem2}. Both of these approaches require a structure theorem for ``1\% additive quadruples attached to multidimensional nilsequences", which we provide in Theorem~\ref{thm:nilsequenceconstruction}. This theorem builds on work of \cite{GTZ12, LSS24b} (and is hence why we use their techniques extensively). This differs from the approach of \cite{Mi21b} since the approach there aims to give a result for ``multidimensional multilinear Freiman homomorphisms" which we avoid needing to do by arguing inductively. 

\subsection{Further directions}
A multidimensional variant of Theorem~\ref{thm:maintheorem} should be rather straightforward from methods of this paper. One would require the aforementioned small modifications to Theorem~\ref{thm:maininversetheorem2}. However, our inverse theorem is unlikely to be sufficient for polynomial multi-correlation sequences. In future work, we hope to simplify and extend Austin's work \cite{A10b, A15}.
\subsection{Organization of the paper}
Section 2 will contain the notation and conventions of the paper. To prove Theorem~\ref{thm:maininversetheorem2}, we will require Theorem~\ref{thm:maininversetheorem}. This will be stated in Section 3. Sections 4 and 5 will be devoted to proving Theorem~\ref{thm:maininversetheorem} assuming a key ingredient, Theorem~\ref{thm:nilsequenceconstruction}. This theorem will be proven in Section 7. Sections 6 and 8 will be devoted to proving Theorem~\ref{thm:maininversetheorem2}. Finally, in Section 9, we prove Theorem~\ref{thm:ergodicinversetheorem} and in Section 10, we prove Theorem~\ref{thm:maintheorem}. The paper is structured in a way that Sections 3-8 (the ``combinatorial portion") and 9-10 (the ``ergodic portion") can be read independently.

In Appendix A, we detail a few auxiliary lemmas used in the proof and in Appendix B, we will state a few other auxiliary lemmas in additive combinatorics. Finally, in Appendix C, we detail how given a structured extension result as in Theorem~\ref{thm:ergodicinversetheorem}, one can deduce a pleasant extensions result as in \cite{A09}.

\subsection{Acknowledgements}
We would like to thank Terry Tao for advisement, Tim Austin for comments and numerous discussions and encouragement regarding his papers \cite{A09, A10, A10b, A15, A15b}, Borys Kuca for helpful comments and encouragement, especially during the preparation of Section~\ref{sec:structuredtopleasant}, Nikos Frantzikinakis, Bryna Kra, and Mehtaab Sawhney for comments. The author was supported by a UCLA dissertation year fellowship and Terry Tao's NSF Grant No. 2347850 while this research was conducted.

\section{Notation and preliminaries}
As the proof of Theorem~\ref{thm:maininversetheorem2} uses extensively the techniques of \cite{GTZ11, GTZ12, Len23, LSS24b}, we require much of the notation in \cite{LSS24b}. Some notation there is standard, such as notions of dimension and complexity of a nilmanifold. We will not restate these notions here and refer the reader to (e.g.,) \cite[Section 2]{Len23b}, \cite[Section 3]{LSS24b}, and \cite{GT12}. However, some are not as standardized, notably the \emph{degree-rank filtration}, and we will restate them here. As this section is rather extensive, we remark that the reader interested only in the ergodic portion can skip to Section 2.6.
\subsection{Basic notation}
We work with the following.
\begin{itemize}
    \item We denote $[N] = \{0, \dots, N - 1\}$.
    \item We denote $[\pm N] = \{-(N - 1), \dots, 0, \dots, N - 1\}$.
    \item We denote $\mathbb{E}_{a \in A} = \frac{1}{|A|} \sum_{a \in A}$.
    \item A one-bounded function denotes a function whose absolute value is bounded above by $1$.
    \item For a finite set $A$, we denote $L^p(A)$ to be the $L^p$ norm with respect to the \textbf{normalized} counting measure on $A$. For example,
    $$\|f\|_{L^2(A)}^2 = \mathbb{E}_{a \in A} |f(a)|^2.$$
    \item For a finite set $A$, we denote $\ell^p(A)$ to be the $\ell^p$ norm with respect to the \textbf{unnormalized} counting measure on $A$. For example,
    $$\|f\|_{\ell^2(A)}^2 = \sum_{a \in A} |f(a)|^2.$$
\end{itemize}

\subsection{Filtrations}
The following section essentially follows \cite[Section 2.2]{LSS24b} verbatim.
\begin{definition}\label{def:arb-filtration}
An \textbf{ordering} $I = (I,\preceq,+,0)$ is a set $I$ with a distinguished element $0$, binary operation $+\colon I\times I\to I$, and a partial order $\preceq$ on $I$ such that   
\begin{itemize}
    \item $+$ is associative and commutative with $0$ acting as an identity element;
    \item $\preceq$ has $0$ as the minimal element;
    \item For all $i,j,k\in I$, if $i\preceq j$ then $i+k\preceq j+k$;
    \item The initial segments $\{i\in I\colon i\preceq d\}$ are finite for all $d$.
\end{itemize}
We define the following three orderings, with addition being the standard addition:
\begin{itemize}
        \item The \textbf{degree} ordering is given by the standard ordering on $\mb{N}$, denoted $I=\mb{N}$ for short;
        \item The \textbf{degree-rank} ordering is given by $\{(d,r)\in\mb{N}^2\colon 0\le r\le d\}$ with the ordering that $(d',r')\preceq(d,r)$ if $d'<d$ or $d' = d$ and $r'\le r$, denoted $I=\mr{DR}$ for short;
        \item The \textbf{multidegree} ordering is given by $\mb{N}^k$ with $(i_1',\ldots,i_k')\preceq(i_1,\ldots,i_k)$ when $i_j'\le i_j$ for all $1\le j\le k$, denoted $I=\mb{N}^k$ for short;
        \item The \textbf{multidegree-rank} ordering is given by $\mb{N}^{k + 1}$ with $\{(d_1, \dots, d_k, r): 0 \le r \le d_1 + \cdots + d_k\}$ and $(i_1', \dots, i_k', r') \preceq (i_1, \dots, i_k, r)$ when $i_j' < i_j$ for all $1 \le j \le k$ or if $i_j' = i_j$ for all $1 \le j \le k$, and $r' \le r$. We denote this as $I = \mr{MDR}$ for short.
\end{itemize}
An \textbf{$I$-filtration of $G$} is a collection of subgroups $G_I=(G_i)_{i\in I}$ such that $G_0=G$ and:
\begin{itemize}
    \item (Nesting) If $i,j\in I$ are such that $i\preceq j$ then $G_i\geqslant G_j$;
    \item (Commutator) For $i,j\in I$, we have $[G_i,G_j]\leqslant G_{i+j}$.
\end{itemize}
We say that a filtered group $G$ has \textbf{degree $\le d$} (for $d\in I$) if $G_i$ is trivial for $i\not\preceq d$. $G$ has \textbf{degree $\subseteq J$} for a downset $J$ if $G_i$ is trivial whenever $i\notin J$.
\end{definition}
\begin{remark}
The multidegree-rank filtration is rarely used and not strictly necessary for our paper, only being used in Lemma~\ref{lem:splitdegreerank}. We included it because it lead to a slightly less cumbersome statement of that lemma and because it was a natural extension of the \emph{degree-rank} filtration.
\end{remark}
Note that the commutator condition implies nested subgroups are normal within each other. We next define degree, degree-rank, and multidegree filtrations.
\begin{definition}\label{def:filt-prec}
Given $d\in\mb{N}$, we say a group $G$ is given a \textbf{degree filtration of degree $d$} if:
\begin{itemize}
    \item $G$ is given a $\mb{N}$-filtration $(G_i)_{i\in\mb{N}}$ with degree $\le d$;
    \item $G_0 = G_1$.
\end{itemize}
Given $(d,r)\in\mb{N}^2$ with $0\le r\le d$, $G$ is given a \textbf{degree-rank filtration of degree-rank $(d,r)$} if:
\begin{itemize}
    \item $G$ is given a $\mr{DR}$-filtration $(G_i)_{i\in\mr{DR}}$ with degree $\le(d,r)$;
    \item $G_{(0,0)}=G_{(1,0)}$ and $G_{(i,0)}=G_{(i,1)}$ for $i\ge 1$. (We also let $G_{(i,j)}=G_{(i+1,0)}$ for $j>i$.)
\end{itemize}
The associated degree filtration with respect to this degree-rank filtration is $(G_{(i,0)})_{i\ge 0}$.

Given $(d_1,\ldots,d_k)\in\mb{N}^k$, $G$ is given a \textbf{multidegree filtration of multidegree $J$} (where $J\subseteq\mb{N}^k$ is a downset) if:
\begin{itemize}
    \item $G$ is given a $\mb{N}^k$-filtration $(G_i)_{i\in\mb{N}^k}$ with degree $\subseteq J$;
    \item $G_{\vec{0}} = \bigvee_{i=1}^k G_{\vec{e_i}}$.
\end{itemize}
The associated degree filtration with respect to the multidegree filtration is $(\bigvee_{|\vec{i}|=i}G_{\vec{i}})_{i\ge 0}$. Here $|\vec{i}| = i_1 + \ldots + i_k$. For a \textbf{multidegree-rank filtration $(J, r)$} with $J \subset \mb{N}^k$ a downset, we specify
\begin{itemize}
    \item $G$ is given a $\mr{MDR}$ filtration, $(G_{i})_{i \in \mr{MDR}}$ with degree-rank $\subseteq \bigcup_{(d_1, \dots, d_k) \in J \text{ maximal}} \{(d_1, \dots, d_k, r)\}$.
    \item $G_{(\vec{0}, 0)} = \bigvee_{i = 1}^k G_{(\vec{e_i}, 0)}$ and $G_{(\vec{v}, 0)} = G_{(\vec{v}, 1)}$ for all $|\vec{v}| \ge 1$.
    \item (For $\vec{d} = (d_1, \dots, d_k)$, we also set $G_{(d_1, \dots, d_k, r')} = \bigvee_{i = 1}^k G_{(\vec{d} + e_i, 0)}$ for $r' > d_1 + \cdots + d_k$.)
\end{itemize}
\end{definition}

We now define polynomial sequences of an $I$-filtered group. The notion of a polynomial sequence for a group $G$ given a degree-rank filtration will be the same as treating this ordering as a $\mr{DR}$-filtration; the same applies for degree and multidegree filtrations. 
\begin{definition}\label{def:I-polynomial}
Given $g\colon H\to G$ a map between groups (not necessarily a homomorphism) and $h\in H$, we define the derivative $\partial_hg\colon H\to G$ via $\partial_hg(n)=g(hn)g(n)^{-1}$ for all $n\in H$. If $H,G$ are $I$-filtered, we say that this map $g$ is \textbf{polynomial} if for all $m\ge 0$ and $i_1,\ldots,i_m\in I$, we have
\[\partial_{h_1}\cdots\partial_{h_m}g(n)\in G_{i_1+\cdots+i_m}\]
for all choices of $h_j\in H_{i_j}$ and $n\in H_0$. The space of all polynomial maps with respect to this data is denoted $\on{poly}(H_I\to G_I)$.
\end{definition}

We will require various general properties of polynomial sequences established in \cite[Appendix~B]{GTZ12}. We will only consider $H=\mb{Z}^k$ for $k\ge 1$ and the following $I$-filtrations on $H$.

\begin{definition}\label{def:Z-filtration}
We define the following filtrations on $H=\mb{Z}^k$:
\begin{itemize}
    \item The \textbf{(domain) degree filtration} is with $I=\mb{N}$ the degree ordering and $H_0=H_1=\mb{Z}^k$, and $H_i=\{0\}$ for $i\ge 2$;
    \item The \textbf{(domain) multidegree filtration} is with $I=\mb{N}^k$ the multidegree ordering, $H_{\vec{0}}=\mb{Z}^k$, $H_{\vec{e}_i}=\mb{Z}\vec{e}_i$ for $i\in[k]$, and $H_{\vec{v}}=\{0\}$ otherwise, where $\vec{e}_i$ forms the standard basis of $\mb{Z}^k$;
    \item The \textbf{(domain) degree-rank filtration} is with $I=\mr{DR}$ the degree-rank ordering and $H_{(0,0)}=H_{(1,0)}=\mb{Z}^k$ and $H_{(d,r)}=\{0\}$ otherwise.
    \item The \textbf{(domain) multidegree-rank filtration} is with $I = \mathbb{N}^{k + 1}$ the multidegree-rank ordering and $H_{(\mathbf{0}, 0)} = \mathbb{Z}^k$, $H_{(\vec{e}_i, 0)} = H_{(\vec{e_i}, 1)} = \mathbb{Z}\vec{e}_i$ for all $i \in [k]$, and $H_{(\vec{v}, r)} = \{0\}$ otherwise with $\vec{e}_i$ forming the standard basis of $\mathbb{Z}^k$.
\end{itemize}
\end{definition}

\subsection{Nilcharacters}
We also require notions of nilcharacters as developed in \cite{GTZ12}. The following definition appears as \cite[Definition 2.15]{LSS24b}. (Note that we make a distinction between vertical characters and nilcharacters in this paper; this distinction is \emph{not} made in several other papers in this area and we caution the reader regarding this point.)
\begin{definition}\label{def:vert-char}
Consider a nilmanifold $G/\Gamma$ and a function $F\colon G/\Gamma\to\mb{C}$. Given a connected, simply connected subgroup $T$ of the center $Z(G)$ which is rational (i.e., $\Gamma\cap T$ is cocompact in $T$) and a continuous homomorphism $\eta\colon T\to\mb{R}$ such that $\eta(T\cap\Gamma)\subseteq \mb{Z}$, if 
\[F(gx)=e(\eta(g))F(x)\emph{ for all }g\in T\]
we say that $F$ is a \textbf{$T$-vertical character} with \textbf{$T$-vertical frequency} $\eta$.
\end{definition}
\begin{notation}
If a filtration on $G/\Gamma$ is implicit and has a specified largest element, we shall refer to a \textbf{vertical character} as a $G_J$-vertical character and a \textbf{vertical frequency} as a $G_J$-vertical frequency, with $J$ the largest element of the filtration.
\end{notation}
The following definition appears as \cite[Definition 2.16]{LSS24b}.
\begin{definition}\label{def:nilcharacter}
A \textbf{nilcharacter} of degree $d$ and output dimension $D$ is the following data. Consider an $I$-filtered nilmanifold $G/\Gamma$ of degree $d$ such that $[G,G_d] = \mr{Id}_G$ and an $I$-filtered abelian group $H$. Let $g\in\on{poly}(H_I\to G_I)$ and consider function $F\colon G/\Gamma\to\mb{C}^D$ such that:
\begin{itemize}
    \item $\snorm{F(x)}_2 = 1$ for all $x\in G/\Gamma$ pointwise;
    \item $F(g_dx) = e(\eta(g_d))F(x)$ for all $g_d\in G_d$ where $\eta$ is some continuous homomorphism $G_d\to\mb{R}$ such that $\eta(\Gamma\cap G_d) \subseteq\mb{Z}$.
\end{itemize}
The values of the nilcharacter are given by $\chi\colon H\to\mb{C}^D$ where $\chi(n)=F(g(n)\Gamma)$ for $n\in H$.
\end{definition}
In the above definition, we highlight the condition $\snorm{F(x)}_2 = 1$. This turns out to be useful in the proof of Lemma~\ref{lem:equiv}. The below definition appears as \cite[Definition 7.3]{LSS24b}.
\begin{definition}\label{def:equiv}
We say nilcharacters $\chi,\chi'$ are \textbf{$(M,D,d)$-equivalent for multidegree $J$} (or simply $(M, d)$-equivalent if $D$ is unambiguous) if $\chi,\chi'$ have output dimensions bounded by $D$ and all coordinates of 
\[\chi\otimes \ol{\chi'}\]
can be represented as sums of at most $M$ nilsequences of multidegree $J$ such that the underlying functions of each nilsequence are $M$-Lipschitz and the underlying nilmanifolds have complexity bounded by $M$ and dimension bounded by $d$.
\end{definition}
The following lemma appears as \cite[Lemma 7.4]{LSS24b}.
\begin{lemma}\label{lem:equiv}
Given a function $f\colon\Omega\to\mb{C}^{L}$ and nilcharacters $\chi,\chi'$ which are $(M,D,d)$-equivalent for multidegree $J$, if
\[\snorm{\mb{E}_{\vec{n}\in\Omega} f(\vec{n}) \otimes \chi(\vec{n})}_{\infty}\ge\rho\]
then 
\[\snorm{\mb{E}_{\vec{n}\in\Omega} f(\vec{n}) \otimes \chi'(\vec{n}) \cdot\psi(\vec{n})}_{\infty}\ge (\rho/(MD))^{O(1)},\]
where $\psi$ can be taken to be one of the nilsequences used as part of a represention of one of the coordinates in $\chi\otimes\ol{\chi'}$. In particular, $\psi$ is a nilsequence of multidegree $J$ such that underlying nilmanifold has complexity bounded by $M$ and dimension bounded by $d$ and the underlying function has Lipschitz constant bounded by $M$.
\end{lemma}
\subsection{Multidimensional Taylor Expansions}
A key difference between the proofs of this multidimensional inverse theorem and the proof of \cite{LSS24b} is that we must use multi-dimensional Taylor coefficients of nilsequences.
\begin{definition}[Multidimensional Taylor coefficients]
Let $g$ be a $\mathbb{Z}^L$-polynomial sequence on $G$ equipped with a $(d, r)$ degree-rank filtration and $\vec{i} = (i_1, \dots, i_L)$ be an integer vector. We define
$$\mathrm{Taylor}_{\vec{i}}(g) = \partial^{i_1}_{e_1} \partial^{i_2}_{e_2} \cdots \partial^{i_L}_{e_L}(g) \mod G_{(|\vec{i}|, 2)}$$
where $|\vec{i}| = i_1 + i_2 + \cdots + i_L$.
\end{definition}
We have the following results of multidimensional Taylor coefficients. Our proof follows closely that of \cite[Lemma 2.12]{LSS24b}.
\begin{lemma}\label{lem:taylorexansion}
Let $G$ be given a $(d, r)$ degree-rank filtration and let $g$ be a $\mathbb{Z}^L$-polynomial sequence on $G$. Then we may write
$$g(n) = \prod_{|\vec{i}| < d} g_{\vec{i}}^{{n \choose \vec{i}}}$$
with $g_{\vec{i}} \equiv \mathrm{Taylor}_{\vec{i}}(g) \mod G_{(|\vec{i}|, 2)}$ (and the order of the product does not matter).
\end{lemma}
\begin{proof}
Consider the filtered group $\widetilde{G} = G/G_{(|\vec{i}|, 2)}$ equipped with $\widetilde{G}_j = G_{(j, 1)}/G_{(|\vec{i}|, 2)}$. We see that $[G_{(j, 1)}/G_{(|\vec{i}|, 2)}, G_{(|\vec{i}| - j, 1)}/G_{(|\vec{i}|, 2)}] = \mathrm{Id}$. Thus, $[\widetilde{G}_{j}, \widetilde{G}_{|\vec{i}| - j}] = \mathrm{Id}$. Thus, letting $\widetilde{g}(n) = g(n) \mod G_{(|\vec{i}|, 2)}$, we see that
$$\widetilde{g}(n) = \prod_{|\vec{i}| \ge |\vec{j}| > 0} \widetilde{g_{\vec{j}}}^{{n \choose \vec{j}}}$$
where $\widetilde{g}_{\vec{j}} = g_{\vec{j}} \mod G_{(|\vec{i}|, 2)}$. Thus, via induction, it suffices to show that
$$\partial_{e_\ell} g(n) = \prod_{|\vec{i}| - 1 \ge |\vec{j}| > 0} \widetilde{g_{\vec{j}}'}^{{n \choose \vec{j}}}$$
for some $\widetilde{g_{\vec{j}}'} \in \widetilde{G}_{|\vec{j}| + 1}$ with $\widetilde{g'}_{\vec{i} - e_\ell} = \widetilde{g}_{\vec{i}}$. This follows from Baker-Campbell-Hausdorff. The crucial reason that $\widetilde{g'}_{\vec{i} - e_\ell} = \widetilde{g}_{\vec{i}}$ is that any ``higher order'' terms which arise in the Baker--Campbell--Hausdorff formula and could contribute are in fact annhilated due to $[\wt{G}_{j},\wt{G}_{|\vec{i}|-j}]=\mr{Id}$, so one can effectively treat the top order term as lying in an abelian group in this proof.
\end{proof}

\begin{lemma}\label{lem:taylorhom}
We have $\mathrm{Taylor}_{\vec{i}}(gh) = \mathrm{Taylor}_{\vec{i}}(g)\mathrm{Taylor}_{\vec{i}}(h)$ and if $g(n) = \exp\left(\sum_{\vec{i}} g_{\vec{i}} {n \choose \vec{i}}\right)$ for $g_{\vec{i}} \in \log(G_{|\vec{i}|})$, then $\mathrm{Taylor}_{\vec{i}}(g) = g_{\vec{i}} \mod G_{(|\vec{i}|, 2)}$.    
\end{lemma}
\begin{proof}
The first part of the lemma has a similar proof to the previous lemma, where we use the Baker-Campbell-Hausdorff theorem and crucially we observe that $[\wt{G}_{j},\wt{G}_{|\vec{i}|-j}]=\mr{Id}$ with $\widetilde{G}$ defined in the proof of Lemma~\ref{lem:taylorexansion}.

For the second claim, suppose that
\[g(n) = \exp\bigg(\sum_{\vec{i}} g_{\vec{i}} \binom{n}{\vec{i}}\bigg) = \prod_{\vec{i}}(g_{\vec{i}}')^{\binom{n}{\vec{i}}}.\]
Via iterated applications of the Baker--Campbell--Hausdorff formula and the commutator relationship that $[G_{(\ell,0)},G_{(j-\ell,0)}]=[G_{(i,1)},G_{(j-\ell,1)}]\subseteq G_{(j,2)}$, we see that $g_{\vec{i}}' = \exp(g_{\vec{i}}) \imod G_{(|\vec{i}|,2)}$ and the result follows. 
\end{proof}
\subsection{Miscellaneous nilsequence notation}
We will finally state the following miscellaneous nilsequence notation.
\begin{notation}
Given a function $f$ and a natural number $L > 0$ with domain $X^L$, we will write $f(n)$ for $n = (n_1, \dots, n_L) \in X^L$ instead of $\vec{n} = (n_1, \dots, n_L)$. 
\end{notation}
\begin{definition}
We will define the following sets of nilsequences.
\begin{itemize}
    \item We denote $\mathrm{Nil}^d(M, m, L, D)$ the set of all $\le M$-Lipschitz $\mathbb{Z}^L$-nilsequences with output dimension $D$ on a nilpotent Lie group equipped with an adapted Mal'cev basis $\{X_1, \dots, X_m\}$ with complexity at most $M$, dimension at most $m$, and degree at most $d$. 
    \item We denote $\mathrm{Nil}^{(d, r)}(M, m, L, D)$ the set of all $\le M$-Lipschitz $\mathbb{Z}^L$-nilsequences with output dimension $D$ on a nilpotent Lie group with degree rank at most $(d, r)$ equipped with an adapted Mal'cev basis $\{X_1, \dots, X_m\}$ with complexity at most $M$, dimension at most $m$, and degree-rank at most $(d, r)$.
    \item We denote $\mathrm{Nil}^{(d_1, \dots, d_k)}(M, m, L, D)$ the set of all $\le M$-Lipschitz $\mathbb{Z}^L$-nilsequences with output dimension $D$ on a nilpotent Lie group with multidegree at most $(d_1, \dots, d_k)$ equipped with an adapted Mal'cev basis $\{X_1, \dots, X_m\}$.
    \item We denote $\mathrm{Nil}^{<J}(M, m, L, D) = \bigcup_{I < J} \mathrm{Nil}^{I}(M, m, L, D)$.
\end{itemize}
(By adapted Mal'cev basis, we require in the first two cases, the Mal'cev basis to have the degree $d$ nesting property and the third case to have the $d_1 + \cdots + d_k$-nesting property as in \cite[Definition B.2]{Len23b}.)
\end{definition}
\begin{remark}
This definition runs into the unfortunate issue that the degree-rank and ``double degree" set of nilsequences are indistinguishable from a symbolic point of view. We will make an effort to specify which one of these sets it is each time we use this notation.
\end{remark}
\begin{remark}
As the multidegree-rank filtration is rarely used in this paper, we will not define a specific notation for that notion.
\end{remark}
We require the following definition in the proof of Theorem~\ref{thm:maininversetheorem}.
\begin{definition}
A nilsequence $\chi : \mathbb{Z}^k \to \mathbb{C}$ is a $\mathbb{Z}$-multidegree nilsequence of degree $(d_1, \dots, d_k)$ if it can be written as a nilsequence $(h_1, \dots, h_k) \mapsto \chi'(h_1, \dots, h_k)$ where $\chi'$ is a multidegree $(d_1, \dots, d_k)$ and $h_1, \dots, h_k \in \mathbb{Z}$.
\end{definition}

\subsection{Box seminorms}
We will next require some notation and preliminaries on box norms.
\begin{definition}[Discrete multiplicative derivatives]
Let $G$ be a group. Given a function $f:G \to \mathbb{C}$ and $h, h_1, \dots, h_\ell \in G$, we define $\Delta_h f(x) = \overline{f(x + h)}f(x)$ nad $\Delta_{h_1, \dots, h_\ell}f(x) = \Delta_{h_1}(\Delta_{h_2} \cdots (\Delta_{h_\ell}f(x)))$. We also define $\Delta_{(h, h')} f(x) = \overline{f(x + h')}f(x + h)$ and analogously define $\Delta_{(h_1, h_1'), \dots, (h_\ell, h_\ell')}$.
\end{definition}
We note that $\Delta_{h_1, \dots, h_\ell}$ is invariant under the order of $(h_1, \dots, h_\ell)$ and similarly $\Delta_{(h_1, h_1'), \dots, (h_\ell, h_\ell')}$ are invariant under the order of $((h_1, h_1'), \dots, (h_\ell, h_\ell'))$.
\begin{definition}[Box (semi)norms]
Let $G$ be a finite group. Given $f: G \to \mathbb{C}$ and subsets $H_1, \dots, H_\ell$ of $G$, we define
$$\|f\|_{U(H_1, \dots, H_\ell)} = \|f\|_{H_1, \dots, H_\ell} = (\mathbb{E}_{x \in G}\mathbb{E}_{h_i, h_i' \in H_i} \Delta_{(h_1, h_1'), (h_2, h_2') \dots, (h_\ell, h_\ell')}f(x))^{1/2^\ell}.$$
\end{definition}
It is classical that $\|\cdot\|_{H_1, \dots, H_\ell}$ satisfies a \emph{Cauchy-Schwarz-Gowers}-type inequality and therefore are seminorms. We record the result in the following.
\begin{lemma}\label{lem:CauchySchwarzGowers}
Let $G$ be a finite group. For $\omega \in \{0, 1\}^\ell$ and $f_\omega: G \to \mathbb{C}$, let
$$\langle (f_\omega)_{\omega \in \{0, 1\}^\ell}\rangle_{H_1, \dots, H_\ell} = \mathbb{E}_{x \in G}\mathbb{E}_{h_i, h_i' \in H_i} \prod_{\omega \in \{0, 1\}^\ell} \mathcal{C}^{|\omega|}f(x + (\vec{h}_i - \vec{h}_i')\cdot \omega + h_1' + \cdots + h_\ell')$$
where $\mathcal{C}$ is the conjugation operator and $\vec{h} = (h_1, \dots, h_\ell)$. Then
$$\langle (f_\omega)_{\omega \in \{0, 1\}^\ell} \rangle \le \prod_{\omega \in \{0, 1\}^\ell} \|f_\omega\|_{H_1, \dots, H_\ell}.$$
In addition $\|\cdot\|_{H_1, \dots, H_\ell}$ is a seminorm.
\end{lemma}
We finally require some asymptotic notation that we shall use.
\begin{notation}\label{not:asymptoticnotation}
Given a parameter $K$, we denote $M(\delta), \epsilon^{-1}(\delta)$ to be any quantity $\ll \exp(\log(1/\delta)^{O_{K}(1)})$, and $m(\delta)$ to be any quantity $\ll \log(1/\delta)^{O_{K}(1)}$.
\end{notation}
\begin{remark}
    A word of caution to the reader. In various statements below, slightly different quantities will be used to denote the same quantity of $M(\delta)$, $m(\delta)$, and $\epsilon(\delta)$. The point of these notions is so that one can substitute $M(\delta), m(\delta), \epsilon(\delta)$ with the values of $\exp(\log(1/\delta)^{O_{K}(1)})$, $\log(1/\delta)^{O_{K}(1)}$, and $\exp(-\log(1/\delta)^{O_{K}(1)})$, respectively.
\end{remark}

\subsection{Measure theoretic notation}
We denote a probability space by a triple $\mathbf{X} = (X, \mathcal{X}, \mu)$ with $X$ the space, $\mu$ the measure, and $\mathcal{X}$ the $\sigma$-algebra. As is standard, we assume when convenient (e.g., when we need to use disintegration to define relatively independent joinings or Host-Kra factors) that our probability space is \emph{regular} in that $X$ is a compact metric space and $\mathcal{X}$ is the set of all Borel sets on $X$. One can find the (rather subtle) details of this reduction in Furstenberg's book \cite[Chapter 5]{Fur81}. We set some conventions regarding measurable functions.
\begin{notation}
Let $\mathbf{X} = (X, \mathcal{X}, \mu)$ be a measure space. We denote $L^p(\mathbf{X})$ as the set of all bounded measurable functions $f$ on $\mathbf{X}$ with finite $L^p$ norm, up to almost-everywhere equivalence. If it is clear from context, we shall also abusively use $L^p(\mu)$ as $L^p(\mathbf{X})$.
\end{notation}
We denote $\mathrm{Aut}(\mathbf{X})$ as the set of all automorphisms of $\mathcal{X}$, that is, measurable maps $T: X \to X$ such that $\mu(T^{-1}A) = \mu(A)$ for all $A \in \mathcal{X}$. For $G$ a group, a measure-preserving $G$-system denotes a quadruple $(\mathbf{X}, T) = (X, \mathcal{X}, \mu, T)$ with $T: G \to \mathrm{Aut}(X, \mathcal{X}, \mu)$ a $G$-action. We next need the notion of factors and isomorphisms.
\begin{definition}[Factors and isomorphisms]
A \textbf{factor} $\mathbf{Y} = (Y, \mathcal{Y}, \nu)$ of a measure space $\mathbf{X} = (X, \mathcal{X}, \mu)$ (resp. $(\mathbf{Y}, S)$ of a measure-preserving $G$-system $(\mathbf{X}, T)$) is a map $\phi: X \to Y$ with $\phi^{-1}(\mathcal{Y}) \subseteq \mathcal{X}$ and $\phi_* \mu := \mu \circ \phi^{-1} = \nu$ (and resp. $S \circ \phi = \phi \circ T$). We say $\mathbf{X}$ (resp. $(\mathbf{X}, T)$) is an extension of $\mathbf{Y}$ (resp. $(\mathbf{Y}, S)$). We say the factor is an \textbf{isomorphism} if $\phi$ is invertible almost everywhere. 
\end{definition}
We next need the notion of conditional expectation.
\begin{definition}[Conditional expectation]
Let $\mathbf{X} = (X, \mathcal{X}, \mu)$ be a measure space. Given a sub-$\sigma$-algebra $\mathcal{Y}$ of $\mathcal{X}$ and $f \in L^2(\mathbf{X})$, we denote $\mathbb{E}_\mu(f|\mathcal{Y})$ (or simply $\mathbb{E}(f|\mathcal{Y})$ if the measure $\mu$ is clear) as the projection from $f$ to $L^2(\mathbf{Y})$ with $\mathbf{Y} = (X, \mathcal{Y}, \mu)$.
\end{definition}
We also established the following conventions regarding the conditional expectation of factors.
\begin{notation}
Let $\mathbf{X} = (X, \mathcal{X}, \mu)$ be a measure space and $f \in L^\infty(\mathbf{X})$. If $\mathbf{Y} = (Y, \mathcal{Y}, \nu)$ is a factor of $\mathbf{X}$ via $\phi$, we denote $\mathbb{E}_\mu(f|\mathcal{Y})$ or $\mathbb{E}_\mu(f|\mathbf{Y})$ as $\mathbb{E}_\mu(f|\phi^{-1}(\mathcal{Y}))$.
\end{notation}
We next need the notion of a joining of two measure spaces (resp. measure-preserving $G$-systems).
\begin{definition}[Joining of measure spaces]
Let $\mathbf{X} = (X, \mathcal{X}, \mu)$ and $\mathbf{X}' = (X', \mathcal{X}', \mu')$ (resp. $(\mathbf{X}, T)$ and $(\mathbf{X}, T')$) be two measure spaces (resp. measure-preserving $G$-systems). We define a \textbf{joining} of $\mathbf{X}$ and $\mathbf{X}'$ to be a measure $\nu$ on $(X \times X', \mathcal{X} \otimes \mathcal{X}')$ such that the projection of $\nu$ to $X$ is equal to $\mu$ and the projection of $\nu$ to $X'$ is equal to $\mu'$. (Resp. in the case of $(\mathbf{X}, T)$ and $(\mathbf{X}', T')$, we require $\nu$ to be $T\times T'$ invariant.) 
\end{definition}
We need the notion of a relatively independent joining.
\begin{definition}[Relatively independent joining]
Let $\mathbf{X} = (X, \mathcal{X}, \mu)$ and $\mathbf{X}' = (X', \mathcal{X}', \mu')$ be extensions of a measure space $\mathbf{Y} = (Y, \mathcal{Y}, \nu)$. We denote $\mathbf{X} \times_\mathbf{Y} \mathbf{X}' = (X \times X', \mathcal{X} \otimes \mathcal{X}', \mu \times_{\mathbf{Y}} \mu)$ to be the \textbf{relatively independent joining} of $\mathbf{X}$ and $\mathbf{X}'$ over $\mathbf{Y}$ where $\mu \times_{\mathbf{Y}} \mu'$ is defined to be the unique measure such that for $f \in L^\infty(\mathbf{X})$ and $f' \in L^\infty(\mathbf{X}')$,
$$\int f \otimes f' d(\mu \times_{\mathbf{Y}} \mu) = \int \mathbb{E}_\mu(f|\mathcal{Y})\cdot \mathbb{E}_{\mu'}(f'|\mathcal{Y}) d\nu.$$
\end{definition}
It is clear that if the relatively independent joining exists, it is unique. To prove existence, one can use disintegration (as in e.g., \cite[6.3 Examples]{Gla03}): (using the notation above) by taking disintegrations $\mu = \int_{Y} \mu_y d\nu$ and $\mu' = \int_{Y} \mu_y' d\nu(y)$, we define
$$\mu \times_{\mathbf{Y}} \mu' = \int_Y \mu_y \times \mu_y' d\nu(y).$$
\begin{remark}
Note that if $T$ is a probability preserving action on $\mathbf{X}$ and $T'$ is a probability preserving action on $\mathbf{X}'$ that both factor to an action of $S$, then $\mu \times_{\mathbf{Y}} \mu'$ is $T \times T'$ invariant. This is because product sets generate the product sigma algebra.
\end{remark}
We next consider Host-Kra seminorms.
\begin{definition}[Discrete multiplicative derivatives of measurable functions]
Given a function $f \in L^\infty(\mu)$ and a measurable transformation $T: (X, \mathcal{X}, \mu) \to (X, \mathcal{X}, \mu)$, we denote $\Delta_T f := \overline{Tf}\cdot f \in L^\infty(\mu)$ and $\Delta_{T_1, \dots, T_k} f = \Delta_{T_1} \Delta_{T_2} \cdots \Delta_{T_k} f$. 
\end{definition}
\begin{definition}[Host-Kra seminorm]\label{def:hostkra}
If $\vec{T_1}, \dots, \vec{T_k}$ denotes commuting $\mathbb{Z}^{m_1}, \dots, \mathbb{Z}^{m_k}$ measure preserving actions on $\mathbf{X} = (X,\mathcal{X}, \mu)$, and $f \in L^\infty(\mathbf{X})$, we denote
$$\|f\|_{\vec{T_1}, \dots, \vec{T_k}}^{2^k} = \lim_{N_1 \to \infty} \cdots \lim_{N_k \to \infty} \mathbb{E}_{n_i \in [N_i]^{m_i}} \int \Delta_{\vec{T_1}^{n_1}, \vec{T_2}^{n_2}, \dots, \vec{T_k}^{n_k}} f d\mu.$$  
\end{definition}
We also require the following result of Austin \cite{A10}.
\begin{theorem}[Austin's ergodic theorem]\label{thm:austinergodictheorem}
Let $T_1, \dots, T_k$ be $\mathbb{Z}^m$ measure preserving tranformations on $\mathbf{X} = (X, \mathcal{X}, \mu)$ and let $I_N$ be a F\o lner sequence in $\mathbb{Z}^m$ and $a_N \in \mathbb{Z}^m$. Then for any $f_1, \dots, f_m \in L^\infty(\mathbf{X})$,
$$\lim_{N \to \infty} \frac{1}{|I_N|} \sum_{n \in I_N + a_N} T_1^nf_1 T_2^nf_2 \cdots T_m^nf_m$$
converges in $L^2(\mathbf{X})$.
\end{theorem}
\begin{remark}
Note that because the average converges for \emph{all} F\o lner sequences $I_N$ and base points $a_N$, the value the average converges to does not depend on the choice of the F\o lner sequence or the base points. A simple consequence of this is that we in fact have
$$\|f\|_{\vec{T_1}, \dots, \vec{T_k}}^{2^k} = \lim_{N \to \infty} \mathbb{E}_{n_i \in [N]} \int \Delta_{\vec{T_1}^{n_1}, \vec{T_2}^{n_2}, \dots, \vec{T_k}^{n_k}} f d\mu.$$
As the argument is implicitly used throughout Sections 9 and 10, it is worth elaborating on. This can be seen since one can extract the iterated limit in Definition~\ref{def:hostkra} a F\o lner sequence for $\mathbb{N}^k$ for which the ergodic average converges. On the other hand, the above average converges along the F\o lner sequence $\Phi_N = [N]^k$ for $\mathbb{N}^k$.
\end{remark}
Finally, we define the Host-Kra factor.
\begin{definition}\label{def:HostKra}[Host-Kra factor]
Let $\vec{T_1}, \dots, \vec{T_k}$ be commuting $\mathbb{Z}^{m_1}, \dots, \mathbb{Z}^{m_k}$ measure preserving transformations on $\mathbf{X} = (X, \mathcal{X}, \mu)$. We denote $\mathbf{Z}_{\vec{T}_1, \dots, \vec{T}_k} = (X, \mathcal{Z}_{\vec{T}_1, \dots, \vec{T}_k}, \mu)$ to be the factor of $\mathbf{X}$ such that $\|f\|_{\vec{T_1}, \dots, \vec{T_k}} = 0 \iff \mathbb{E}(f|\mathbf{Z}_{\vec{T}_1, \dots, \vec{T}_k}) = 0$.
\end{definition}
The existence of the Host-Kra factor (as well as various properties of them) is proven in \cite{Ho09}. As we gave a brief recap of these properties in our prequel paper \cite{Len24}, we will omit much of these properties in this document.

\section{Another finitary inverse theorem}
In this section, we state the following theorem which will be used to deduce Theorem~\ref{thm:maininversetheorem2}.
\begin{theorem}\label{thm:maininversetheorem}
Let $\delta \in (0, 1/10)$, $k, \ell, K$ positive integers with $\ell \le k \le K$, and $f: [N]^k \to S^1$ be such that 
$$\|f\|_{U([N]^k, e_1[N], \dots, e_\ell[N])} \ge \delta.$$
Then there exist one-bounded functions $f_1, \dots, f_\ell$ with $f_i$ not depending on coordinate $i$, and a nilsequence $\chi \in \mathrm{Nil}^\ell(M, m, k, 1)$ that only depend on the coordinates $1, \dots, \ell$ such that
$$|\mathbb{E}_{n \in [N]^k} f(n)\chi(n)f_1(n)\cdots f_k(n)| \ge \epsilon$$
where
$$\epsilon^{-1}, M \le \exp(\log(1/\delta)^{O(1)}), m \le \log(1/\delta)^{O(1)}.$$
\end{theorem}
Throughout the proof of this theorem, we work with the following conventions.
\begin{notation}\label{not:coordinization}
In context of the proof of Theorem~\ref{thm:maininversetheorem} (where parameters $k$ and $\ell$ are fixed), we denote $x \in [N]^k$ as $(x_\ell, x_*)$ with $x_\ell$ denoting the $\ell$th coordinate and $x_*$ denoting the rest of the coordinates. We also denote for $\chi \in \mathrm{Nil}^\ell(M, m, L, D)$ to be a nilsequence that depends only on the first $L$ coordinates.
\end{notation}

\section{Initial maneuvers}
In this section, we make initial maneuvers regarding the proofs of Theorem~\ref{thm:maininversetheorem}. We begin with the following.
\begin{prop}\label{prop:initialmaneuvers}
Assume Theorem~\ref{thm:maininversetheorem} holds for $(k, \ell - 1)$. Let $\delta \in (0, 1/10)$, $k, \ell, K$ be positive integers with $\ell \le k \le K$, and $f: [N]^k \to S^1$ be such that 
$$\|f\|_{U([N]^k, e_1[N], \dots, e_\ell[N])} \ge \delta.$$
Then there exist one-bounded functions $f_1, \dots, f_\ell$ with $f_i$ not depending on coordinate $i$, and a family of nilcharacters $F(g_{n_\ell}(n)\Gamma) \in \mathrm{Nil}^{\ell - 1}(M(\delta), m(\delta), k, M(\delta))$ which depend only on the coordinates $1, \dots, \ell - 1$ such that
$$\|\mathbb{E}_{n \in [N]^k} f(n)f_1(n)\cdots f_\ell(n) \otimes F(g_{n_\ell}(n)\Gamma)\|_\infty \ge \epsilon(\delta).$$
\end{prop}
\begin{proof}
Let $N'$ be a prime lying between $100N$ and $200N$. By extending $f$ to be zero outside $[N]$, we may suppose that
$$\|f\|_{U(\mathbb{Z}_{N'}^k, e_1\mathbb{Z}_{N'}, \dots, e_\ell\mathbb{Z}_{N'})} \ge \delta.$$
From hereon, we work with $N$ in place of $N'$. As $N \asymp N'$, the conclusion for $N'$ implies the conclusion for $N$. Then
$$\mathbb{E}_{h \in \mathbb{Z}_{N}} \|\Delta_{he_\ell}f\|_{U(\mathbb{Z}_{N}^k, e_1\mathbb{Z}_{N}, \dots, e_{\ell - 1}\mathbb{Z}_{N})}^{2^{k - 1}} \ge \delta^{2^k}.$$
Applying the inductive argument, for $\delta^{O(1)}N$ many $h \in \mathbb{Z}_{N}$, we may find nilsequences $\chi_h = F_h(g_h(n)\Gamma_h) \in \mathrm{Nil}^{\ell - 1}(m(\delta), M(\delta), k, M(\delta))$ and one-bounded $f_{1, h}, \dots, f_{j - 1, h}$ such that
$$\mathbb{E}_{h \in \mathbb{Z}_{N}} |\mathbb{E}_{x \in \mathbb{Z}_{N}^k} \Delta_{he_j} f(x) f_{1, h}(x)f_{2, h}(x) \cdots f_{j - 1, h}(x) \cdot \chi_h(x)| \ge \epsilon(\delta)$$
with $f_{i, h}$ not depending on coordinate $i$ and $\chi_h$ depends only on the coordinates $e_1, \dots, e_{j - 1}$. Here, we denote the underlying nilmanifold of $\chi_h$ as $G_h/\Gamma_h$. We now use a standard argument to remove the $h$-dependence of $F_h$, $G_h$, and $\Gamma_h$. The removal of $h$-dependence of $G_h$ and $\Gamma_h$ follows from a pigeonholing argument on the number of nilmanifolds. Indeed, there are at most $M(\delta)$ nilmanifolds of complexity $M(\delta)$. This removes the $h$-dependence of $G_h$ and $\Gamma_h$ at the cost of shrinking our proportion of $h$'s from $\delta^{O(1)}$ to $\epsilon(\delta)$. To reduce the $h$-dependence of $F_h$, we use a partition of unity argument. By \cite[Lemma B.3]{LSS24b}, we may write
$$F_h(g\Gamma) = \sum_{j \in I} \tau_j(g\Gamma)^2 \cdot F_h(g\Gamma)$$
for $|I| \le (1/\epsilon)^{m(\delta)}$ with $\epsilon > 0$ to be chosen later and $\tau_j$ obeying the properties specified in the lemma. Choosing $\epsilon$ to be sufficiently small (say $\epsilon(\delta)$), we may specify that
$$\sup_{g \in G} \left|F_h(g\Gamma) - \sum_{i \in I} a_{i, h}\tau_i(g\Gamma)^2\right| \le \delta^{2^k}/4$$
for some appropriate values of $a_i$; the point is that we may, via the triangle inequality, replace $F_h$ with $\sum_{i \in I} a_{i, h}\tau_i(g\Gamma)^2$ and one may take $a_{i, h}$ to vary over $M(\delta)$ number of elements as $h$ varies. By pigeonholing over one of these choices of $a_{i, h}$, we have demonstrated independence of $F_h$ from $h$.

Now, by Fourier expanding as in Lemma~\ref{lem:nilcharacters}, we may assume that $\chi_h$ is a $G_k$-vertical character with frequency bounded by $M(\delta)$. Making a change of variables $h \mapsto h - x_\ell$, $x_\ell \to h$, and $h \to x_\ell$, we find that there exists a one-bounded function $c$ such that
$$\mathbb{E}_{h \in \mathbb{Z}_{N}} \mathbb{E}_{x \in \mathbb{Z}_{N}^k} c(x_\ell)f(x)f(x_*, h)\prod_{i = 1}^{\ell - 1}f_{i, x_\ell - h}(x + he_\ell - x_\ell e_\ell)  F(g_{x_\ell - h}(x_*)\Gamma) \ge \epsilon(\delta).$$
Pigeonholing in $h$, we find that there exists $h_0$ such that $f_{i}(x) = f_{i, x_\ell - h_0}(x + he_\ell - x_\ell e_\ell)$ for $i \neq 1$, $f_1(x) = f_{1, x_\ell - h_0}(x + he_\ell - x_\ell e_\ell)c(x_\ell)$, and $f_\ell(x) = f(x_*, h_0)$, we find that
$$|\mathbb{E}_{x \in \mathbb{Z}_{N}^k} f(x)F(g_{x_\ell}(x_*)\Gamma)\prod_{i = 1}^\ell f_i(x)| \ge \epsilon(\delta).$$
We finally wish to replace $F(g_{x_\ell}(x_*)\Gamma)$ with a component of a nilcharacter. By \cite[Lemma B.4]{LSS24b}, there exists a nilcharacter $\chi$ which has the same $G_k$-vertical frequency as $F$. Finally, by observing that $(F/(2\|F\|_\infty), \chi \cdot \sqrt{1 - |F/(2\|F\|_\infty)|^2})$ is a nilcharacter, we have the desired result.
\end{proof}

\section{Deduction of Theorem~\ref{thm:maininversetheorem}}
We shall now prove Theorem~\ref{thm:maininversetheorem} assuming the theorem holds for $(k, \ell - 1)$. We shall assume the following structure theorem.
\begin{theorem}\label{thm:nilsequenceconstruction}
Let $\delta \in (0, 1/10)$, $k, D, d, K \in \mathbb{N}$ with $k, d, D \le K$, $H \subseteq [N]^{D}$ of size at least $\delta N^{D}$ and for each $h \in H$ let $\chi_h(x) = F(g_h(x)\Gamma) \in \mathrm{Nil}^{d - 1}(M(\delta), m(\delta), k, M(\delta))$ with $G_{d}$-vertical frequency $\eta$ and $\psi_{\vec{h}} = \psi_{\vec{h}}(g_{\vec{h}}(n)\Gamma') \in \mathrm{Nil}^{d - 2}(M(\delta), m(\delta), k, M(\delta))$. Suppose for $\delta |H|^3$ many additive quadruples $(h_1, h_2, h_3, h_4) \in H^4$ we have
$$\|\mathbb{E}_{x \in [N]^{k}} \chi_{h_1}(x) \otimes \chi_{h_2}(x) \otimes \overline{\chi_{h_3}(x)\otimes \chi_{h_4}(x)} \cdot \psi_{\vec{h}}(x)\|_\infty \ge \epsilon(\delta).$$
Then there exists a set $H' \subseteq H$ of size at least $\epsilon(\delta)|H|$ and a multidegree nilcharacter 
$$\widetilde{\chi} \in \mathrm{Nil}^{(1, d - 1)}(M(\delta), m(\delta), D + k, M(\delta))$$
such that for each $h \in H$', $\chi_h(\cdot)$ is $(m(\delta), M(\delta))$-equivalent to $\chi(h, \cdot)$.
\end{theorem}
The reason we require this structure theorem will become clear in the proof of the inverse theorem. The structure theorem will be deduced by the following proposition via a degree-rank induction.
\begin{prop}\label{prop:degreerankextraction}
Let $\delta \in (0, 1/10)$, $k, D, r, d, K \in \mathbb{N}$ with $k, k', d, r, D \le K$, $H \subseteq [N]^D$ of size at least $\delta N^D$ and for each $h \in H$, let $\chi_h(x) = F(g_h(x)\Gamma)$ be a family of degree-rank $(d - 1, r)$ nilcharacters in $\mathrm{Nil}^{(d - 1, r)}(M(\delta), m(\delta), k, M(\delta))$ with $G_{(d, r)}$-vertical frequency $\eta$ and $\psi_{\vec{h}} = \psi_{\vec{h}}(g_{\vec{h}}(n)\Gamma') \in \mathrm{Nil}^{d - 2}(M(\delta), m(\delta), k, M(\delta))$. Suppose for $\epsilon(\delta) N^{3D}$ many additive quadruples $(h_1, h_2, h_3, h_4) \in H^4$ we have
$$\|\mathbb{E}_{x \in [N]^{D}} \chi_{h_1}(x) \otimes \chi_{h_2}(x) \otimes \overline{\chi_{h_3}(x)\otimes \chi_{h_4}(x)} \cdot \psi_{\vec{h}}(x)\|_\infty \ge \epsilon(\delta).$$
Then there exists a set $H' \subseteq H$ of size at least $\epsilon(\delta)|H|$ and a multi-degree nilcharacter $\widetilde{\chi}$ in $\mathrm{Nil}^{(1, d - 1)}(M(\delta), m(\delta), k, M(\delta))$ such that for each $h \in H'$, $\chi_h(\cdot)$ is $(m(\delta), M(\delta))$-equivalent to $\chi(h, \cdot)$ up to a degree-rank $< (d - 1, r)$ nilsequence.
\end{prop}
Additionally, we need the following regularity lemmas from \cite{Tao07}. In the two lemmas below, $H$ will denote a Hilbert space with norm $\|\cdot \|_H$ and $S \subset H$ will denote a class of \emph{structured} functions. We say that a function $f$ is $(M, K)$-structured if it can be written as
$$f = \sum_{i = 1}^M c_i v_i$$
where $|c_i| \le K$ and $v_i \in S$. We say that $f$ is $\epsilon$-pseudorandom if $|\langle f, v \rangle| \le \epsilon$ for all $v \in S$. The first lemma is the weak regularity lemma, following \cite[Corollary 2.5]{Tao07}.
\begin{lemma}[Weak regularity lemma]\label{lem:weakregularity}
Let $(H, \langle \cdot, \cdot \rangle_H)$ be a Hilbert space and $f \in H$ with $\|f\|_H \le 1$ and $0 \le \epsilon \le 1$. Then we may write $f = f_{\mathrm{str}} + f_{\mathrm{psd}}$ with $f_{\mathrm{str}}$ $(1/\epsilon^2, 1/\epsilon)$-structured, and $f_{\mathrm{psd}}$ is $\epsilon$-pseudorandom and having norm at most $\|f\|_H$.
\end{lemma}
We will also need the strong regularity lemma later in the argument in Section 9. We state it now for convenience.
\begin{lemma}[Strong regularity lemma]\label{lem:strongregularity}
Let $f \in H$ with $\|f\|_H \le 1$, $\epsilon > 0$, and $\mathcal{F}: \mathbb{Z}_{> 0} \to \mathbb{R}_{>0}$ be any function. Then we may find some $M = O_{\mathcal{F}, \epsilon}(1)$ and a decomposition
$$f = f_{\mathrm{str}} + f_{\mathrm{psd}} + f_{\mathrm{sml}}$$
with $f_{\mathrm{str}}$ being $(M, M)$-structured, $f_{\mathrm{sml}}$ having norm at most $\epsilon$, and $f_{\mathrm{psd}}$ being $1/\mathcal{F}(M)$-pseudorandom and having norm at most $\|f\|_H$.
\end{lemma}
We also require the following lemma. 
\begin{lemma}[Corollary of the splitting Lemma]\label{lem:splitting}
Let $\delta \in (0, 1/10)$, $\ell \le k \in \mathbb{N}$, $f: [N]^k \to \mathbb{C}$, and $\chi \in \mathrm{Nil}^\ell(M, m, \ell, 1)$ be such that
$$|\mathbb{E}_{x \in [N]^k} f(x)\chi(x)| \ge \delta.$$
Then there exists $\chi' \in \mathrm{Nil}^{(1, 1, \dots, 1)}(M, m, \ell, 1)$ and one-bounded functions $f_1, \dots, f_\ell:[N] \to \mathbb{C}$ which depend only on the first $\ell$ coordinates with $f_i$ not depending on coordinate $i$ such that
$$|\mathbb{E}_{x \in [N]^k} f(x)\chi'(x)f_1(x) \cdots f_\ell(x)| \ge (\delta/M)^{m^{O_k(1)}}.$$
\end{lemma}
\begin{proof}
This follows by the fact that a $\mathbb{Z}^\ell$ input nilsequence of degree $d$ is multidegree $J = \{(j_1, \dots, j_\ell): j_1 + \cdots + j_\ell \le d\}$ (via a filtration $G_{(j_1, \dots, j_\ell)} = G_{j_1 + \cdots + j_\ell}$) and then iteratively applying \cite[Lemma C.6]{LSS24b} and noting that multidegree $(d_1, \dots, d_\ell)$ with $d_j$ zero does not depend on coordinate $j$.
\end{proof}
We are now ready to prove Theorem~\ref{thm:maininversetheorem} assuming Theorem~\ref{thm:nilsequenceconstruction}. As the proof is iterative in nature, we have decided to keep track of some of the nature of quantifier dependence. In particular, ``$\epsilon(\delta)^C$" for an absolute constant $C$ will not be suppressed as ``$\epsilon(\delta)$" as we have done in the rest of the paper.

\begin{proof}[Proof of Theorem~\ref{thm:maininversetheorem} assuming Theorem~\ref{thm:nilsequenceconstruction}.]
Let $H$ be the Hilbert space $L^2((\mathbb{Z}/N\mathbb{Z})^k)$ and $S$ denote the set of all functions that are a component of $F(g_{x_j}(x_*)\Gamma) \otimes f_1 \cdots f_j$ with $f_1, \dots, f_j$ one-bounded and $f_i$ not depending on the $i$th coordinate with $F(g_{x_j}(x_*)\Gamma) \in \mathrm{Nil}^{(1, 1, \dots, 1)}(M(\delta), m(\delta), \ell - 1, M(\delta))$. Invoking Lemma~\ref{lem:weakregularity}, Proposition~\ref{prop:initialmaneuvers}, and Lemma~\ref{lem:splitting}, we may write
$$f = f_{\mathrm{str}} + f_{\mathrm{psd}}$$
where $f_{\mathrm{str}}$ is $(1/\epsilon(\delta)^C, 1/\epsilon(\delta)^C)$-structured for some $C \ge 10$ and can be written as
$$\sum_{i = 1}^M v_i$$
with $v_i \in S$, , and $\|f_{\mathrm{psd}}\|_{[N]^k, e_1, \dots, e_\ell} \le \epsilon(\delta)^{10}$. Note that if $\|f - f_{\mathrm{psd}}\|_{L^2} \le \epsilon(\delta)^{2^{\ell + 1}C}$, then by the Cauchy-Schwarz inequality, we have $\|f - f_{\mathrm{psd}}\|_{[N]^k, e_1, \dots, e_\ell} \le c(\delta)^C$ and hence $\|f_{\mathrm{psd}}\|_{[N]^k, e_1, \dots, e_\ell} \ge \epsilon(\delta)^{2}$, a contradiction. It follows that
$$|\langle f, f_{\mathrm{str}} \rangle| \ge \epsilon(\delta)^{2^{\ell + 2}C}.$$

Let $v_i = f_{i, 1}(\vec{x})f_{i, 2}(\vec{x}) \cdots f_{i, \ell}(\vec{x}) F(g_{x_\ell}(\vec{x})\Gamma)$ where:
\begin{itemize}
    \item $f_{i, i'}$ does not depend on the $i'$th coordinate.
    \item A family of nilsequences $F(g_{x_\ell}(x)\Gamma) \in \mathrm{Nil}^{(1, 1, \dots, 1)}(M(\delta), m(\delta), \ell - 1, 1)$ whose components depend on $x_1, \dots, x_{\ell - 1}$ and;
    \item A family of nilcharacters $\chi_{x_\ell}(x_*)  \in \mathrm{Nil}^{(1, 1, \dots, 1)}(M(\delta), m(\delta), \ell - 1, M(\delta))$ for which $F(g_{x_\ell}(x)\Gamma)$ is a component of one of these nilcharacters.
\end{itemize}
We shall iteratively refine $v_i = v_i^1$, until $v_i^P = f_{i, 1}^P(\vec{x})f_{i, 2}^P(\vec{x}) \cdots f_{i, j}^P(\vec{x}) F(g_{x_\ell}^P(\vec{x})\Gamma)$ with $P = \exp(\log(1/\delta)^{O_{k, C}(1)})$, $F(g_{x_\ell}(\vec{x}))$ equal to a degree $\ell$ nilsequence $\widetilde{F}(g(x_1, \dots, x_\ell)\Gamma)$, and $|\langle v_i^L, f \rangle| \ge \epsilon(\delta)$. Given this, Theorem~\ref{thm:maininversetheorem} follows. At each stage, we will have
$$f = f_{\mathrm{str},j} + f_{\mathrm{psd}} + f_{\mathrm{strong}, j}$$
with 
$$f_{\mathrm{strong}, j} = \sum_{i = 1}^{M'} v_{i, \mathrm{strong}}^j$$
with $M' \le (\epsilon(\delta)/M(\delta))^{-m(\delta)^{O_C(1)}}$, and $v_{i, \mathrm{strong}}^j$ is strongly structured in the sense that $F(g_{x_\ell}(x_*)\Gamma)$ (as a function of $(x_*, x_\ell)$) is $(M(\delta), m(\delta))$-equivalent to a nilcharacter in $\mathrm{Nil}^{(1, 1, \dots, 1)}(M(\delta), m(\delta), \ell, M(\delta))$ for $x_\ell$ lying in a set of size at least $\epsilon(\delta) N$. $f_{\mathrm{str}, j}$ consists of structured parts as above. We will further assume that the $x_\ell$ coordinates of $f_{\mathrm{str}, j}$ are supported on a set $S_i^j$.

Now assume for the sake of argument that $\sum_i |\langle f, v_{i, \mathrm{strong}}^j \rangle| \le \epsilon(\delta)^2$. If this does not occur, then $f$ correlates with a strongly structured function. Given this, we see by the pigeonhole principle (and given that $C$ is chosen to be sufficiently large) that there exists some $v_i^j$ such that
$$\|v_i^j\|_{e_1[N], e_2[N], \dots, e_\ell[N], [N]^k} \ge (\epsilon(\delta)/M(\delta))^{m(\delta)^{O_C(1)}}.$$
Hence, applying Lemma~\ref{lem:splitting} and absorbing lower order terms to $f_i$, and applying the Cauchy-Schwarz-Gowers inequality while absorbing $f_i$ into other terms in the Gowers norm inner product, we have
$$\|\mathbb{E}_{\vec{x}, \vec{h}, \vec{h'} \in [N]^k} \Delta^{\otimes}_{h_1e_1, \dots, h_\ell e_\ell, (h_1', h_2', \dots, h_\ell')}\chi^j_{x_\ell}(x_*)\Delta_{h_\ell, h_\ell'}1_{S_i^j}(x_\ell)\|_{\infty} \ge (\epsilon(\delta)/M(\delta))^{m(\delta)^{O_C(1)}}.$$
By \cite[Lemma C.4]{LSS24b}, we see that $\Delta^\otimes_{h_ie_i}\chi^j_{x_\ell}(x_*)$ is $((\epsilon(\delta)/M(\delta))^{-m(\delta)^{O_C(1)}}, m(\delta)^{O_C(1)})$-equivalent to
$$\chi^j_{x_\ell}(x_1, \dots, x_{i - 1}, h_i, x_{i + 1}, \dots, x_{\ell - 1}).$$
Hence, for some $\psi_{x_\ell, h_\ell, h_\ell'}$ in $\mathrm{Nil}^{< \ell - 1}((\epsilon(\delta)/M(\delta))^{-m(\delta)^{O_C(1)}}, m(\delta)^{O_C(1)}, \ell - 1, (\epsilon(\delta)/M(\delta))^{-m(\delta)^{O_C(1)}})$, we see that
$$\|\mathbb{E}_{\vec{x}, \vec{h}, \vec{h'} \in [N]^k} \Delta_{h_\ell e_\ell, h_\ell' e_\ell}^{\otimes}\chi^j_{x_\ell}(h'_*)\Delta_{h_\ell, h_\ell'}1_{S_i^j}(x_\ell) \cdot \psi_{x_\ell, h_\ell, h_\ell'}(x_*, \vec{h}, \vec{h'})\|_{\infty} \ge (\epsilon(\delta)/M(\delta))^{m(\delta)^{O_C(1)}}.$$
Here, we define $\Delta^{\otimes}$ as the same as $\Delta$ except with multiplication replaced by tensor product. Applying Theorem~\ref{thm:nilsequenceconstruction}, it follows that there exists a degree $\ell$ nilsequence $\widetilde{\chi}^j$ and a subset $Y \subseteq S_i^j$ of size $(\epsilon(\delta)/M(\delta))^{m(\delta)^{O_C(1)}}$ such that for any $h \in Y$, we see that $\chi^j(h, x_*)$ is equivalent to $\widetilde{\chi}^j(h, x_*)$ up to a degree $\ell - 2$ ($h$-dependent) nilsequence. We now write $1_{S_i^j} = 1_{S_i^j \setminus Y} + 1_Y$, so $v_i^j(x) = v_i^j(x) 1_{S_i^j \setminus Y}(x_\ell) + v_i^j(x) 1_Y(x_\ell)$. If $|\langle f, v_i^j 1_Y \rangle| \ge (\epsilon(\delta)/M(\delta))^{m(\delta)^{O_C(1)}}$, then we terminate the procedure, as $v_i^j 1_Y$ is equivalent to a degree $\ell$ nilsequence; this is enough to conclude the main theorem. Otherwise, we take $f_{j + 1, \mathrm{strong}} = f_{j, \mathrm{strong}} + v_i^j 1_Y$ and $f_{j + 1, \mathrm{str}} = f_{j, \mathrm{str}} - v_i^j 1_Y$. We then set $v_i^{j + 1} = 1_Y v_i^j$ and for all other $i'$, we set $v_{i'}^{j + 1} = v_{i'}^j$ and $v_{i', \mathrm{strong}}^{j + 1} = v_{i', \mathrm{strong}}^j$, and then finally set $S_i^{j + 1} = S_i^j \setminus Y$. This procedure terminates after at most $(\epsilon(\delta)/M(\delta))^{-m(\delta)^{O_C(1)}}$ many iterations since $|Y|$ has a lower bound \textbf{independent} of the step of the iteration, after which we have a strongly structured function which correlates with $f$. This implies that for some $i$,
$$|\langle f, v_{i, \mathrm{strong}}^P \rangle| \ge (\epsilon(\delta)/M(\delta))^{m(\delta)^{O_C(1)}}$$
which concludes the proof of Theorem~\ref{thm:maininversetheorem} by Lemma~\ref{lem:splitting} since $C$ is constant.
\end{proof}
We will now deduce Theorem~\ref{thm:nilsequenceconstruction} from Proposition~\ref{prop:degreerankextraction}.
\begin{proof}[Proof of Theorem~\ref{thm:nilsequenceconstruction} assuming Proposition~\ref{prop:degreerankextraction}.]
The proof is little more than observing that for additive quadruples $(h_1, h_2, h_3, h_4)$ with $h_1 + h_2 = h_3 + h_4$, $\chi(h_1, x) \otimes \chi(h_2, x) \otimes \overline{\chi(h_3, x) \otimes \chi(h_4, x)}$ lies inside $\mathrm{Nil}^{d - 2}(M(\delta), d(\delta), \ell, M(\delta))$. The theorem then follows from Lemma~\ref{lem:equiv}.
\end{proof}
\section{Initial Maneuvers for Theorem~\ref{thm:maininversetheorem2}}
Assume that $|f| \le 1$ and
$$\|f\|_{U([N]^k, \dots, [N]^k, e_1[N], \dots, e_\ell [N])} \ge \delta$$
for $\ell' + 1$ copies of $[N]^k$. We thus have (for $\ell'$ copies of $[N]^k$)
$$\mathbb{E}_{h \in [N]^k} \|\Delta_h f\|_{U([N]^k, \dots, [N]^k, e_1 [N], \dots, e_\ell [N])}^{2^{\ell + \ell' - 1}} \ge \delta^{2^{\ell + \ell'}}$$
Then, by induction on $\ell'$ there exists a family of nilcharacters $\chi_h \in \mathrm{Nil}^{\ell + \ell' - 1}(M(\delta), d(\delta), \ell, M(\delta))$ and functions $f_{1, h}, \dots, f_{\ell, h}$ such that $f_{i, h}$ does not depend on coordinate $i$, and
$$\mathbb{E}_{h \in [N]^k} \bigg\|\mathbb{E}_{n \in [N]^k} \Delta_h f(n) \chi_h(n)\prod_{i = 1}^{\ell} f_{i, h}(n)\bigg\|_\infty^2 \ge \delta.$$
We will next use a Cauchy-Schwarz argument of Gowers. This is the analogue of \cite[Lemma 7.2]{LSS24b} and its proof is essentially identical so we omit it.
\begin{lemma}\label{lem:CS-basic}
Suppose $\delta\in(0,1/2)$, $f_1,f_2\colon[N]^k\to\mb{C}$ are $1$-bounded, and $\chi_h\colon\mb{Z}\to\mb{C}$ are all $1$-bounded. Suppose that 
\[\mb{E}_{h\in[N]^k}|\mb{E}_{n\in[N]^k}f_2(n)\Delta_hf_1(n)\ol{\chi_h(n)}|\ge\delta.\]
Then there exists $\Theta$ such that 
$\snorm{\Theta}_{\mb{R}/\mb{Z}}\le\delta^{-O_k(1)}/N$ and 
\[\mb{E}_{\substack{h_1 + h_2 = h_3 + h_4\\ h_i\in[N]^k}}\bigg|\mb{E}_{n\in[N]^k}\chi_{h_1}(n)\chi_{h_2}(n + h_1 - h_4)\ol{\chi_{h_3}(n)} \ol{\chi_{h_4}(n+ h_1 - h_4)} \cdot e\big(\Theta \cdot n \big)\bigg|\ge\delta^{O_k(1)}.\]
\end{lemma}

We next require the following lemma which is the analogue of \cite[Lemma 7.5]{LSS24b}.
\begin{lemma}\label{lem:four-fold-biased}
Fix $\delta > 0$, $d, k, k' \le K$ be positive integers, and $H \subset [N]^{k'}$. By letting the last $k - k'$ coordinates zero, we see that $H$ embeds in $[N]^k$. Let $f\colon[N]^k\to\mb{C}$ be a $1$-bounded function. Suppose that there exist nilcharacters $\chi_h \in \mathrm{Nil}^{d - 1}(M(\delta), m(\delta), k, M(\delta))$ and $\chi(h, n) \in \mathrm{Nil}^d(M(\delta), m(\delta), k + k', M(\delta))$, and a nilsequence $\psi_h \in \mathrm{Nil}^{d - 2}(M(\delta), m(\delta), k, M(\delta))$ and functions $f_{i, h}$ not depending on coordinate $i$ such that
$$\left\|\mathbb{E}_{n \in [N]^k} \Delta_h f(n) \otimes \chi(h, n) \otimes \chi_h(n) \otimes \prod_{i = 1}^\ell f_{i, h}(n) \cdot \psi_h(n)\right\|_\infty \ge \epsilon(\delta).$$
Furthermore, suppose $\chi_h$ and $\chi$ satisfy the additional property that they are a product of $\mathbb{Z}$-multidegree nilcharacters that are not equivalent to the zero nilcharacter and are not degree zero on coordinates $e_1, \dots, e_\ell$. Then for at least $(\epsilon(\delta)/M(\delta))^{m(\delta)^{O_k(1)}} N^{3k'}$ quadruples $h_1,h_2,h_3,h_4\in H$ with $h_1 + h_2 = h_3 + h_4$, there exists $\psi_{h_1, h_2, h_3, h_4} \in \mathrm{Nil}^{d - 2}(M(\delta), m(\delta), k, M(\delta))$ such that
$$\|\mathbb{E}_{n \in [N]^k} \chi_{h_1}(n) \otimes \chi_{h_2}(n + h_1-h_4)\otimes \ol{\chi_{h_3}(n)}\otimes \ol{\chi_{h_4}(n + h_1-h_4)} \cdot \psi_{h_1, \dots, h_4}\|_\infty \ge \epsilon(\delta).$$
\end{lemma}

\begin{proof}
Since we are assuming that a multidimensional input nilcharacter is a product of multidegree $\mathbb{Z}$-nilcharacters (rather than multidegree $\mathbb{Z}^k$-nilcharacters), by \cite[Lemma C.5]{LSS24b}, these multidimensional nilcharacters are $(M(\delta), m(\delta))$-equivalent to multilinear nilcharacters $\chi'_h$ and $\chi'$ which are symmetric. (There are potentially other ways to proceed without needing this rather strange-looking hypothesis, such as proving a multidimensional input multilinearization result as in \cite[Lemma C.5]{LSS24b} but we found this approach to be most convenient.)

We first show that for some collection $\psi_{h_1, h_2, h_3, h_4} \in \mathrm{Nil}^{d - 2}(M(\delta), m(\delta), k, M(\delta))$,
$$\|\mathbb{E}_{n \in [N]^k} \widetilde{\chi}_{h_1}(n) \otimes \widetilde{\chi}_{h_2}(n + h_1-h_4)\otimes \ol{\widetilde{\chi}_{h_3}(n)}\otimes \ol{\widetilde{\chi}_{h_4}(n + h_1-h_4)} \cdot \psi_{h_1, \dots, h_4}\|_\infty \ge \epsilon(\delta)$$
with $\widetilde{\chi}_h = \chi_h \prod_{i = 1}^\ell f_{i, h}$. 

To show this, we follow the procedure of \cite[Lemma 7.5]{LSS24b}. By \cite[Lemma C.5]{LSS24b}, we see that $\chi(h, n)$ is $(M(\delta), m(\delta))$-equivalent to some $\chi'(h, n, \dots, n) \in \mathrm{Nil}^{(1, 1, \dots, 1)}(M(\delta), m(\delta), k + k', M(\delta))$ with frequency also bounded by $M(\delta)$. Applying Lemma~\ref{lem:equiv}, we have that
\begin{align*}
\snorm{\mb{E}_{n\in[N]^k, h \in [N]^{k'}}f(n) \ol{f(n+h)} \otimes \ol{\chi'(h,n,\ldots,n)} \otimes \ol{\widetilde{\chi}_h(n)} \cdot\ol{\psi_h(n)} \cdot \wt{\psi}(h,n)}_{\infty} \ge (MD/\rho)^{-O_s(d^{O_s(1)})},
\end{align*}
where $\wt{\psi}(h,n)$ is a nilsequence in $\mathrm{Nil}^{d - 1}(M(\delta), m(\delta), k + k', M(\delta))$. Applying the splitting lemma (\cite[Lemma C.6]{LSS24b}), we see that
\begin{align*}
\snorm{\mb{E}_{n\in[N]^k, h \in [N]^{k'}}f(n) \ol{f(n+h)} \otimes \ol{\chi'(h,n,\ldots,n)} \otimes \ol{\widetilde{\chi}_h(n)} \cdot\ol{\wt{\psi}_h(n)} \cdot b(n)}_{\infty} \ge \epsilon(\delta)
\end{align*}
where $\wt{\psi}_h \in \mathrm{Nil}^{d - 2}(M(\delta), m(\delta), k, M(\delta))$. Applying a variant of \cite[Lemma 7.2]{LSS24b}, we have
\begin{align*}
&\mb{E}_{\substack{h_1 + h_2 = h_3 + h_4\\ h_i\in[N]^{k'}}}\snorm{\mb{E}_{n\in[N]^k}\chi'(h_1,n,\ldots,n) \otimes \chi'(h_2,n + h_1 -h_4,\ldots,n + h_1 -h_4) \otimes \ol{\chi'(h_3,n,\ldots,n)} \\
&\otimes \ol{\chi'(h_4,n + h_1 -h_4,\ldots,n + h_1 -h_4)} \otimes \widetilde{\chi}_{h_1}(n) \otimes \widetilde{\chi}_{h_2}(n+h_1-h_4) \otimes \ol{\widetilde{\chi}_{h_3}(n)} \otimes \ol{\widetilde{\chi}_{h_4}(n+h_1-h_4)}\\
&\cdot \wt{\psi}_{h_1}(n)\wt{\psi}_{h_2}(n)\ol{\wt{\psi}_{h_3}(n)}\ol{\wt{\psi}_{h_4}(n)}e(\Theta n)}_{\infty} \ge \epsilon(\delta).
\end{align*}
We may combine $\wt{\psi}_{h_1}(n)\wt{\psi}_{h_2}(n)\ol{\wt{\psi}_{h_3}(n)}\ol{\wt{\psi}_{h_4}(n)}e(\Theta n)$ to form $\psi_{h_1,h_2,h_3,h_4}^\ast(n) \in \mathrm{Nil}^{d - 2}(M(\delta), m(\delta), k, M(\delta))$. Additionally, we may twist $\psi_{h_1,h_2,h_3,h_4}^\ast$ by an $(h_1,h_2,h_3,h_4)$-dependent complex phase to bring the outer expectation inside the norm. Thus we have
\begin{align*}
&\snorm{\mb{E}_{\substack{h_1 + h_2 = h_3 + h_4\\ h_i\in[N]}}\mb{E}_{n\in[N]}\chi'(h_1,n,\ldots,n) \otimes \chi'(h_2,n + h_1 -h_4,\ldots,n + h_1 -h_4) \otimes \ol{\chi'(h_3,n,\ldots,n)} \\
&\qquad\otimes \ol{\chi'(h_4,n + h_1 -h_4,\ldots,n + h_1 -h_4)}\otimes \widetilde{\chi}_{h_1}(n) \otimes \widetilde{\chi}_{h_2}(n+h_1-h_4) \otimes \ol{\widetilde{\chi}_{h_3}(n)}\\
&\qquad\otimes \ol{\widetilde{\chi}_{h_4}(n+h_1-h_4)} \cdot \psi^{\ast}_{h_1,h_2,h_3,h_4}(n)}_{\infty} \ge \epsilon(\delta).
\end{align*}

By \cite[Lemma C.5]{LSS24b}, $\chi'(h_2,n+h_1-h_4,\ldots,n+h_1-h_4)$ is $(M(\delta), m(\delta))$-equivalent to
\[\bigotimes_{\ell=0}^{s-1}\chi(h_2,n,\ldots,n,h_1-h_4,\ldots,h_1-h_4)\]
where there are $d-k-1$ copies of $n$ and $k$ copies of $h_1-h_4$ and we have a similar expansion for $\chi(h_4,n+h_1-h_4,\ldots,n+h_1-h_4)$. Note that all terms in this expansion except for $\ell = 0$ may be absorbed into $\psi^{\ast}$. Therefore applying Lemma~\ref{lem:equiv}, we have that 
\begin{align*}
&\snorm{\mb{E}_{\substack{h_1 + h_2 = h_3 + h_4\\ h_i\in[N]}}\mb{E}_{n\in[N]}\chi'(h_1,n,\ldots,n) \otimes \chi'(h_2,n,\ldots,n) \otimes \ol{\chi'(h_3,n,\ldots,n)} \\
&\qquad\otimes \ol{\chi'(h_4,n,\ldots,n)} \otimes \wt{\chi}_{h_1}(n) \otimes \wt{\chi}_{h_2}(n+h_1-h_4) \otimes \ol{\wt{\chi}_{h_3}(n)}\\
&\qquad\otimes \ol{\wt{\chi}_{h_4}(n+h_1-h_4)} \cdot \psi^{\ast}_{h_1,h_2,h_3,h_4}(n) \cdot \tau(n,h_1,h_2,h_3,h_4)}_{\infty} \ge \epsilon(\delta);
\end{align*}
here $\tau(n,h_1,h_2,h_3,h_4) \in \mathrm{Nil}^{d - 1}(M(\delta), m(\delta), 4k' + k, M(\delta))$. By \cite[Lemma C.5]{LSS24b}, we have that
\[\chi'(h_1,n,\ldots,n) \otimes \chi'(h_2,n,\ldots,n) \otimes \ol{\chi'(h_3,n,\ldots,n)}\otimes \ol{\chi'(h_4,n,\ldots,n)}\]
and $\chi'(h_1+h_2-h_3-h_4,n,\ldots,n)$ are $(M(\delta), m(\delta))$-equivalent. Thus applying Lemma~\ref{lem:equiv}, we have
\begin{align*}
&\mb{E}_{\substack{h_1 + h_2 = h_3 + h_4\\ h_i\in[N]^{k'}}}\snorm{\mb{E}_{n\in[N]^{k}}\wt{\chi}_{h_1}(n) \otimes \wt{\chi}_{h_2}(n+h_1-h_4) \otimes \ol{\wt{\chi}_{h_3}(n)} \otimes \ol{\wt{\chi}_{h_4}(n+h_1-h_4)}\\
&\qquad\qquad\cdot\wt{\chi}(0,n,\ldots,n) \psi^{\ast}_{h_1,h_2,h_3,h_4}(n) \tau(n,h_1,h_2,h_3,h_4)}_{\infty} \ge \epsilon(\delta);
\end{align*}
here we have folded in various terms into $\tau(n,h_1,\ldots,h_4)$ and the complexity bounds have not changed modulo implicit constants. Note that by \cite[Lemma C.2]{LSS24b}, $\wt{\chi}(0,n,\ldots,n)$ is a degree $(d-1)$ nilsequence in $n$ and thus may abusively also be absorbed into $\tau$. Finally noting that a degree $(d-1)$ nilsequence may also be viewed as a multidegree $(d-1,0,\ldots,0) \cup (d-2,d-1,\ldots,d-1)$ nilsequence and thus applying \cite[Lemma C.6]{LSS24b}, we have 
\begin{align*}
&\snorm{\mb{E}_{\substack{h_1 + h_2 = h_3 + h_4\\ h_i\in[N]^{k'}}}\mb{E}_{n\in[N]^{k}} \wt{\chi}_{h_1}(n) \otimes \wt{\chi}_{h_2}(n+h_1-h_4) \otimes \ol{\wt{\chi}_{h_3}(n)} \otimes \ol{\wt{\chi}_{h_4}(n+h_1-h_4)}\cdot \psi^{\ast}_{h_1,h_2,h_3,h_4}(n)b(n)}_{\infty}\\
&\qquad\ge \epsilon(\delta),
\end{align*}
where $b(n)$ is an $M(\delta)$-bounded function and $\psi_{h_1, h_2, h_3, h_4}^{\ast} \in \mathrm{Nil}^{d - 2}(M(\delta), m(\delta), k, M(\delta))$ for each additive quadruple $h_1 + h_2 = h_3 + h_4$.

Splitting $n = (n_{k'}, n_{k - k'})$ where $n_{k'}$ represents the first $k'$ coordinates of $n$ and $n_{k - k'}$ represents the last $k - k'$ coordinates of $n$, we now reparameterize with
\[h_1 = m-n_{k'}, h_2 = m'-n', h_3 = m' - n_{k'}, h_4 = m-n'.\]
By approximating with regions where we take $m,m',n_{k'},n'$ to live in short intervals, there exist intervals $I_1,\ldots,I_4$ each of density $\epsilon(\delta)$ in $[\pm 2N]^{k'}$ such that 
\begin{align*}
\norm{\mb{E}_{m\in I_1,m'\in I_2, n_{k'}\in I_3,n'\in I_4, n_{k - k'} \in [N]^{k - k'}}&\wt{\chi}_{m-n_{k'}}(n)\otimes\wt{\chi}_{m'-n'}(n')\otimes\ol{\wt{\chi}_{m'-n_{k'}}(n)}\otimes\ol{\wt{\chi}_{m-n'}(n')}\\
&\qquad\cdot\psi^{\ast}_{m-n_{k'},m'-n',m'-n_{k'},m-n'}(n)b(n)}_\infty\ge \epsilon(\delta)
\end{align*}
where $\psi_{m-n_{k'},m'-n',m'-n_{k'},m-n'}$ is a degree $(d-2)$ nilsequence. Now by Cauchy--Schwarz, duplicating the variable $m$ and denoting the copies by $m,m''$, we obtain
\begin{align*}
\norm{\mb{E}_{m,m''\in I_1,m'\in I_2,n_{k'}\in I_3,n'\in I_4, n_{k - k'} \in [N]^{k - k'}}&\wt{\chi}_{m-n_{k'}}(n)\otimes\ol{\wt{\chi}_{m''-n_{k'}}(n)}\otimes\ol{\wt{\chi}_{m-n'}(n')}\otimes\wt{\chi}_{m''-n'}(n')\\
&\cdot\psi^{\ast}_{m-n_{k'},m'-n',m'-n_{k'},m-n',m''-n,m''-n'}(n)}_\infty\ge \epsilon(\delta).
\end{align*}
Note that every term not involving $m$ was removed using appropriate boundedness. Now we may Pigeonhole on $m'-n_{k'} = t$ and deduce
\begin{align*}
\norm{\mb{E}_{m,m''\in I_1,n_{k'}\in I_3,n'\in I_4, n_{k - k'} \in [N]^{k - k'}}&\wt{\chi}_{m-n_{k'}}(n)\otimes\ol{\wt{\chi}_{m''-n_{k'}}(n)}\otimes\ol{\wt{\chi}_{m-n'}(n')}\otimes\wt{\chi}_{m''-n'}(n')\\
&\qquad\cdot\psi^{\ast}_{m-n_{k'},m-n',m''-n_{k'},m''-n'}(n) \cdot \mbm{1}[n_{k'}+t\in I_2]}_\infty\ge \epsilon(\delta).
\end{align*}
Let $m''-n_{k'} = h_1$, $m-n' = h_2$, $m-n_{k'} = h_3$, and $m''-n' = h_4$ (abusively). We have
\begin{align*}
\norm{&\mb{E}_{n\in[N]^k}\mb{E}_{\substack{h_1 + h_2 = h_3 + h_4\\h_i\in[\pm N]^{k'}}} \wt{\chi}_{h_1}(n) \otimes \wt{\chi}_{h_2}(n + h_1 - h_4) \otimes \ol{\wt{\chi}_{h_3}(n)} \otimes \ol{\wt{\chi}_{h_4}(n + h_1 - h_4)}\\
&\quad\cdot\chi_{h_1,h_2,h_3,h_4}(n) \cdot\mbm{1}[n_{k'} + h_3,n_{k'}+h_1\in I_1, n_{k'}\in I_3,n_{k'}+h_1-h_4\in I_4, n_{k'}+t\in I_2]}
_{\infty} \ge \epsilon(\delta)
\end{align*}
where $\chi_{h_1,h_2,h_3,h_4}(n) \in \mathrm{Nil}^{d - 2}(M(\delta), m(\delta), k, M(\delta))$.

Therefore, by the triangle inequality we have that 
\begin{align*}
\mb{E}&_{\substack{h_1 + h_2 = h_3 + h_4\\h_i\in[\pm N]^{k'}}}\norm{\mb{E}_{n\in[N]^k} \wt{\chi}_{h_1}(n) \otimes \wt{\chi}_{h_2}(n + h_1 - h_4) \otimes \ol{\wt{\chi}_{h_3}(n)} \otimes \ol{\wt{\chi}_{h_4}(n + h_1 - h_4)}\\
&\qquad\chi_{h_1,h_2,h_3,h_4}(n) \cdot\mbm{1}[n_{k'} + h_3,n_{k'}+h_1\in I_1, n_{k'}\in I_3,n_{k'}+h_1-h_4\in I_4, n_{k'}+t\in I_2]}
_{\infty} \gtrsim \epsilon(\delta).
\end{align*}
As the last term is a $\vec{h}$-dependent indicator, this gives the desired estimate.

Next, by \cite[Lemma C.5]{LSS24b}, there exist $\chi_h'' \in \mathrm{Nil}^{(1, 1, \dots, 1)}(M(\delta), m(\delta), k, M(\delta))$ for which $\chi_h(n)$ is $(M(\delta), m(\delta))$-equivalent to $\chi_h''(n, n, \dots, n)$. We invoking Lemma~\ref{lem:equiv} we may replace $\chi_h$ with $\chi_h''$ at the cost of a lower degree nilsequence with appropriate complexity parameters and by the Cauchy-Schwarz-Gowers inequality (in particular Cauchy-Schwarzing away the factors of $f_{i, h}$), we see that
$$\|\|\chi_{h_1}''(n, \dots, n) \otimes \chi_{h_2}''(n, \dots, n)\otimes \ol{\chi_{h_3}''(n, \dots, n)}\otimes \ol{\chi_{h_4}''(n, \dots, n)} \cdot \psi_{h_1, \dots, h_4}\|_{U(e_1[N], \dots, e_\ell[N])}\|_\infty \ge \epsilon(\delta).$$
Unfortunately, $\chi'_h(n)$ is not approximately linear in $n$; rather, it is only approximately \emph{multilinear} in $n$. (Note that if $\chi_h$ is degree zero in $e_1, \dots, e_\ell$, then taking a difference in the Box norm yields a strictly lower degree nilsequence, which is undesirable; this is why this slightly strange condition of the lemma is present.) This can be remedied by taking more and more differences. Hence,
$$\|\|\chi_{h_1}''(n, \dots, n) \otimes \chi_{h_2}''(n, \dots, n)\otimes \ol{\chi_{h_3}''(n, \dots, n)}\otimes \ol{\chi_{h_4}''(n, \dots, n)} \cdot \psi_{h_1, \dots, h_4}\|_{U(e_1[N], \dots, e_\ell[N], [N]^k, \dots, [N]^k)}\|_\infty \ge \epsilon(\delta)$$
where we add sufficiently many differences so that the number of differences is $d - 1$. By expanding out the box norm, using multilinear properties of nilcharacters, and pigeonholing in the sufficient variables, we find that
$$\|\mathbb{E}_{n \in ([N]^k)^d} (\chi_{h_1}''(n) \otimes \chi_{h_2}''(n) \otimes \overline{\chi_{h_3}''(n)\otimes \chi_{h_4}''(n)})^{\otimes q} \cdot \psi_{\vec{h}}(n)\|_\infty \ge \epsilon(\delta)$$
for some $q = O_K(1)$. This, combined with \cite[Lemma C.3]{LSS24b}, implies that $\chi_{h_1}'' \otimes \chi_{h_2}'' \otimes \overline{\chi_{h_3}'' \otimes \chi_{h_4}''}$ is degree $\le d - 2$, and since $\chi_h$ is equivalent to $n \mapsto \chi_h''(n, n, \dots, n)$, we find that $\chi_{h_1} \otimes \chi_{h_2} \otimes \overline{\chi_{h_3} \otimes \chi_{h_4}}$ is degree $\le d - 2$. This gives the desired conclusion.
\end{proof}

\section{Deduction of Proposition~\ref{prop:degreerankextraction}}
We will now deduce Proposition~\ref{prop:degreerankextraction}. The individual steps will follow \cite{LSS24b} rather closely, with the biggest difference is that of the construction of the nilmanifold in the proof of Proposition~\ref{prop:degreerankextraction}.
\subsection{A sunflower-type argument} We prove the following lemma, which is essentially the same as \cite[Lemma 6.1]{LSS24b}.
\begin{lemma}\label{lem:sunflower}
Let $\epsilon > 0$, $(r \le d), D, D' \le K$ be positive integers, and $d > 1$. Let $\chi_h(x) = F(g_h(x)\Gamma) \in \mathrm{Nil}^{(d, r)}(M(\delta), m(\delta), D, M(\delta))$ with $G_{(d, r)}$ frequency $\eta$ and $\psi_{\vec{h}} = \psi_{\vec{h}}(g_{\vec{h}}(n)\Gamma') \in \mathrm{Nil}^{d - 1}(M(\delta), m(\delta), D, 1)$. Let $H \subseteq [N]^{D'}$ be a subset of size at least $\delta N^{D'}$ and suppose for at least $\delta |H|^3$ many additive quadruples $(h_1, h_2, h_3, h_4) \in H^4$, we have
$$\|\mathbb{E}_{x \in [N]^{D}} \chi_{h_1}(x) \otimes \chi_{h_2}(x) \otimes \overline{\chi_{h_3}(x)\otimes \chi_{h_4}(x)} \cdot \psi_{\vec{h}}(x)\|_\infty \ge \epsilon(\delta).$$
Then there exists a set $H' \subseteq H$ of size at least $\epsilon(\delta)|H|$ such that for any $h \in H'$, we have the following.
\begin{itemize}
\item There exists a collection of $\mathbb{R}$-vector spaces $V_{i, \mathrm{Dep}} \le V_i \le G_{(i, 1)}/G_{(i, 2)}$ which are $M(\delta)$-rational with respect to $\exp(\mathcal{X}_i)$ for each $i$;
\item We have a factorization $g_h = \epsilon_hg_h'\gamma_h$ with $\epsilon_h$ $(M(\delta), \vec{N})$-smooth and $\gamma_h$ $M(\delta)$-rational;
\item For $\vec{i}$ with $1 \le |\vec{i}| \le \ell' - 1$ and $h, h_1, h_2 \in H'$, we have
$$\mathrm{Taylor}_{\vec{i}}(g_h) \in V_{|\vec{i}|}, \mathrm{Taylor}_{\vec{i}}(g_{h_1}) - \mathrm{Taylor}_{\vec{i}}(g_{h_2}) \in V_{|\vec{i}|, \mathrm{Dep}};$$
\item For integers $i_1 + i_2 + \cdots + i_{r} = \ell' - 1$, and $v_{i_\ell} \in V_{i_\ell}$ such that for at least two distinct indices $\ell_1, \ell_2$ we have $v_{i_{\ell_1}} \in V_{i_{\ell_1}, \mathrm{Dep}}$ and $v_{i_{\ell_2}} \in V_{i_{\ell_2}, \mathrm{Dep}}$, we have for any $r - 1$-fold commutator,
$$\eta([v_1, v_2, \dots, v_{i_{r}}]) = 0.$$
\end{itemize}
\end{lemma}
\begin{proof}

The proof is essentially the same as in \cite[Lemma 6.1]{LSS24b} and we repeat large portions of it here essentially verbatim. (The only meaningful difference is in step 5 of this proof.) First, we note that the statement of the lemma is trivial for $r = 1$ as the only content of the lemma is the third bullet point, and there we may simply take $V_i = V_{i, \mathrm{Dep}} = G_{(i, 1)}/G_{(i, 2)}$. 

\noindent\textbf{Step 1: Setup for invoking equidistribution theory.}
By hypothesis, we have
\begin{align*}
\norm{\mb{E}\bigg[\chi_{h_1}(n)\otimes \chi_{h_2}(n + h_1-h_4)\otimes \ol{\chi_{h_3}(n)}\otimes \ol{\chi_{h_4}(n + h_1-h_4)}\cdot\psi_{\vec{h}}(g_{\vec{h}}(n)\Gamma')\bigg]}_{\infty} \ge \delta.
\end{align*}
for at least $\rho$ fraction of additive quadruples $h_1+h_2=h_3+h_4$. Let $\chi_h(n) = F_h(g_h(n)\Gamma)$. By pigeonholing in the nilmanifolds of complexity $M(\delta)$ and using a similar partition of unity argument as in Proposition~\ref{prop:initialmaneuvers}, we may assume that $G_{\mathrm{Error}}/\Gamma_{\mathrm{Error}}$, $G_h/\Gamma_h$, and $F_h$ are independent of $h$. We may additionally assume that $g_{\vec{h}}(0) = \mr{id}_{G_{\mr{Error}}}$ via \cite[Lemma 2.1]{Len23b}. 

We now consider the group $\wt{G} = G\times G \times G \times G \times G_{\mr{Error}}$. $\wt{G}$ may naturally be given a degree-rank $(d,r)$ product filtration and Mal'cev basis, with $G_{\mathrm{Error}}$ may be given a degree-rank $(d - 1, d - 1)$ filtration (note that in our definition, we have $(d - 1, d - 1) < (d, r)$). Note that $F(x_1\Gamma)\otimes \ol{F(x_2\Gamma)\otimes F(x_3\Gamma)}\otimes F(x_4\Gamma)\cdot\psi_{\vec{h}}(x_5\Gamma_{\mr{Error}})$ has a vertical frequency $\eta_{\mr{Prod}}=(\eta,\eta,-\eta,-\eta,0)$.

For the sake of convenience, we set \[g_{\vec{h}}^\ast(n) = (g_{h_1}(n),g_{h_2}(n),g_{h_3}(n),g_{h_4}(n),g_{\vec{h}}(n)).\]

\noindent\textbf{Step 2: Invoking equidistribution theory.}
By applying \cite[Corollary 5.2]{LSS24b} (since $\eta$ is nonzero), there exists a $M(\delta)$-rational subgroup $J=J_{\vec{h}}$ of $\wt{G}$ such that $\eta_{\mr{Prod}}(J\cap\wt{G}_{(s-1,r)}) = 0$ and such that 
\[g_{\vec{h}}^\ast = \eps_{\vec{h}}\cdot\wt{g_{\vec{h}}} \cdot\gamma_{\vec{h}}\]
where:
\begin{itemize}
    \item $\eps_{\vec{h}}(0) = \wt{g_{\vec{h}}}(0) = \gamma_{\vec{h}}(0) = \mr{id}_{\wt{G}}$;
    \item $\wt{g_{\vec{h}}}$ takes values in $J$;
    \item $\gamma_{\vec{h}}$ is $M(\delta)$-rational (with respect to the lattice $\Gamma\times\Gamma\times\Gamma\times\Gamma\times\Gamma_{\mr{Error}}$);
    \item $\eps$ is $(M(\delta), \vec{N})$-smooth.
\end{itemize}
By pigeonholing to a subset of additive quadruples of density $M(\delta)$ we may in fact assume that the group $J$ is independent of $\vec{h}$ under consideration. Recall the definition of $\mathrm{Horiz}_i$, which we define below.
\begin{definition}[$i$th horizontal component]
Let $G$ be given a degree-rank $(d', r')$ filtration. We define $\mathrm{Horiz}_i(G) = G_{(i, 1)}/G_{(i, 2)}$.
\end{definition}

We define
\[J_i' := (J\cap\wt{G}_{(i,1)})/(J\cap\wt{G}_{(i,2)}),\qquad J_i := \tau_i(J_i')\]
where $\tau_i\colon\on{Horiz}_i(G)^{\otimes 4}\times\on{Horiz}_i(G_{\mr{Error}})\to\on{Horiz}_i(G)^{\otimes 4}$ is the natural projection map to the four-fold product. Since $\eta_{\mr{Prod}}(J\cap\wt{G}_{(s-1,r)}) = 0$, we have 
\[\eta_{\mr{Prod}}([J_{i_1}',\ldots,J_{i_{r}}']) = 0\]
for $i_1' + \ldots + i_{r}' = d-1$ where the commutator bracket is taken with respect to $\wt{G}$ and $[\cdot,\ldots,\cdot]$ denotes any possible $(r-1)$-fold commutator bracket. 

Note that since $\eta_{\mathrm{Prod}}$ vanishes on $(G_{\mathrm{Error}})_{(d, r)}$ we in fact have
\[\eta_{\mr{Prod}}([J_{i_1},\ldots,J_{i_{r}}]) = 0\]
where we abusively descend $\eta_{\mr{Prod}}$ to $G^{\otimes 4}$.

\noindent\textbf{Step 3: Furstenberg--Weiss commutator argument.}
We now perform the crucial Furstenberg--Weiss commutator argument to find $V_i$ and $V_{i, \mathrm{Dep}}$. Given $T\subseteq[4]$, we define $\pi_T((v_1,\ldots,v_4)) = (v_i)_{i\in T}$ with the coordinates represented in increasing order of index.

We define
\begin{align*}
\pi_{123}(J_i)^\ast &= \pi_{123}(J_i) \cap\{(v,0,0)\colon v\in\on{Horiz}_i(G)\},\\
\pi_{124}(J_i)^\ast &= \pi_{124}(J_i) \cap\{(v,0,0)\colon v\in\on{Horiz}_i(G)\}.
\end{align*}
Note that $\pi_{123}(J_i)^\ast$ and $\pi_{124}(J_i)^\ast$ may (abusively) be viewed as subspaces of $\on{Horiz}_i(G)$. The crucial claim is that 
\[\eta([v_{i_1},\ldots, v_{i_{r}}]) = 0\]
if $i_1+\cdots+i_{r} = s-1$, each $v_{i_\ell}\in\pi_1(J_{i_\ell})$, and for two distinct indices $\ell_1,\ell_2$ we have that $v_{i_{\ell_1}}\in\pi_{123}(J_{i_{\ell_1}})^\ast$ and $v_{i_{\ell_2}}\in\pi_{124}(J_{i_{\ell_2}})^\ast$. Note that $\eta$ lives on $G$ and the commutator brackets are taken with respect to $G$, not $G^{\otimes 4}$. 

Note that an element $v_i\in\pi_1(J_{i_\ell})$ lifts to an element $\wt{v_{i_\ell}}$ of the form $(v_{i_\ell},\cdot,\cdot,\cdot)\in\on{Horiz}_{i_\ell}(G)^{\otimes 4}$. Furthermore note that $v_{i_\ell}\in\pi_{123}(J_{i_\ell})$ ``lifts'' to an element $\wt{v_{i_\ell}}$ of the form $(v_{i_\ell},0,0,\cdot)\in\on{Horiz}_{i_\ell}(G)^{\otimes 4}$ while $v_{i_\ell}\in\pi_{124}(J_{i_\ell})$ lifts to an element $\wt{v_{i_\ell}}$ of the form $(v_{i_\ell},0,\cdot,0)\in\on{Horiz}_{i_\ell}(G)^{\otimes 4}$.

Given the above setup, we have 
\[[\wt{v_{i_{1}}},\ldots,\wt{v_{i_{{r}}}}] = ([v_{i_{1}},\ldots,v_{i_{{r}}}],\mr{id}_G,\mr{id}_G,\mr{id}_G).\]
To see this note that the iterated commutator of elements in $G\times\mr{Id}_G\times\mr{Id}_G \times G$ (with any elements in $G^{\otimes 4}$) remains in the subgroup $G\times\mr{Id}_G\times\mr{Id}_G \times G$; an analogous fact holds true for $G\times\mr{Id}_G\times G \times\mr{Id}_G$. Since we assumed that our commutator contains elements in both $G\times\mr{Id}_G\times\mr{Id}_G \times G$ and $G\times\mr{Id}_G\times G \times\mr{Id}_G$, the commutator must in fact live in $G\times\mr{Id}_G\times\mr{Id}_G \times\mr{Id}_G$, and the first coordinates of the desired commutators is trivially seen to match. 

Recalling that we have
\[\eta_{\mr{Prod}}([J_{i_1},\ldots,J_{i_{r}}]) = 0,\]
and noting that $\eta_{\mr{Prod}}$ descends to $\eta$ on the subgroup $G_{(s-1,r)}\times\mr{Id}_{G}^{\otimes 3}$, we have 
\[\eta([v_{i_{1}},\ldots,v_{i_{{r}}}]) = 0\]
as claimed.

\noindent\textbf{Step 4: Finding $(h_2,h_3)$ and $(h_2',h_4')$ which extend to many ``good'' $h_1$.}
Recall that we are looking at the at least $\epsilon(\delta)$ fraction of additive quadruples $(h_1,h_2,h_3,h_4)\in H^4\subseteq[N]^{4D'}$ which are such that $\wt{g_{\vec{h}}}$ lives on a specified subgroup $J$. Call this set of quadruples $\mc{S}$.
By the pigeonhole principle, there are at least $\epsilon(\delta)N^{D'}$ many $h_1\in[N]^{D'}$ which extend to at least $\epsilon(\delta)N^{2D'}$ quadruples in $\mc{S}$. Thus there are at least $\epsilon(\delta)N^{5D'}$ pairs of additive tuples of the form
\[(h_1,h_2,h_3,h_1+h_2-h_3),~(h_1,h_2',h_1 + h_2'-h_4',h_4')\in\mc{S}.\] 

By averaging, there exists a pair of pairs $(h_2,h_3)$
and $(h_2',h_4')$ such that there are at least $\epsilon(\delta)N^{D'}$ many $h_1\in[N]^{D'}$ which live in such additive tuples. We fix such a pair of pairs and define $\mc{T}$ to denote the set of $h_1\in[N]^{D'}$ such that $(h_1,h_2,h_3,h_1+h_2-h_3)\in\mc{S}$ and $(h_1,h_2',h_1+h_2'-h_4',h_4')\in\mc{S}$.

\noindent\textbf{Step 5: Extracting coefficient data.}
Consider $h_1\in\mc{T}$ and define
\[h^{123} = (h_1,h_2,h_3,h_1+h_2-h_3),\quad h^{124} = (h_1,h_2',h_1+h_2'-h_4',h_4').\]
Recall $\mc{X}_i = (\mc{X}\cap\log(G_{(i,1)}))/\log(G_{(i,2)})$ and assign the basis $\exp(\mc{X}_i)$ to $G_{(i,1)}/G_{(i,2)}$ (viewed as a vector space). 
Finally we assign the basis $\mc{Z}_i=\bigcup_{X_i\in\exp(\mc{X}_i)}\{(X_i,0,0), (0,X_i,0),(0,0,X_i)\}$ to $(G_{(i,1)}/G_{(i,2)})^{\otimes 3}$. Write
$$J_{i, 123} = \{(g_1, g_2, g_3): v_1^i(g_1) + v^i_2(g_2) + v^i_3(g_3) = 0 \forall (v_1, v_2, v_3) \in S_{i, 123}\}$$
$$J_{i, 124} = \{(g_1, g_2, g_4): v_1^i(g_1) + v_i^2(g_2) + v^i_4(g_4) = 0 \forall (v_1, v_2, v_4) \in S_{i, 124}\}.$$
Fix $i$ and let the projection to $h_1$ component of elements in $S_{i, 123} \cup S_{i, 124}$ be $(\tilde{v}^{1}, \dots, \tilde{v}^{r})$ and suppose without a loss of generality that $\tilde{v}^i$ are linearly independent. We take $V_{i, \mathrm{Dep}} = \pi_1(J_{i, 123}) \cap \pi_2(J_{i, 124})$, and $V_i = \pi_1(J_i)$. By our analysis in Step 3, it follows that for this choice of vector spaces satisfies the fourth bullet point of the lemma.

Let us also assume without a loss of generality that $\tilde{v}^1, \dots, \tilde{v}^r$ are Gaussian eliminated and (after possibly reordering the variables) are in row-reduced echelon form. This can be done while keeping similar bounds for these elements. Now let $\vec{i}$ be a $D$-dimensional integer vector. Noticing that $h_2, h_3, h_2', h_4'$ are fixed, there exists fixed $h$-independent phases $\alpha_1, \dots, \alpha_r$ such that for each $h \in \mathcal{T}$, we have for some integer $T_h \le M(\delta)$,
$$\tilde{v}^j(\mathrm{Taylor}_{\vec{i}}(g_h)) + \alpha_j \in T_h^{-1}\mathbb{Z} + M(\delta) N^{-|\vec{i}|}.$$
Pigeonholing in $T$, we have for some $H' \subseteq \mathcal{T}$ of size at least $\epsilon(\delta)|H|$ such that
$$\tilde{v}^j(\mathrm{Taylor}_{\vec{i}}(g_h)) + \alpha_j \in T^{-1}\mathbb{Z} + M(\delta) N^{-|\vec{i}|}.$$
This shows that for $h, h' \in \mc{T}$,
$$\tilde{v}^j(\mathrm{Taylor}_{\vec{i}}(g_hg_{h'}^{-1})) \in T^{-1}\mathbb{Z} + M(\delta) N^{-|\vec{i}|}.$$
In addition, since $\mathrm{Taylor}_{\vec{i}}(g_{h_1}, g_{h_2}, g_{h_3}, g_{h_4}) \in J_{|\vec{i}|}$, it follows that $\mathrm{Taylor}_{\vec{i}}(g_{h_1}) \in \pi_1(J_{|\vec{i}|}) = V_{|\vec{i}|}$. Applying the factorization theorem to $g_hg_{h'}^{-1}$, we have obtained the third bullet point of the lemma. This completes the proof of the lemma.
\end{proof}

\subsection{Linearization}
We will now refine the nilsequence so that it admits bracket linear behavior in $h$. This is essentially \cite[Lemma 9.1]{LSS24b}. 
\begin{lemma}
Let $\delta > 0$ and $d, r, D, D', K$ be positive integers with $d, r, D, D', \le K$. For $h \in [N]^D$, let $\chi_h(x) = F(g_h(x)\Gamma) \in \mathrm{Nil}^d(M(\delta),m(\delta), D, M(\delta))$ be a degree-rank $(d, r)$ nilcharacter $G_{(d, r)}$-vertical frequency $\eta$ and $\psi_{\vec{h}} = \psi_{\vec{h}}(g_{\vec{h}}(n)\Gamma') \in \mathrm{Nil}^{d - 1}(M(\delta), m(\delta), D, M(\delta))$ a nilsequence on $G_{\mathrm{Error}}/\Gamma_{\mathrm{Error}}$. Let $H \subseteq [N]^{D'}$ be a subset of size at least $\delta N^{D'}$ and suppose for $\delta |H|^3$ many additive quadruples $(h_1, h_2, h_3, h_4) \in H^4$ we have
$$\|\mathbb{E}_{x \in [N]^{D}} \chi_{h_1}(x) \otimes \overline{\chi_{h_2}(x) \otimes \chi_{h_3}(x)}\otimes \chi_{h_4}(x) \cdot \psi_{\vec{h}}(x)\|_\infty \ge \epsilon(\delta).$$
Then there exists a set $H' \subseteq H$ with $|H'| \ge \epsilon(\delta)|H|$ such that for any $h \in H'$, we have the following.
\begin{itemize}
\item For each $h \in H'$, there exists a factorization $g_h = \epsilon_h g_h' \gamma_h$ with $\epsilon_h$ $(M(\delta), N^D)$-smooth, $\gamma_h$ is $M(\delta)$-rational; 
\item There exists a collection $\mathcal{X}_{i, *}$, $\mathcal{X}_{i, \mathrm{lin}}$, and $\mathcal{X}_{i, \mathrm{pet}}$ which are $M(\delta)$-rational combinations of elements in $(\mathcal{X} \cap G_{(i, 1)})/G_{(i, 2)}$. 
\item There is a collection of $\mb{R}$-vector spaces $W_{i,\ast}, W_{i,\mr{Lin}}, W_{i,\mr{Pet}}\leqslant G_{(i,1)}/G_{(i,2)}$ for each $i$;
\item If $W_i := W_{i,\ast} +  W_{i,\mr{Lin}} + W_{i,\mr{Pet}}$ then $\dim(W_i) = \dim(W_{i,\ast}) +  \dim(W_{i,\mr{Lin}}) + \dim(W_{i,\mr{Pet}})$, i.e., the three spaces are linearly disjoint;
\item There exist bases $\mc{X}_{i,\ast}$, $\mc{X}_{i,\mr{Lin}}$, and $\mc{X}_{i,\mr{Pet}}$ of the corresponding spaces which are composed of $M(\delta)$-rational combinations of elements of $(\mc{X}\cap G_{(i,1)})/G_{(i,2)}$;
\item For $1\le |\vec{i}| \le d$ and $h,h_1,h_2\in H'$ we have
    \begin{align*}
    \on{Taylor}_{\vec{i}}(g_h') &\in W_{|\vec{i}|,\ast} + W_{|\vec{i}|,\mr{Lin}} + W_{|\vec{i}|,\mr{Pet}} = W_{|\vec{i}|},\\
    \on{Taylor}_{\vec{i}}(g_{h_1}') - \on{Taylor}_{\vec{i}}(g_{h_2}')&\in W_{|\vec{i}|,\mr{Lin}} + W_{|\vec{i}|,\mr{Pet}},\\
    \on{Proj}_{W_{|\vec{i}|,\mr{Lin}}}(\on{Taylor}_{\vec{i}}(g_h')) &= \sum_{Z_{|\vec{i}|,j}\in\mc{X}_{|\vec{i}|,\mr{Lin}}}\bigg(\gamma_{\vec{i},j} + \sum_{k=1}^{m^\ast}\alpha_{\vec{i},j,k}\{\beta_k\cdot h\}\bigg)Z_{|\vec{i}|,j},
    \end{align*}
    with $m^\ast\le m(\delta)$ and $\beta_k\in(1/N')\mb{Z}^{D'}$ where $N'$ is a prime in $[100N,200N]$;
\item For any integers $i_1 + \cdots + i_{r} = d$, suppose that $v_{i_j}\in V_{i_j}$ for all $j$. If for at least one index $\ell_0$ we have $v_{i_{\ell_0}}\in W_{i_{\ell_0},\mr{Pet}}$, then if $w$ is any $(r-1)$-fold commutator of $v_{i_1},\ldots,v_{i_{r}}$ we have 
    \[\eta(w) = 0.\]
    Furthermore, if instead for at least two indices $\ell_1,\ell_2$ we have $v_{i_{\ell_1}}\in W_{i_{\ell_1},\mr{Lin}}$ and $v_{i_{\ell_2}}\in W_{i_{\ell_2},\mr{Lin}}$, then if $w$ is any $(r-1)$-fold commutator of $v_{i_1},\ldots,v_{i_{r}}$ we have
    \[\eta(w) = 0.\]
\end{itemize}
\end{lemma}
\begin{proof}
Again, much of this proof will essentially repeat the proof of \cite[Lemma 9.1]{LSS24b} essentially verbatim. For the majority of the proof we will assume $d\ge 2$; we indicate the minor changes required for $d = 1$ for the end of the proof. Note that the case when $\eta$ is trivial follows via taking $W_{i,\mr{Pet}} = G_{(i,1)}/G_{(i,2)}$, $W_{i,\mr{Lin}}$ and $W_{i,\ast}$ to be trivial; therefore we may assume that $\eta$ is nontrivial for the remainder of the proof.

\noindent\textbf{Step 1: Applying Lemma~\ref{lem:sunflower} and linear-algebraic setup.}
Applying Lemma~\ref{lem:sunflower}, we obtain a factorization $g_h = \epsilon_h g'_h \gamma_h$ for $h \in H'$ of size at least $\epsilon(\delta)|H|$ such that $\epsilon_h$ is $(M(\delta), \vec{N})$-smooth, $\gamma_h$ is $M(\delta)$-rational, and $g_h'$ satisfies the conclusions of the lemma. By restricting in a progression of the period of $\gamma$, Fourier expanding the progression (via e.g., \cite[Lemma 7.1]{LSS24b}), pigeonholing in a phase $e(\Theta_h n)$, and absorbing the phase in $\psi_h$ (and here we are using the fact that $d > 1$), we may assume that (after possibly also conjugating $g_h'$ by a constant $\epsilon_0$)
\[\|\mathbb{E}_{x \in [N]^{D}} \chi_{h_1}'(x) \otimes \overline{\chi_{h_2}'(x) \otimes \chi_{h_3}'(x)}\otimes \chi_{h_4}'(x) \cdot \psi_{\vec{h}}(x)\|_\infty \ge \epsilon(\delta)\]
where if $\chi_h(n) = F(g_h(n)\Gamma)$, $\chi_h'(n) = \widetilde{F}(g_h'(n)\Gamma)$ with $\widetilde{F} = F(\epsilon_0 \cdot )$ for some choice of $\epsilon_0$ with $|\psi_{\mathcal{X}}(\epsilon_0)| \le 1$. We shall now abusively relabel $g_h'$ as $g_h$, $H'$ as $H$, and $\chi_h'$ as $\chi_h$.

We now require some additional linear algebraic setup for the proof. We define
\begin{align*}
V_{i,\mr{Dep}} &= \on{span}_{\mb{R}}(w_{i,1},\ldots,w_{i,\dim(V_{i,\mr{Dep}})})\\
V_i &= \on{span}_{\mb{R}}(w_{i,1},\ldots,w_{i,\dim(V_{i,\mr{Dep}})},w_{i,\dim(V_{i,\mr{Dep}})+1},\ldots,w_{i,\dim(V_i)}),\\
 G_{(i,1)}/G_{(i,2)}&= \on{span}_{\mb{R}}(w_{i,1},\ldots,w_{i,\dim(\on{Horiz}_i(G))}).
\end{align*}
Given $v\in V_i$, there is a unique linear combination 
\[v = \sum_{j=1}^{\dim(V_{i})}\alpha_jw_{i,j}.\]
We define 
\[P_i v = \sum_{j=\dim(V_{i,\mr{Dep}})+1}^{\dim(V_{i})}\alpha_jw_{i,j},\qquad Q_i v = \sum_{j=1}^{\dim(V_{i,\mr{Dep}})}\alpha_jw_{i,j}.\]
By construction $P_i^2 = P_i$, $Q_i^2 = Q_i$, $Q_i(V_i) \cap P_i(V_i) = 0$, and $P_i v + Q_i v = v$ for $v\in V_i$. We also (abusively) extend the operator $P_i$ to $V_i^{\otimes \ell}$ and $(G_{(i,1)}/G_{(i,2)})^{\otimes 4}$ in the obvious manners by acting on each copy of $V_i$ separately (and zeroing out basis elements $w_{i,\dim(V_i)+1},\ldots,w_{i,\dim(\on{Horiz}_i(G))}$).

\noindent\textbf{Step 2: Invoking equidistribution theory.}
We thus have
\begin{align*}
\snorm{\mb{E}[\chi_{h_1}(n)\otimes \chi_{h_2}(n)\otimes \ol{\chi_{h_3}(n)}\otimes \ol{\chi_{h_4}(n)}\cdot\psi_{\vec{h}}(g_{\vec{h}}(n)\Gamma')]}_{\infty} \ge \epsilon(\delta)
\end{align*}
for a $\epsilon(\delta)$ density of additive tuples. We define $G_{\mr{Error}}$, $\wt{G}$, and $\eta_{\mr{Prod}}$ as in the proof of Lemma~\ref{lem:sunflower} and as before we may assume that $g_{\vec{h}}(0) = \mr{id}_{G_{\mr{Error}}}$. Define 
\[g_{\vec{h}}^\ast(n) = (g_{h_1}(n),g_{h_2}(n ),g_{h_3}(n),g_{h_4}(n ),g_{\vec{h}}(n)).\]
By applying \cite[Corollary 5.5]{LSS24b}, we have
\[g_{\vec{h}}^\ast =\eps_{\vec{h}} \cdot\wt{g_{\vec{h}}}\cdot\gamma_{\vec{h}}\]
with
\begin{itemize}
    \item $\eps_{\vec{h}}(0) = \wt{g_{\vec{h}}}(0) = \gamma_{\vec{h}}(0) = \mr{id}_{\wt{G}}$;
    \item $\wt{g_{\vec{h}}}$ takes values in $K$;
    \item $\gamma_{\vec{h}}$ is $M(\delta)$-rational;
    \item $\eps$ is $(M(\delta), \vec{N})$-smooth
\end{itemize}
where $\eta_{\mr{Prod}}(K\cap\wt{G}_{(d-1,r)}) = 0$ and $K$ is a $M(\delta)$-rational subgroup of $\wt{G}$. By passing to a subset of additive quadruples of density $\epsilon(\delta)$ we may in fact assume that the group $K$ is independent of $\vec{h}$ under consideration.

\noindent\textbf{Step 3: Linear algebra deductions from equidistribution theory.}
Note that at present the subgroup $K$ does not account for the deductions given in Lemma~\ref{lem:sunflower}; these initial deductions are designed essentially to account for this. Let $\tau_i\colon\on{Horiz}_i(G)^{\otimes 4}\times\on{Horiz}_i(G_{\mr{Error}})\to\on{Horiz}_i(G)^{\otimes 4}$ be the natural projection to the four-fold product. We define the following set of vector spaces:
\begin{align*}
R_i &:= \{(Q_iv_1,Q_iv_2,Q_iv_3,Q_iv_4)\in V_i^{\otimes 4}\colon Q_iv_1 + Q_iv_2 - Q_iv_3 -Q_iv_4 = 0\},\\
K_i &:= \tau_i(K\cap\wt{G}_{(i,1)}\imod\wt{G}_{(i,2)}),\\
S_i &:= \{(v_1,v_2,v_3,v_4)\in V_{i}^{\otimes 4}\colon P_{i} v_1 =  P_{i} v_2 = P_{i} v_3 = P_{i} v_4\},\\
K_{i,1} &:= K_i \cap S_i,\\
\wt{K_i} &:= K_{i,1} + R_i,\\
L_i &:= \pi_1(\wt{K_i} \cap\{(v_1,v_2,v_3,v_4)\in V_i^{\otimes 4}\colon Q_i v_2 = Q_iv_3 = Q_i v_4 = 0\}) + Q_i(V_i).
\end{align*}
By inspection, we have $R_i\leqslant S_i$ hence $\wt{K_i}\leqslant S_i$. Note that 
\[\eta_{\mr{Prod}}([K_{i_1,1},\ldots,K_{i_r^{\ast},1}]) = 0\]
whenever one has that $i_1 + \cdots + i_r = d-1$ and $[\cdot,\ldots,\cdot]$ denotes any possible $(d -1)$-fold commutator bracket. This is a consequence of the fact that $\eta_{\mr{Prod}}(K\cap\wt{G}_{(d-1,r)}) = 0$ and noting that $K_{i,1}\leqslant K_i$. Note that we are implicitly using that $\eta_{\mr{Prod}}$ is trivial on $G_{\mr{Error}}$ as well, and we abusively descend $\eta_{\mr{Prod}}$ to $G^{\otimes 4}$.

We now claim that 
\[\eta_{\mr{Prod}}([v_{i_1},\ldots,v_{i_{r}}]) = 0\]
if $v_{i_\ell}\in S_{i_\ell}$ for all $\ell$ and there is at least one index $j$ such that $v_{i_j}\in R_{i_j}$.

To prove this, note by the final bullet point of Lemma~\ref{lem:sunflower} and multilinearity that 
\begin{align*}
\eta_{\mr{Prod}}([v_{i_1},\ldots,v_{i_{r}}]) &= \eta_{\mr{Prod}}([P_{i_1}v_{i_1},\ldots,P_{i_{r}}v_{i_{r}}]) + \sum_{k=1}^{r}\eta_{\mr{Prod}}([P_{i_1}v_{i_1},\ldots, Q_{i_k}v_{i_k},\ldots,P_{i_{r}}v_{i_{r}}])\\
&= \eta_{\mr{Prod}}([P_{i_1}v_{i_1},\ldots, Q_{i_j}v_{i_j},\ldots,P_{i_{r}}v_{i_{r}}])=0.
\end{align*}
The first equality uses that every bracket with at least two $Q_{i_k}v_{i_k}$ has two $V_{i,\mr{Dep}}$ terms so is $0$, the second equality uses $P_{i_j}v_{i_j} = 0$, and the third equality follows by noting that 
\[P_{i}v_{i} \in\{(v_1,v_2,v_3,v_4)\in V_i^{\otimes 4}\colon P_i v_1 =  P_i v_2 = P_i v_3 = P_i v_4, Q_i v_1 =  Q_i v_2 = Q_i v_3 = Q_i v_4 = 0\}\]
and $\eta_{\mr{Prod}} = (\eta,\eta,-\eta,-\eta)$.
Now, we may ultimately deduce
\[\eta_{\mr{Prod}}([\wt{K_{i_1}},\ldots,\wt{K_{i_{r}}}])=0\]
because $\wt{K_i} = K_{i,1} + R_i$ and $R_i,K_{i,1}\leqslant S_i$.

Finally, let $\pi_T$ for $T\subseteq [4]$ is as in the proof of Lemma~\ref{lem:sunflower} (namely, an appropriate projection map). We have
\[\pi_1(\wt{K_i})\leqslant L_i.\]
This follows because if
\[((Qv_1,Pv_1),(Qv_2,Pv_2),(Qv_3,Pv_3),(Qv_4,Pv_4))\in\wt{K_i} \]
then
\[((Q(v_1+v_2-v_3-v_4),Pv_1),(0,Pv_2),(0,Pv_3),(0,Pv_4))\in\wt{K_i}.\]

\noindent\textbf{Step 4: Constructing a decomposition of $Q_i(V_i)$.}
We will now decompose $Q_i(V_i) = V_{i,\mr{Dep}}$ into a pair of subspaces. On one of these subspaces we will deduce an improved vanishing for the commutator while on the other subspace we will deduce an approximate linearity for $\on{Taylor}_i(g_h)$. Let 
\[L_i^\ast = \{(v_1,v_2,v_3,v_4)\in S_i\colon Pv_1 = 0, v_2 = v_3 = v_4 = 0\} \cap\wt{K_i}.\]
Note that $L_i^\ast$ may abusively be viewed as a subspace of $V_i$ (instead of $V_i^{\otimes 4}$) and under this identification $L_i^\ast\leqslant Q_i(V_i)=V_{i,\mr{Dep}}\leqslant L_i$.

The key claim in our analysis is if $i_1 + \cdots + i_{r} = d-1$, $v_{i_\ell}\in L_{i_\ell}$ for all indices $\ell$, and $v_{i_j}\in L_{i_{j}}^\ast$ for at least one index $j$ we have 
\[\eta([v_{i_1},\ldots,v_{i_\ell}]) = 0.\]

To prove this, note that $Q_{i_j}v_{i_j} = v_{i_j}$ and $P_{i_j}v_{i_j} = 0$ and using the last bullet point of Lemma~\ref{lem:sunflower}, we have 
\[\eta([v_{i_1},\ldots,v_{i_\ell}]) = \eta([P_{i_1}v_{i_1},\ldots,Q_{i_j}v_{i_j},\ldots,P_{i_\ell}v_{i_\ell}]),\]
similar to the argument in Step 3.

Next note that $P_iQ_iv = 0$ for all $v\in V_i$ and therefore 
\[P_i(L_i) \leqslant P_i(\pi_1(\wt{K_i}\cap\{(v_1,v_2,v_3,v_4)\in V_i^{\otimes 4}\colon Q_iv_2 = Q_iv_3 = Q_iv_4 = 0\})).\]
Therefore we may lift $P_{i_\ell}v_{i_\ell}$ for $\ell\neq j$ to $\wt{v_{i_\ell}} = (P_{i_\ell}v_{i_\ell} + w_{i_\ell},P_{i_\ell}v_{i_\ell},P_{i_\ell}v_{i_\ell},P_{i_\ell}v_{i_\ell})\in\wt{K_{i_\ell}}$ where $w_{i_\ell}\in Q_{i_\ell}(V_{i_\ell})$. We lift $v_{i_j}$ to $\wt{v_{i_{j}}}$ which has the form $(Q_{i_j}v_{i_j},0,0,0)\in\wt{K_{i_j}}$.

Note that we have 
\begin{align*}
 0 &= \eta_{\mr{Prod}}([\wt{v_{i_{1}}},\ldots,\wt{v_{i_{r}}}]) = \eta([P_{i_{1}}v_{i_{1}} + w_{i_{1}},\ldots, Q_{i_j}v_{i_{j}},\ldots, P_{i_{r}}v_{i_{r}} + w_{i_{r}}])\\
&= \eta([P_{i_{1}}v_{i_{1}},\ldots, Q_{i_j}v_{i_{j}},\ldots, P_{i_{r}}v_{i_{r}}])
\end{align*}
where in the first equality we have used for all $\ell$ that $\wt{v_{i_\ell}}\in\wt{K_{i_\ell}}$ and the result from Step 3, in the second equality that $\wt{v_{i_{j}}}$ has the final three coordinates identically zero, and in the final equality that $w_{i_\ell}\in Q_{i_\ell}(V_{i_\ell})=V_{i_\ell,\mr{Dep}}$ and the final item of Lemma~\ref{lem:sunflower}. 

The desired decomposition of spaces for the lemma will have 
\begin{align*}
W_{i,\mr{Pet}} &:= L_i^\ast,\quad W_{i,\ast} := P_i(L_i)\leqslant L_i\cap P_i(V_i).
\end{align*}
The fact $P_i(L_i)\leqslant L_i\cap P_i(V_i)$ is deduced from $Q_i(V_i)\leqslant L_i$. $W_{i,\mr{Lin}}$ will be constructed explicitly in the next step but is chosen so that
\[W_{i,\mr{Lin}}\leqslant Q_i(L_i) = Q_i(V_i)=V_{i,\mr{Dep}}\]
and $W_{i,\mr{Lin}} + W_{i,\mr{Pet}} = V_{i,\mr{Dep}}\leqslant L_i$. Given these properties of $W_{i,\mr{Lin}}$, note that the above analysis, along with Lemma~\ref{lem:sunflower}, establishes the final bullet point for our output.

\noindent\textbf{Step 5: Controlling approximate homomorphisms.}
Recall $\wt{K_i}\leqslant S_i$ and there is a natural isomorphism of groups
\[S_i \simeq \{(v,v_1,v_2,v_3,v_4)\colon v\in P_i(V_i), v_1,\ldots,v_4\in Q_i(V_i)\}.\]
Using this as an identification, we may write 
\[\wt{K_i} = \bigcap_{j=1}^{\dim(S_i) - \dim(\wt{K_i})}\on{ker}((\xi_j^{P_i}, \xi_j^{Q_i},\xi_j^{Q_i},-\xi_j^{Q_i},-\xi_j^{Q_i}))\]
where $\xi_j^{P_i}\in P_i(V_i)^\vee$ and $\xi_j^{Q_i}\in Q_i(V_i)^\vee$ (i.e., corresponding dual vector spaces). Note that the annihilators all have the special form of $(\cdot, \xi_j^{Q_i},\xi_j^{Q_i},-\xi_j^{Q_i},-\xi_j^{Q_i})$ since
\[\{(Q_iv_1,Q_iv_2,Q_iv_3,Q_iv_4)\in V_i^{\otimes 4}\colon Q_iv_1 + Q_iv_2 - Q_iv_3 -Q_iv_4 = 0\}=R_i\leqslant\wt{K_i}.\]
Note that
\[L_i^\ast = \{v\in V_i\colon P_i v = 0\text{ and }\xi_j^{Q_i}(Q_iv) = 0\text{ for all }j\}\]
since $v\in L_i^\ast$ is equivalent under this identification to $(0,Q_iv,0,0,0)\in\wt{K_i}$. Without loss of generality we may assume that for $1\le j\le\dim(V_{i,\mr{Dep}}) - \dim(L_i^\ast)$, vectors $\xi_j^{Q_i}$ are independent in $Q_i(V_i)^\vee$ (and they must span the orthogonal space to $L_i^\ast$ within $Q_i(V_i)^\vee$).

By appropriate scaling, we may assume $\xi_j^{Q_i}(w_{i,j})$ is an integer bounded by $M(\delta)$ for $1\le j\le\dim(V_{i,\mr{Dep}})$. We extend each $\xi_j^{Q_i}$ to an operator on $(G_{(i,1)}/G_{(i,2)})^\vee$ by setting $\xi_j^{Q_i}(w_{i,j}) = 0$ for $j>\dim(V_{i,\mr{Dep}})$. Possibly at the cost of another $M(\delta)$ scaling, we may assume that $\xi_j^{Q_i}(\Gamma\cap G_{(i,1)}\imod G_{(i,2)})\in\mb{Z}$. We extend $\xi_j^{P_i}(\cdot)$ in an analogous manner to $(G_{(i,1)}/G_{(i,2)})^\vee$ by setting $\xi_j^{P_i}(w_{i,j}) = 0$ for $1\le j\le\dim(V_i,\mr{Dep})$ and $j>\dim(V_i)$. Again, we may scale such that $\xi_j^{P_i}(\Gamma\cap G_{(i,1)}\imod G_{(i,2)})\in\mb{Z}$. The crucial point here is that now $\xi_j^{P_i}$ and $\xi_j^{Q_i}$ are $i$-th horizontal characters of height at most $M(\delta)$.

At this point, we now let the parameter $\vec{i}$ be a vector in $\mathbb{Z}^{D'}$. We have
\[\tau_{|\vec{i}|}(\on{Taylor}_{\vec{i}}(\wt{g}_{\vec{h}})) \in K_{|\vec{i}|}\]
and thus
\[\on{dist}(\tau_{|\vec{i}|}(\on{Taylor}_{\vec{i}}(g_{\vec{h}}^\ast)), S_{|\vec{i}|} + T^{-1}\on{Horiz}_{|\vec{i}|}(\Gamma^{\otimes 4}))\le M(\delta) N^{-|\vec{i}|}\]
after Pigeonholing $\vec{h}$ appropriately. Here distance is in $L^{\infty}$ after expressing both of these expressions in the basis $\exp(\mc{X}_{|\vec{i}|})^{\otimes 4}$ and $T$ is an integer bounded by $M(\delta)$.

Furthermore note that 
\[\tau_{|\vec{i}|}(\on{Taylor}_{\vec{i}}(g_{\vec{h}}^\ast))\in S_{|\vec{i}|}\]
by Lemma~\ref{lem:sunflower}. So if we choose a set of horizontal characters of height $M(\delta)$ relative to $\on{Horiz}_{|\vec{i}|}(\Gamma^{\otimes 4})$ which cut out $\wt{K_{|\vec{i}|}}$ as their common kernel, then noting that $K_{|\vec{i}|}\cap S_{|\vec{i}|} \leqslant \wt{K_{|\vec{i}|}}$ and applying Lemma~\ref{lem:factor}, we may assume that 
\begin{equation}\label{eq:linear-output-1}
\tau_{|\vec{i}|}(\on{Taylor}_{\vec{i}}(\wt{g}_{\vec{h}}))\in\wt{K_{|\vec{i}|}}
\end{equation}
and $\eps_{\vec{h}},\gamma_{\vec{h}}$ have identical properties up to changing implicit constants. We will assume this refined property of the factorization for the remainder of our analysis. 

Given the factorization of $g_{\vec{h}}^\ast$, we thus deduce (taking an appropriate least common multiple) that
\begin{align}
\norm{T_1\cdot\bigg(\xi_j^{P_{|\vec{i}|}}&(\on{Taylor}_{\vec{i}}(g_{h_2})) + \xi_j^{Q_{|\vec{i}|}}(\on{Taylor}_{\vec{i}}(g_{h_1})) + \xi_j^{Q_{|\vec{i}|}}(\on{Taylor}_{\vec{i}}(g_{h_2}))\notag\\
&\quad- \xi_j^{Q_{|\vec{i}|}}(\on{Taylor}_{\vec{i}}(g_{h_3})) - \xi_j^{Q_{|\vec{i}|}}(\on{Taylor}_{\vec{i}}(g_{h_4}))\bigg)}_{\mb{R}/\mb{Z}}\le M(\delta) N^{-|\vec{i}|}\label{eq:linear-output-2}
\end{align}
for all $1\le j\le\dim(V_{|\vec{i}|,\mr{Dep}})-\dim(L_{|\vec{i}|}^\ast)$ where $T_1$ is an integer bounded by $M(\delta)$. Here we have used that $\xi_j^{P_{|\vec{i}|}}(\on{Taylor}_{\vec{i}}(g_h))$ is equal for all $h\in H$ by Lemma~\ref{lem:sunflower}.

We define functions $f,g\colon H\to\mb{R}^{\sum_{|\vec{i}| < d - 1}\dim(V_{|\vec{i}|,\mr{Dep}})-\dim(L_{|\vec{i}|}^\ast)}$ via
\begin{align*}
f(h) &= (T_1\xi_j^{Q_{|\vec{i}|}}(\on{Taylor}_{\vec{i}}(g_h)))_{|\vec{i}| < d - 1, 1\le j\le\dim(V_{|\vec{i}|,\mr{Dep}})-\dim(L_{|\vec{i}|}^\ast)},\\
g(h) &= (T_1\xi_j^{P_{|\vec{i}|}}(\on{Taylor}_{\vec{i}}(g_h)) + T_1\xi_j^{Q_{|\vec{i}|}}(\on{Taylor}_{\vec{i}}(g_h)))_{1\le |\vec{i}| \le d-1,~1\le j\le\dim(V_{|\vec{i}|,\mr{Dep}})-\dim(L_{|\vec{i}|}^\ast)}.
\end{align*}
Note that for the additive quadruples on which we have \eqref{eq:linear-output-2}, we are exactly in the situation necessary to apply results on approximate homomorphisms. 

In particular, we may apply Lemma~\ref{lem:approximate}. We see that there exists $H'\subseteq H$ having density at least $\epsilon(\delta)$ such that for all $\vec{i},j$ and $h\in H'$, we have
\begin{equation}\label{eq:linear-output-3}
\norm{T_1\xi_j^{Q_{|\vec{i}|}}(\on{Taylor}_{\vec{i}}(g_h)) - \bigg(\gamma_{\vec{i},j} + \sum_{k=1}^{d^\ast}\alpha_{\vec{i},j,k} \{\beta_k h\}\bigg)}_{\mb{R}/\mb{Z}}\le M(\delta) N^{-|\vec{i}|},
\end{equation}
where:
\begin{itemize}
    \item $d^\ast\le m(\delta)$;
    \item $\beta_k\in(1/N')\mb{Z}$ where $N'$ is a prime between $100N$ and $200N$.
\end{itemize} 

At this point, for each $\vec{i}$ we find elements $Z_{|\vec{i}|,j}$ for $1\le j\le\dim(V_{|\vec{i}|,\mr{Dep}})-\dim(L_{|\vec{i}|}^\ast)$ which are $M(\delta)$-rational combinations of $\{w_{|\vec{i}|,j}\colon 1\le j\le\dim(V_{|\vec{i}|,\mr{Dep}})\}$ such that 
\begin{equation}\label{eq:basis-def}
T_1\xi_j^{Q_{|\vec{i}|}}(Z_{|\vec{i}|,j}) = 1\text{ and }\xi_{j'}^{Q_{|\vec{i}|}}(Z_{|\vec{i}|,j}) = 0
\end{equation}
for $j'\neq j$ such that $1\le j,j'\le\dim(V_{|\vec{i}|,\mr{Dep}})-\dim(L_{|\vec{i}|}^\ast)$. We define
\[W_{|\vec{i}|,\mr{Lin}} = \on{span}_{\mb{R}}((Z_{|\vec{i}|,j})_{1\le j\le\dim(V_{|\vec{i}|,\mr{Dep}})-\dim(L_{|\vec{i}|}^\ast)}).\]

We see that there are no nontrivial linear relations between $W_{|\vec{i}|,\ast}$ and $W_{|\vec{i}|,\mr{Pet}}+W_{|\vec{i}|,\mr{Lin}}$ since $W_{|\vec{i}|,\ast}\leqslant P_{|\vec{i}|}(V_{|\vec{i}|})$ and $W_{|\vec{i}|,\mr{Pet}}+W_{|\vec{i}|,\mr{Lin}}\leqslant Q_{|\vec{i}|}(V_{|\vec{i}|})$. There are no linear relations between $W_{|\vec{i}|,\mr{Pet}}$ and $W_{|\vec{i}|,\mr{Lin}}$ as $W_{|\vec{i}|,\mr{Pet}}$ lies in the joint kernel of the $\xi_j^{Q_{|\vec{i}|}}$ and therefore using \eqref{eq:basis-def} one can prove any such relation is trivial. Furthermore, by construction we have $V_{|\vec{i}|,\mr{Dep}} = Q_{|\vec{i}|}(V_{|\vec{i}|}) = W_{|\vec{i}|,\mr{Lin}} + W_{|\vec{i}|,\mr{Pet}}$. Finally $L_{|\vec{i}|} = P_{|\vec{i}|}(L_{|\vec{i}|}) + Q_{|\vec{i}|}(L_{|\vec{i}|})  = W_{|\vec{i}|,\ast} + W_{|\vec{i}|,\mr{Lin}} + W_{|\vec{i}|,\mr{Pet}}$; this implicitly uses $Q(L_{|\vec{i}|}) = Q_{|\vec{i}|}(V_{|\vec{i}|}) \leqslant L_{|\vec{i}|}$.

\noindent\textbf{Step 6: Constructing the desired factorizations and completing the proof.}
Using the refined factorization \eqref{eq:linear-output-1} implies that
\[\pi_1(\tau(\on{Taylor}_{\vec{i}}(\wt{g}_{\vec{h}}))) \in L_{|\vec{i}|}\]
since $\pi_1(\wt{K_{|\vec{i}|}}) \leqslant L_{|\vec{i}|}$. Applying $g_{\vec{h}}^\ast = \eps_{\vec{h}}\cdot\wt{g_{\vec{h}}}\cdot\gamma_{\vec{h}}$ in the first coordinate then implies that
\begin{equation}\label{eq:linear-output-5}
\on{dist}(\on{Taylor}_{\vec{i}}(g_{h_1}), L_{|\vec{i}|} + T_2^{-1}\on{Horiz}_i(\Gamma))\le M(\delta) N^{-|\vec{i}|}
\end{equation}
for $h_1\in H'$ where distance is in $L^\infty$ after expressing values in terms of $\exp(\mc{X}_i)$. Here $T_2$ is an integer bounded by $M(\delta)$. Furthermore recall from Lemma~\ref{lem:sunflower} that
\begin{equation}\label{eq:linear-output-6}
\on{Taylor}_{\vec{i}}(g_h)-\on{Taylor}_{\vec{i}}(g_{h'})\in V_{|\vec{i}|,\mr{Dep}}=Q_{|\vec{i}|}(V_{|\vec{i}|})
\end{equation}
for $h,h'\in H'\subseteq H$.

Let $Y_{|\vec{i}|,j}\in\on{span}_{\mb{R}}(\mc{X}\cap\log(G_{(|\vec{i}|,1)})\setminus \mc{X}\cap\log(G_{(|\vec{i}|,2)}))$ be such that $\exp(Y_{|\vec{i}|,j}) \imod G_{(|\vec{i}|,2)}= Z_{|\vec{i}|,j}$. Then for $h\in H'$, we define (with the product taken in, say, lexicographic order)
\begin{equation}\label{eq:linear-output-7}
\wt{g}_h(n) = \prod_{|\vec{i}| \le d}\prod_{j=1}^{\dim(W_{|\vec{i}|,\mr{Lin}})}\exp(Y_{|\vec{i}|,j})^{T_1^{-1}\binom{n}{\vec{i}}\cdot (\gamma_{\vec{i},j}+\sum_{k=1}^{d^\ast}\alpha_{\vec{i},j,k}\{\beta_k h\})}.
\end{equation}
By construction and Lemma~\ref{lem:taylorexansion}, for $h,h'\in H'$ we have 
\begin{align*}
\on{Taylor}_{\vec{i}}(\wt{g}_h^{-1}g_h) - \on{Taylor}_{\vec{i}}(\wt{g}_{h'}^{-1}g_{h'}) & \in Q_{|\vec{i}|}(V_{|\vec{i}|}),\\
\on{dist}(\on{Taylor}_{|\vec{i}|}(\wt{g}_h^{-1}g_h), L_{|\vec{i}|} + T_2^{-1} \on{Horiz}_{|\vec{i}|}(\Gamma))&\le M(\delta) \cdot N^{-|\vec{i}|},\\
\snorm{T_1\xi_j^{Q_{|\vec{i}|}}(\wt{g}_h^{-1}g_h)}_{\mb{R}/\mb{Z}}&\le M(\delta) \cdot N^{-|\vec{i}|},
\end{align*}
where $1\le |\vec{i}|\le d-1$ and $1\le j\le\dim(V_{|\vec{i}|,\mr{Dep}})-\dim(L_{|\vec{i}|}^\ast)$. The first line comes from \eqref{eq:linear-output-6}, the second line from \eqref{eq:linear-output-5}, and the third from \eqref{eq:linear-output-3} and \eqref{eq:linear-output-7}, in conjunction with \eqref{eq:basis-def}.

We now fix an element $h_2\in H'$. For each $h_1\in H'$ we write 
\begin{align*}
g_{h_1}' &= \wt{g}_{h_1} \cdot (\wt{g}_{h_1}^{-1} g_{h_1}')= \wt{g}_{h_1} \cdot (\wt{g}_{h_1}^{-1} g_{h_1}') \cdot (\wt{g}_{h_2}^{-1} g_{h_2}')^{-1} \cdot (\wt{g}_{h_2}^{-1} g_{h_2}').
\end{align*}
By applying Lemma~\ref{lem:factor}, we may write 
\[(\wt{g}_{h_1}^{-1} g_{h_1}') \cdot (\wt{g}_{h_2}^{-1} g_{h_2}')^{-1} = \eps_{h_1}^\ast g_{h_1}^\ast \gamma_{h_1}^\ast,\qquad(\wt{g}_{h_2}^{-1} g_{h_2}') = \eps^\ast g^\ast \gamma^\ast\]
where $\gamma^{\ast},\gamma_{h_1}^\ast$ are $M(\delta)$-rational, $\eps^{\ast},\eps_{h_1}^\ast$ are $(M(\delta),\vec{N})$-smooth, and we have $\on{Taylor}_{\vec{i}}(g_{h_1}^\ast) \in L_{|\vec{i}|}^\ast = W_{|\vec{i}|,\mr{Pet}}$ using the first and third lines above and $\on{Taylor}_{\vec{i}}(g^\ast) \in L_{|\vec{i}|}^\ast+P_{|\vec{i}|}(L_{|\vec{i}|})=W_{|\vec{i}|,\ast} + W_{|\vec{i}|,\mr{Pet}}$ using the second and third lines above. (Recall that $L_{|\vec{i}|}^\ast\leqslant Q_{|\vec{i}|}(V_{|\vec{i}|})$ is cut out by the $\xi_j^{Q_{|\vec{i}|}}$.) Additionally, these sequences are the identity at $0$.

Therefore, for $h_1\in H'$ we have
\begin{align*}
g_{h_1}' &= \eps_{h_1}^\ast \eps^\ast ((\eps_{h_1}^\ast \eps^\ast)^{-1}\wt{g}_{h_1}(\eps_{h_1}^\ast \eps^\ast))((\eps^\ast)^{-1}g_{h_1}^\ast\eps^\ast)((\eps^\ast)^{-1}\gamma_{h_1}^\ast\eps^\ast(\gamma_{h_1}^\ast)^{-1})(\gamma_{h_1}^{\ast}g^\ast(\gamma_{h_1}^\ast)^{-1})(\gamma_{h_1}^\ast\gamma^\ast)\\
&=:(\eps_{h_1}^\ast\eps^\ast) \cdot g_{h_1}^{\triangle} \cdot (\gamma_{h_1}^\ast\gamma^\ast).
\end{align*}
So, for $h_3,h_4\in H$ we deduce using Lemma~\ref{lem:taylorexansion} and the above analysis that
\begin{align*}
\on{Taylor}_{\vec{i}}(\wt{g}_{h_3}) &\in W_{|\vec{i}|,\mr{Lin}},\\
\on{Taylor}_{\vec{i}}(g_{h_3}^{\triangle}) &\in L_{|\vec{i}|} =  W_{|\vec{i}|,\ast} + W_{|\vec{i}|,\mr{Lin}}  + W_{|\vec{i}|,\mr{Pet}},\\
\on{Proj}_{W_{|\vec{i}|},\mr{Lin}}(\on{Taylor}_{\vec{i}}(g_{h_3}^{\triangle})) &= \on{Proj}_{W_{|\vec{i}|},\mr{Lin}}(\on{Taylor}_{|\vec{i}|}(\wt{g}_{h_3})),\\
\on{Taylor}_{\vec{i}}(\wt{g}_{h_3}^{-1}g_{h_3}^{\triangle}) - \on{Taylor}_{\vec{i}}(\wt{g}_{h_4}^{-1}g_{h_4}^{\triangle}) &\in W_{|\vec{i}|,\mr{Pet}}.
\end{align*}
Furthermore note that $\eps_{h_1}^\ast \eps^\ast$ is sufficiently smooth and $\gamma_{h_1}^\ast\gamma^\ast$ is appropriately rational. This proves the lemma provided we handle the exceptional case of $d = 2$ which we will next step.

\noindent\textbf{Step 7: Handling the exceptional case $d = 2$.}
In this exceptional case, we have $r = 1$ and $d=2$, and $\eta$ is nontrivial. In this case, we have $\psi_h \equiv 1$, by projecting to the one-dimensional subgroup generated by $\eta$, we may assume that the nilmanifold underlying $\chi_h$ is one-dimensional and $\chi_h = e(\vec{\xi_h} \cdot)$. By pigeonholing in a progression and applying the hypothesis, we have that for some integer $T \le M(\delta)$, and a proportion $\epsilon(\delta)$ of additive quadruples $(h_1, h_2, h_3, h_4)$ in $H$ that for each component $i$
$$\|T(\xi_{h_1}^i + \xi_{h_2}^i - \xi_{h_3}^i - \xi_{h_4}^i)\|_{\mathbb{R}/\mathbb{Z}} \le M(\delta)N^{-1}.$$
Applying Lemma~\ref{lem:approximate}, there exists a subset $H' \subseteq H$ of size at least $\epsilon(\delta)|H|$ such that
$$\xi_h^i = \gamma + \sum_{k = 1}^{d^*} \alpha_{i, k}\{\beta_{i, k} \cdot h\} + M(\delta)^{-1}\mathbb{Z}$$
where
\begin{itemize}
    \item $d^\ast\le m(\delta)$;
    \item $\beta_{i, k}\in(1/N')\mb{Z}$ where $N'$ is a prime between $100N$ and $200N$.
\end{itemize}
We can then take $W_{1, \mathrm{Lin}}$ to be the one-dimensional subspace generated by $\eta$, and $W_{1, *}$ the orthogonal complement, and $W_{1, \mathrm{Pet}}$ to be the zero subspace. Each of these spaces enjoy $M(\delta)$-rational bases via \cite[Lemma B.11]{Len23b} and dividing $\gamma$ by $\alpha_{i, k}$ by an integer bounded by $M(\delta)$, we have achieved the conclusion of the lemma.
\end{proof}
\subsection{Constructing a higher step nilsequence from bracket linear coefficients}
We are now ready to prove Proposition~\ref{prop:degreerankextraction}.
\begin{proof}[Proof of Proposition~\ref{prop:degreerankextraction}.]
The proof is again very similar to that of \cite[Lemma 6.3]{LSS24b}, but as we work with multidimensional (input) nilsequences, we must make some modifications. Unfortunately, the nature of the modifications we make are such that they require extensive justification; hence, our proof is still rather lengthy.

We begin with the following lemma which is analogous to \cite[Lemma 10.2]{LSS24b}.
\begin{lemma}\label{lem:Taylor-mod}
Let $G/\Gamma$ be a nilmanifold with filtration degree $k$ and consider the adapted Malcev basis $\mathcal{X} = \{X_1, \dots, X_m\}$. Let $g: \mathbb{Z}^D \to G$ be a polynomial sequence. Then $g$ admits a representation of the form
$$g(\vec{n}) = \prod_{i = 0}^k \prod_{x_1 + \cdots + x_D = i} \prod_{j = m - \mathrm{dim}(G_i) + 1}^m \exp(X_j)^{\alpha_{x, j} \cdot \frac{n^x}{x!}}$$
(with the product along $x_1 + \cdots + x_D = i$ taken in any order with the coefficients $\alpha_{x, j}$ depending on the order in which we take the product).
\end{lemma}
\begin{proof}
By Baker-Campbell Hausdorff and Taylor expansion, we may write
$$g(n) = \exp\left(\sum_{i = 0}^k \sum_{x_1+ \cdots + x_D = i} g_{x} \cdot \frac{n^x}{x!}\right)$$
with $g_x \in \log(G_{|x|})$. We now take $g_0(n) := g(n)$, $g_{0, j} := g_j$, and iteratively define a sequence $g_\ell(n)$ by the following process: for $|x| = \ell$, write $g_{x, \ell} = \sum_{j = m - \mathrm{dim}(G_\ell) + 1} \alpha_{x, j}X_j$. Then let (for the outer product in any order)
$$g_{\ell + 1}(n) := \left(\prod_{x_1 + \cdots + x_D = \ell} \prod_{j = m - \mathrm{dim}(G_\ell) + 1}^m  \exp(X_j)^{\alpha_{x, j} \cdot \frac{n^x}{x!}}\right)^{-1}g_\ell(n).$$
By Baker-Campbell-Hausdorff, we may write for some $g_{x, \ell + 1}$
$$g_{\ell + 1}(n) = \exp\left(\sum_{i = \ell + 1}^k \sum_{x_1 + \cdots + x_D = i} g_{x, \ell + 1} \cdot \frac{n^x}{x!}\right)$$
from which we complete the inductive definition of $g_\ell$. Note that $\alpha_{x, j}$ exists and is well-defined since $\mathcal{X}$ is a nesting basis. In addition, $g_{x, \ell + 1}$ exists and is well-defined since inductively and the nesting property we see that $g_{\ell + 1}$ has no Taylor coefficients of degree $0$ to $\ell$. This completes the proof.
\end{proof}
We will now indicate modifications one must make for the proof of Proposition~\ref{prop:degreerankextraction}. Let $H^*$ denote the set of $h$'s from the linearization step. Again, we deal with multidimensional input nilcharacters rather than single dimensional input nilcharacters. We must take care that the Taylor expansions we work with are multidimensional. Thus, our Taylor coefficients are slightly different. Instead of \cite[Page 52 bullet point 8]{LSS24b}, we work with
 \begin{align*}
    \on{Taylor}_{\vec{i}}(g_h) &= \prod_{j=1}^{\dim(W_{i,\ast})}\exp(Z_{|\vec{i}|,j}^\ast)^{z_{x,j}^\ast}\cdot\prod_{j=1}^{\dim(W_{|\vec{i}|,\mr{Pet}})} \exp(Z_{|\vec{i}|,j}^{\mr{Pet}})^{z_{\vec{i},j}^{h,\mr{Pet}}}\\
    &\qquad \qquad\cdot\prod_{j=1}^{\dim(W_{|\vec{i}|,\mr{Lin}})} \exp(Z_{|\vec{i}|,j}^{\mr{Lin}})^{z_{\vec{i},j}^{h,\mr{Lin}}} \imod G_{(|\vec{i}|,2)}
    \end{align*}
    where 
    \[z_{\vec{i},j}^{h,\mr{Lin}} = \gamma_{\vec{i},j} + \sum_{k=1}^{d^\ast}\alpha_{\vec{i},j,k}\{\beta_k \cdot h\}\]
    where $d^\ast\le m(\delta)$ and $\beta_k\in(1/N')\mb{Z}^{D}$ where $N'$ is a prime in $[100N,200N]$.

    Also, we work with the following expansions of $g_h^*$, $g_h^\mathrm{Lin}$, and $g_h^{\mathrm{Pet}}$. 
\begin{align*}
g_h^\ast &= \prod_{i \in [d - 1]} \prod_{x_1 + \cdots + x_D =  i}\prod_{j=1}^{\dim(W_{i,\ast})}\exp(z_{x,j}^\ast)^{z_{x,j}^\ast\cdot\frac{n^x}{x!}} \cdot\prod_{i \in [d - 1]}\prod_{x_1 + \cdots  + x_{D} = i} \prod_{j=1}^{\dim(W_{i,\mr{Lin}})}\exp(Z_{i,j}^{\mr{Lin}})^{\gamma_{x,j}\cdot\frac{n^x}{x!}},\\
g_h^{\mr{Lin}} &= \prod_{i \in [d - 1]} \prod_{x_1 + \cdots + x_D =  i}\prod_{j=1}^{\dim(W_{i,\mr{Lin}})}\exp(Z_{i,j}^{\mr{Lin}})^{(z_{x,j}^{h,\mr{Lin}}-\gamma_{x,j})\cdot\frac{n^x}{x!}},\\
g_h^{\mr{Pet}} &= \prod_{i \in [d - 1]} \prod_{x_1 + \cdots + x_D =  i}\prod_{j=1}^{\dim(W_{i,\mr{Pet}})}\exp(Z_{i,j}^{\mr{Pet}})^{z_{x,j}^{h,\mr{Pet}}\cdot\frac{n^x}{x!}}
\end{align*}
We similarly define $g_h^{\mathrm{Rem}}$ via $g_h = g_h^* \cdot g_h^{\mathrm{Lin}} \cdot g_h^{\mathrm{Pet}} \cdot g_h^{\mathrm{Rem}}$. Thus, by Lemma~\ref{lem:Taylor-mod}, we have
\[g_h^{\mr{Rem}} = \prod_{i \in [d - 1]} \prod_{x_1 + \cdots + x_D =  i}\prod_{j=\dim(G)-\dim(G_{i,1})+1}^{\dim(G)}\exp(X_j)^{\kappa_{x,j}^{h}\cdot\frac{n^x}{x!}}.\]
We now lift to the universal nilmanifold. We require the following definition, which is \cite[Defintion 10.3]{LSS24b}.
\begin{definition}\label{def:universal-1}
The \textbf{universal nilmanifold of degree-rank $(s-1,r^\ast)$} and the associated discrete cocompact subgroup are defined as follows. We write $G_{\mr{Univ}} = G_{\mr{Univ}}^{\vec{D}}$ where $\vec{D} = \vec{D}^{\ast}+\vec{D}^{\mr{Lin}}+\vec{D}^{\mr{Pet}}$ with $\vec{D}^{\ast},\vec{D}^{\mr{Lin}},\vec{D}^{\mr{Pet}}\in (\mb{Z}_{\ge 0})^{s-1}$. We specify $G_{\mr{Univ}}^{\vec{D}}$ by formal generators of the Lie algebra
$e_{i,j}$ for $1\le i\le s-1$ and $1\le j\le D_i$ where $D_i = D_i^\ast + D_i^{\mr{Lin}} + D_i^{\mr{Pet}}$ with the relations:
\begin{itemize}
    \item Any $(r-1)$-fold commutator of $e_{i_1,j_1},\ldots, e_{i_r,j_r}$ with $i_1 + \cdots + i_r > (s-1)$ vanishes;
    \item Any $(r-1)$-fold commutator of $e_{i_1,j_1},\ldots, e_{i_r,j_r}$ with $i_1 + \cdots + i_r = (s-1)$ and $r>r^\ast$ vanishes.
\end{itemize}
The associated discrete group which we will be concerned with is $\Gamma_{\mr{Univ}}$ which is the discrete group generated by $\exp(e_{i,j})$ for $1\le i\le s-1$ and $1\le j\le D_i$.
\end{definition}
Also, when we lift to the universal nilmanifold, we will define
\begin{align*}
g_h^{\ast,\mr{Univ}}(n) &= \prod_{i \in [d - 1]} \prod_{x_1 + \cdots + x_D =  i} \prod_{j=1}^{\dim(W_{i,\ast})} \exp(e_{i,j})^{z_{x,j}^\ast\cdot\frac{n^x}{x!}}  \prod_{i \in [d - 1]} \prod_{\alpha_1 + \cdots+ \alpha_D = i} \prod_{j=1}^{\dim(W_{i,\mr{Lin}})} \exp(e_{i,j +  \dim(W_{i,\ast})})^{\gamma_{x,j}\cdot\frac{n^{\alpha}}{x!}},\\
g_h^{\mr{Lin},\mr{Univ}}(n)  &= \prod_{i \in [d - 1]} \prod_{x_1 + \cdots + x_D =  i} \prod_{j=1}^{\dim(W_{i,\mr{Lin}})} \prod_{k=1}^{d^\ast}\exp(e_{i,D_i^\ast + (j-1)d^\ast+k})^{\alpha_{x,j,k}\{\beta_k \cdot h\}\cdot\frac{n^x}{x!}},\\
g_h^{\mr{Pet},\mr{Univ}}(n) &=\prod_{i \in [d - 1]} \prod_{j=1}^{\dim(W_{i,\mr{Pet}})} \exp(e_{i,j + D_i^\ast + D_i^{\mr{Lin}}})^{z_{x,j}^{h,\mr{Pet}}\cdot\frac{n^x}{x!}},\\
g_h^{\mr{Rem},\mr{Univ}}(n)  &= \prod_{i \in [d - 1]} \prod_{x_1 + \cdots + x_D =  i}\prod_{j=1}^{\dim(G_{(i,2)})}\exp(e_{i,j + D_i - \dim(G_{(i,2)})})^{\kappa_{i,j}^h\cdot\frac{n^x}{x!}}
\end{align*}
and 
\[g_h^{\mr{Univ}} := g_h^{\ast,\mr{Univ}} \cdot g_h^{\mr{Lin},\mr{Univ}} \cdot g_h^{\mr{Pet},\mr{Univ}} \cdot g_h^{\mr{Rem},\mr{Univ}}.\]
We define a homomorphism $\phi\colon G_{\mr{Univ}}\to G$ by defining the map on generators. Define
\begin{align*}
\phi(\exp(e_{i,j})) = 
\begin{cases}
\exp(z_{x,j}^\ast) &\text{ if } 1\le j\le\dim(W_{i,\ast}),\\
\exp(Z_{i,j - \dim(W_{i,\ast})}^{\mr{Lin}}) &\text{ if } \dim(W_{i,\ast}) + 1\le j\le\dim(W_{i,\ast}) + \dim(W_{i,\mr{Lin}})=D_i^\ast,\\
\exp(Z_{i,\ell}^{\mr{Lin}}) &\text{ if } 1 + (\ell-1)d^\ast\le j - D_i^\ast\le\ell d^\ast\text{ for }1\le\ell\le\dim(W_{i,\mr{Lin}}),\\
\exp(Z_{i,j - D_i^\ast-D_i^{\mr{Lin}}}^{\mr{Pet}}) &\text{ if } D_i^{\ast}+D_i^{\mr{Lin}} + 1\le j\le D_i^{\ast}+D_i^{\mr{Lin}} + \dim(W_{i,\mr{Pet}}),\\
\exp(X_{j - D_i + \dim(G)}) &\text{ if } D_i - \dim(G_{(i,2)}) + 1\le j\le D_i.
\end{cases}
\end{align*}
We now define
\[\wt{F}(g\Gamma_{\mr{Univ}}) := F(\phi(g)\Gamma).\]
Similar to \cite[Claim 10.6]{LSS24b}, we have the following. (Note that $\exp(e_{i, j})$ often does not get sent to $Z_{x, j}$.)
\begin{claim}\label{clm:univ-lift}
Given the above setup, we have 
\[\wt{F}(g_h^{\mr{Univ}}(n)\Gamma_{\mr{Univ}}) = F(g_h(n)\Gamma) = \chi_h(n).\]
\end{claim}
We next define the quotient universal nilmanifold in which we lift to.
\begin{definition}\label{def:univer-2}
We define $G_{\mr{Rel}} = G_{\mr{Rel}}^{\vec{D}^{\ast},\vec{D}^{\mr{Lin}},\vec{D}^{\mr{Pet}}}$ as the Lie subgroup of $G_{\mr{Univ}}$ where $\log(G_{\mr{Rel}})$ is spanned by:
\begin{itemize}
    \item Any $(r-1)$-fold commutator of $e_{i_1,j_1},\ldots, e_{i_r,j_r}$ with at least one index $\ell$ such that $j_\ell>D^\ast_{i_\ell} + D^{\mr{Lin}}_{i_\ell}$;
    \item Any $(r-1)$-fold commutator of $e_{i_1,j_1},\ldots, e_{i_r,j_r}$ with $j_\ell>D^\ast_{i_\ell}$ for at least two distinct indices $\ell$.
\end{itemize}
We then define $G_{\mr{Quot}}$ as $G_{\mr{Quot}}:= G_{\mr{Univ}}/G_{\mr{Rel}}$ and $\Gamma_{\mr{Quot}} = \Gamma_{\mr{Univ}}/(\Gamma_{\mr{Univ}}\cap G_{\mr{Rel}})$.
\end{definition}
As in \cite[Lemma 10.10, Lemma 10.11]{LSS24b}, we have the following.
\begin{lemma}\label{lem:universal-complexity-2}
Given the above setup, let $G_{\mr{Quot}} = G_{\mr{Quot}}^{\vec{D}}$ and note that $G_{\mr{Quot}}$ has a degree-rank $(d-1,r)$ filtration given by 
\[(G_{\mr{Quot}})_{(d',r')} = (G_{\mr{Univ}})_{(d',r')}/((G_{\mr{Univ}})_{(d',r')}\cap G_{\mr{Rel}}).\]
Furthermore the dimension of $G_{\mr{Univ}}$ is bounded by $O_s(\snorm{D}_{\infty}^{O_K(1)})$ and one may find an adapted Mal'cev basis $\mc{X}_{\mr{Quot}}$ such that the complexity of $G_{\mr{Quot}}/\Gamma_{\mr{Quot}}$ is $\exp(\snorm{D}_{\infty}^{O_K(1)})$.
\end{lemma}
We will define
\[g_h^{\ast,\mr{Quot}}(n) = \prod_{i \in [d - 1]} \prod_{x_1 + \cdots + x_D =  i} \prod_{j=1}^{\dim(W_{i,\ast})} \exp(\wt{e}_{i,j})^{z_{x,j}^\ast\cdot\frac{n^x}{x!}}  \prod_{i \in [d - 1]} \prod_{x_1 + \cdots+ x_D = i} \prod_{j=1}^{\dim(W_{i,\mr{Lin}})} \exp(\wt{e}_{i,j +  \dim(W_{i,\ast})})^{\gamma_{x,j}\cdot\frac{n^{x}}{x!}}\]
\[g_h^{\mr{Lin},\mr{Quot}}(n)  = \prod_{i \in [d - 1]} \prod_{x_1 + \cdots + x_D =  i} \prod_{j=1}^{\dim(W_{i,\mr{Lin}})} \prod_{k=1}^{d^\ast}\exp(\wt{e}_{i,D_i^\ast + (j-1)d^\ast+k})^{\alpha_{x,j,k}\{\beta_k \cdot h\}\cdot\frac{n^x}{x!}}\]
Note that
\[g_h^{\ast,\mr{Quot}} = g_h^{\ast,\mr{Univ}}\imod G_{\mr{Rel}},\qquad g_h^{\mr{Lin},\mr{Quot}} = g_h^{\mr{Lin},\mr{Univ}}\imod G_{\mr{Rel}}.\]
We will define
\[g_h^{\mathrm{Pet}, \mathrm{Quot}} = g_h^{\mathrm{Rem}, \mathrm{Quot}} = \mathrm{Id}, g_h^{\mathrm{Quot}} = g_h^{\ast, \mathrm{Quot}} \cdot g_h^{\mr{Lin}, \mr{Quot}}.\]
Given this setup, we have the following which is exactly \cite[Lemma 10.12]{LSS24b}.
\begin{lemma}\label{lem:reduct-quot}
Given the above setup, let 
\[G_{\mr{Univ}}^{\triangle} :=\{(g,g \imod G_{\mr{Rel}})\in G_{\mr{Univ}}\times G_{\mr{Quot}}\colon g\in G_{\mr{Univ}}\}\]
which is given the degree-rank filtration
\[(G_{\mr{Univ}}^{\triangle})_{(d,r)} :=\{(g,g \imod G_{\mr{Rel}})\in (G_{\mr{Univ}})_{(d,r)}\times (G_{\mr{Quot}})_{(d,r)}\colon g\in (G_{\mr{Univ}})_{(d,r)}\}.\]
Define 
$\Gamma_{\mr{Univ}}^{\triangle} = G_{\mr{Univ}}^{\triangle} \cap (\Gamma_{\mr{Univ}}\times\Gamma_{\mr{Quot}})$. We have:
\begin{itemize}
    \item $(g_h^{\mr{Univ}},g_h^{\mr{Quot}})$ is a polynomial sequence on $G_{\mr{Univ}}^{\triangle}$ with respect to the given degree-rank filtration;
    \item The function
    \[(g,g')\mapsto\wt{F}(g\Gamma_{\mr{Univ}}) \otimes \ol{F^\ast}(g'\Gamma_{\mr{Quot}})\]
    for $(g,g')\in G_{\mr{Univ}}^{\triangle}$ is $(G_{\mr{Univ}}^{\triangle})_{(d-1,r)}$-invariant;
    \item $G_{\mr{Univ}}^{\triangle}$ has complexity bounded by $M(\delta)$;
    \item Each coordinate of $ \wt{F}(g\Gamma_{\mr{Univ}}) \otimes \ol{F^\ast}(g'\Gamma_{\mr{Quot}})$ is $M(\delta)$-Lipschitz.
\end{itemize}
\end{lemma}
The point, of course, is the second bullet point, which states that the $F(g_h(n)\Gamma)$ and $\wt{F}(g_h^{\mathrm{Univ}}\Gamma)$ differ by a nilsequence of strictly lower degree-rank via viewing $F \otimes \overline{\wt{F}}$ on a quotient nilmanifold of the product. 

We will now actually construct the additional nilmanifold which incorporates the bracket linear structure in the nilmanifold formalism. Via reindexing, we may write
\[g_h^{\mr{Quot}}(n) = \prod_{i \in [d - 1]} \prod_{x_1 + \cdots+ x_D = i} \prod_{j=1}^{\dim(W_{i,\mr{Lin}})} \exp(\wt{e}_{i,j +  \dim(W_{i,\ast})})^{\gamma_{x,j}\cdot\frac{n^{x}}{x!}}\]
\[\prod_{i \in [d - 1]} \prod_{x_1 + \cdots + x_D =  i} \prod_{j=1}^{\dim(W_{i,\mr{Lin}})} \exp(\wt{e}_{i,j})^{\alpha_{x,j}\{\beta_{x, j} \cdot h\}\cdot\frac{n^x}{x!}}.\]
\textbf{The next part of the proof will contain the most significant changes from \cite{LSS24b}.} Since $\{\beta_{x, j} \cdot h\}$ are multidimensional in $x$, we will need to lift to a larger universal nilmanifold. Thankfully, since our undesirable $h$-dependent terms $g_h^{\mathrm{Pet}}$ and $g_h^{\mathrm{Rem}}$ are eliminated from the argument above, this makes it so that we do not need to run a multidegree-rank induction argument. We now define the multidegree universal nilmanifold with associated dimensional data. We now consider $G_{\mathrm{Multi-Univ}}/\Gamma_{\mathrm{Multi-Univ}}$ to be the same as $G_{\mathrm{Univ}}$ except for each $x$, we have generators $e_{x, j}$ with $e_{x, j}$ satisfying the same commutator relations with $e_{|x|, j}$. Similarly, we define $G_{\mathrm{MultiRel}}/\Gamma_{\mathrm{MultiRel}}$ as the same as $G_{\mathrm{Rel}}/\Gamma_{\mathrm{Rel}}$ except that we take generators $e_{x, j}$ instead of $e_{i, j}$ and relations with $e_{x, j}$ satisfying the same commutator relations with $\wt{e}_{|x|, j}$. Take $G_{\mathrm{MultiQuot}}$ as $G_{\mathrm{Multi-Univ}}/G_{\mathrm{MultiRel}}$ and specify $\wt{e}_{x, j}$ as projection of $e_{x, j}$. We define $\Gamma_{\mathrm{MultiQuot}}$ similarly.

We now lift $F^*$ to $F_{\mathrm{MultiQuot}}^*$ on $G_{\mathrm{MultiQuot}}/\Gamma_{\mathrm{MultiQuot}}$ by taking each generator $\wt{e}_{x, j}$ to $\wt{e}_{|x|, j}$. Hence, defining
\[g_h^{\mr{MultiQuot}}(n) = \prod_{i \in [d - 1]} \prod_{x_1 + \cdots+ x_D = i} \prod_{j=1}^{\dim(W_{i,\mr{Lin}})} \exp(\wt{e}_{x,j +  \dim(W_{i,\ast})})^{\gamma_{x,j}\cdot\frac{n^{x}}{x!}}\]
\[\prod_{i \in [d - 1]} \prod_{x_1 + \cdots + x_D =  i} \prod_{j=1}^{\dim(W_{i,\mr{Lin}})} \exp(\wt{e}_{x,j})^{\alpha_{x,j}\{\beta_{x, j} \cdot h\}\cdot\frac{n^x}{x!}},\]
we see that $F_{\mathrm{MultiQuot}}^*(g_h^{\mr{MultiQuot}}(n)\Gamma_{\mr{MultiQuot}})$ is a well-defined nilsequence on $G_{\mathrm{MultiQuot}}/\Gamma_{\mathrm{MultiQuot}}$.

We now define $G_{\mr{Lin}}$ to be the Lie subgroup of $G_{\mr{MultiQuot}}$ such that $\log(G_{\mr{Lin}})$ is the subspace generated by all $(r-1)$-fold iterated commutators (with $r\ge 1$) of $\wt{e}_{x_1,j_1},\ldots,\wt{e}_{x_r,j_r}$ with $j_\ell>D_{|x|_\ell}^\ast$ for exactly one index $\ell$. We have the following observation, which is essentially \cite[Claim 11.1]{LSS24b}.

\begin{claim}\label{clm:abel}
We have that $G_{\mr{Lin}}$ is well-defined, abelian, and normal with respect to $G_{\mr{MultiQuot}}$.  
\end{claim}
Due to normality, $G_{\mr{MultiQuot}}$ acts on $G_{\mr{Lin}}$ via conjugation. In particular, we define $G_{\mr{MultiQuot}}\rtimes G_{\mr{Lin}}$ with the group law given by 
\[(g,g_1) (g',g_1') := (gg',g_1^{g'}g_1') = (gg',((g')^{-1} g_1 g')g_1').\]
We now introduce a manner in which the additive group $R=\mb{R}^{\sum_{i=1}^{d-1}\sum_{x_1 + \cdots + x_D = i} D_x^{\mr{Lin}}}$, with elements denoted
\[t = (t_{x,j})_{x,~D_{|x|,\ast}<j\le D_{|x|} + D_{|x|}^{\mr{Lin}}},\]
acts on $G_{\mr{MultiQuot}}\ltimes G_{\mr{Lin}}$. 

For each $t\in R$, we define the homomorphism $g\mapsto g^{t}$ from $G_{\mr{MultiQuot}}$ to itself on generators. We map $\exp(\wt{e}_{x,j})\to\exp(\wt{e}_{x,j})^{t_{x,j}}$ for $1\le |x|\le d-1$ and $D_{|x|}^\ast<j\le D_{|x|}^\ast + D_i^{\mr{Lin}}$ while $\exp(\wt{e}_{x,j})$ is fixed
for $1\le |x|\le d-1$ and $1\le j\le D_{|x|}^\ast$ and we define
$\rho\colon R\to\on{Aut}(G_{\mr{MultiQuot}}\rtimes G_{\mr{Lin}})$ by
\[\rho(t)(g,g_1) := (g\cdot g_1^{t},g_1).\]
Via calculations in \cite[pg. 60]{LSS24b}, $\rho$ is well-defined. Thus, we can take the semidirect product $G_{\mathrm{Multi}} = R \rtimes_\rho (G_{\mathrm{MultiQuot}} \rtimes G_{\mathrm{Lin}})$ with respect to $\rho$. We give $G_{\mathrm{Multi}}$ a multi-degree filtration as follows. For $d_1 \in \mathbb{N}$ and $d_2 \in \mathbb{N}^D$, we define
\begin{itemize}
    \item If $x_1>1$ then $(G_{\mr{Multi}})_{(x_1,d_2)}= \mr{Id}_{G_{\mr{Multi}}}$;
    \item If $|d_2|>0$ then $(G_{\mr{Multi}})_{(1,d_2)} = \{(0,(g,\mr{id}_{G_{\mr{Lin}}}))\colon g\in (G_{\mr{MultiQuot}})_{d_2}\cap G_{\mr{Lin}}\}$;
    \item $(G_{\mr{Multi}})_{(1,0)} = \{(t,(g,\mr{id}_{G_{\mr{Lin}}}))\colon t\in R, g\in (G_{\mr{MultiQuot}})_{0}\cap G_{\mr{Lin}}\}$ or equivalently just $\{(t,(g,\mr{id}_{G_{\mr{Lin}}}))\colon t\in R, g\in G_{\mr{Lin}}\}$;
    \item If $|d_2| > 0$ then $(G_{\mr{Multi}})_{(0,d_2)} = \{(0,(g,g_1))\colon g\in (G_{\mr{Quot}})_{d_2}, g_1\in (G_{\mr{MultiQuot}})_{(d_2,0)}\cap G_{\mr{Lin}}\}$;
    \item $(G_{\mr{Multi}})_{(0,0)} = G_{\mr{Multi}}$.
\end{itemize}
\begin{claim}\label{clm:filtration}
$(G_{\mr{Multi}})_{(d_1,d_2)}$ is a valid multidegree filtration on $G_{\mr{Multi}}$.
\end{claim}
\begin{proof}
(The proof is little more than simply changing $G_{\mathrm{Quot}}$ to $G_{\mathrm{MultiQuot}}$ in \cite[Claim 11.2]{LSS24b}.) Note that 
\[(t,(g,g_1)) = (t,(\mr{id}_{G_{\mr{MultiQuot}}},\mr{id}_{G_{\mr{Lin}}})) \cdot (0, (g,g_1))\]
and therefore $(G_{\mr{Multi}})_{(0,0)} = (G_{\mr{Multi}})_{(1,0)}\vee \bigvee_{i = 1}^D (G_{\mr{Multi}})_{(0,e_i)}$. We next check various commutator relations. First note that 
\[[(G_{\mr{Multi}})_{(1,0)},(G_{\mr{Multi}})_{(1,0)})] = \mr{Id}_{G_{\mr{Multi}}}.\]
This follows because if $g,h\in G_{\mr{Lin}}$ we have $gh=hg$ hence
\[(t,(g,\mr{id}_{G_{\mr{Lin}}})) \cdot (t',(h,\mr{id}_{G_{\mr{Lin}}}))) = (t + t',(gh,\mr{id}_{G_{\mr{Lin}}})) = (t',(h,\mr{id}_{G_{\mr{Lin}}}))) \cdot (t,(g,\mr{id}_{G_{\mr{Lin}}})).\]
Therefore it suffices to verify that
\begin{align*}
[(G_{\mr{Multi}})_{(0,\vec{a})},(G_{\mr{Multi}})_{(0,\vec{b})}]&\leqslant(G_{\mr{Multi}})_{(0,\vec{a} + \vec{b})},\\
[(G_{\mr{Multi}})_{(1,\vec{a})},(G_{\mr{Multi}})_{(0,\vec{b})}]&\leqslant(G_{\mr{Multi}})_{(1,\vec{a} + \vec{b})}.
\end{align*}

We first tackle the first claim, in which we may reduce to the case $|\vec{a}|,|\vec{b}|>0$. We wish to show
\[[(g,g_1),(g',g_1')] \in \{(h,h_1)\colon h\in (G_{\mr{MultiQuot}})_{\vec{a}+\vec{b}}, h_1\in (G_{\mr{MultiQuot}})_{\vec{a}+\vec{b}}\cap G_{\mr{Lin}}\}\]
if $g,g_1\in (G_{\mr{MultiQuot}})_{\vec{a}}$, $g',g_1'\in (G_{\mr{MultiQuot}})_{\vec{b}}$, and $g_1,g_1'\in G_{\mr{Lin}}$. It suffices to prove $(G_{\mr{MultiQuot}})_{\vec{a}+\vec{b}}$ is normal in $(G_{\mr{MultiQuot}})_{\vec{a}}$ and $(G_{\mr{MultiQuot}})_{\vec{b}}$ and then check at the level of generators. 

To check normality, we have
\begin{align*}
(g,g_1)(g',g_1')(g,g_1)^{-1} &= (g,g_1)(g',g_1')(g^{-1},gg_1^{-1}g^{-1})\\
&=(gg',(g')^{-1}g_1g' \cdot g_1')(g^{-1},gg_1^{-1}g^{-1})\\
& = (gg'g^{-1}, (g(g')^{-1})g_1 (g'g^{-1}) \cdot gg_1'g^{-1} \cdot gg_1^{-1}g^{-1})
\end{align*}
and the result follows noting that $G_{\mr{Lin}},(G_{\mr{MultiQuot}})_{\vec{a}}$ are normal in $G_{\mr{MultiQuot}}$ for all $\vec{a} \neq 0$. 

Since
\[(g,g_1) = (g,\mr{id}_{G_{\mr{Lin}}}) \cdot (\mr{id}_{\mr{MultiQuot}}, g_1)\]
and it suffices to check the claim on generators, we may reduce to the case where exactly one of $g,g_1$ and exactly one of $g_1,g_1'$ are the identity. The result is clear when $g,g'$ are trivial and the case when $g_1,g_1'$ are trivial follows from the fact that we have a valid filtration on $G_{\mr{MultiQuot}}$. In the remaining cases, we may assume by symmetry that $g_1 = \mr{id}_{G_{\mr{Lin}}}$ and $g' = \mr{id}_{G_{\mr{MultiQuot}}}$. We have 
\[(g^{-1},\mr{id}_{G_{\mr{Lin}}})(\mr{id}_{G_{\mr{MultiQuot}}},(g_1')^{-1})(g,\mr{id}_{G_{\mr{Lin}}})(\mr{id}_{G_{\mr{MultiQuot}}},g_1') = (\mr{id}_{G_{\mr{MultiQuot}}},g^{-1}(g_1')^{-1}gg_1')\]
and we see that the final coordinate satisfies $[g,g_1']\in (G_{\mr{MultiQuot}})_{(\vec{a}+\vec{b},0)}\cap G_{\mr{Lin}}$. We have finished verifying the first claim.

Now note that $\{(h,\mr{id}_{G_{\mr{Lin}}})\colon h\in G_{\mr{Lin}}\}$ is a normal subgroup of $G_{\mr{MultiQuot}}\ltimes G_{\mr{Lin}}$, since $G_{\mr{Lin}}$ is abelian. Thus, combining with the first claim gives the second claim, namely
\[[(G_{\mr{Multi}})_{(1,\vec{a})},(G_{\mr{Multi}})_{(0,\vec{b})})]\leqslant(G_{\mr{Multi}})_{(1,\vec{a} + \vec{b})},\]
for $|\vec{a}|>0$.

The only nontrivial case left is $\vec{a}=0$ and $|\vec{b}|>0$ for the second claim. Furthermore, combining what we know, it suffices to check the case when $(t,(\mr{id}_{G_{\mr{MultiQuot}}},\mr{id}_{G_{\mr{Lin}}}))$ is the element from $(G_{\mr{Multi}})_{(1,0)}$. Note however that
\begin{align*}
(t,&(\mr{id}_{G_{\mr{MultiQuot}}},\mr{id}_{G_{\mr{Lin}}}))\cdot(0,(g,g_1))\cdot(-t,(\mr{id}_{G_{\mr{MultiQuot}}},\mr{id}_{G_{\mr{Lin}}}))\cdot(0,(g,g_1))^{-1}\\
&= (t,(g,g_1))\cdot(-t,(\mr{id}_{G_{\mr{MultiQuot}}},\mr{id}_{G_{\mr{Lin}}}))\cdot(0,(g^{-1},gg_1^{-1}g^{-1})) = (0,(gg_1^{-t},g_1))\cdot(0,(g^{-1},gg_1^{-1}g^{-1}))\\
&= (0,(gg_1^{-t}g^{-1},\mr{id}_{G_{\mr{Lin}}})).
\end{align*}
and the fact that if $g,g_1\in(G_{\mr{MultiQuot}})_{(\vec{b},0)}$ and $g_1\in G_{\mr{Lin}}$ then $gg_1^{-t}g^{-1}\in(G_{\mr{MultiQuot}})_{(\vec{b},0)}$. This follows because if $g_1\in(G_{\mr{MultiQuot}})_{(\vec{b},0)}\cap G_{\mr{Lin}}$ then $g_1^t$ is in the same group.
\end{proof}
Writing $t=(t_{x,j})_{1\le |x|\le d-1,~D_{|x|}^\ast<j\le D_{|x|} + D_{|x|}^{\mr{Lin}}}$, we define
\[\Gamma_{\mr{Multi}} = \{(t,(g,g_1))\colon t_{x,j}\in\mb{Z}, g\in\Gamma_{\mr{MultiQuot}},g_1\in\Gamma_{\mr{MultiQuot}}\cap G_{\mr{Lin}}\}.\] 
Let $\phi\colon\mb{R}\to\mb{R}$ be a $1$-bounded, $1$-periodic function such that:
\begin{itemize}
    \item $\phi(x)=1$ if $|\{x\}|\le 1/2-2\epsilon(\delta)$;
    \item $\phi(x)=0$ if $|\{x\}|\ge 1/2-\epsilon(\delta)$;
    \item $\phi$ is $O(1/\epsilon(\delta))$-Lipschitz.
\end{itemize}
Define $H^\ast\subseteq H$ such that for all $1\le |x|\le d-1$ and $D_{|x|}^\ast<j\le D_{|x|}^\ast + D_{|x|}^{\mr{Lin}}$ we have $|\{\beta_{x,j} \cdot h\}|\ge 1/2-\epsilon(\delta)$. Using the fact that $\beta_{x,j}\in(1/N')\mb{Z}^{D}$ where $N'$ is a prime between $100N$ and $200N$, we see that there are at most $O(\epsilon(\delta) \cdot N^{D'}\cdot\sum_{i=1}^{d-1} D_i^{\mr{Lin}})$ indices which do not satisfy the criterion and choosing our parameter $\epsilon(\delta)$ sufficiently small while still preserving the property of being $\epsilon(\delta)$, we may assume that $H^\ast$ is at least half the size of $H$.

Given $(t,(g,g_1))\in G_{\mr{Multi}}$, we may find $(t',(g',g_1'))\in (t,(g,g_1))\Gamma_{\mr{Multi}}$ such that $(t')_{i,j}\in(-1/2,1/2]$ for all $i,j$. Define
\[F_{\mr{Multi}}((t,(g,g_1)) \Gamma_{\mr{Multi}}) = F^\ast(g'\Gamma_{\mr{MultiQuot}})\cdot\prod_{\substack{1\le i\le d-1\\ D_i^{\ast}<j\le D_i^\ast + D_i^{\mr{Lin}}}} \phi(t_{i,j}') ;\]
we check that this in fact gives a well-defined function on $G_{\mr{Multi}}/\Gamma_{\mr{Multi}}$. Note that if $(t',(g',g_1'))\in (t,(g,g_1))\Gamma_{\mr{Multi}}$ and $t_{i,j}'\in(-1/2,1/2]$ then $t_{i,j}' = \{t_{i,j}\}$ and hence $t'$ is unique. Furthermore, note that 
\[(t',(g',g_1')) \cdot (0, (\gamma',\gamma_1')) = (t', (g',g_1')\cdot (\gamma',\gamma_1')) = (t', (g' \gamma', (\gamma')^{-1} g_1' \gamma' \gamma_1'))\]
and trivially 
\[F^\ast(g'\Gamma_{\mr{MultiQuot}}) = F^\ast(g'\gamma'\Gamma_{\mr{MultiQuot}})\]
if $\gamma'\in\Gamma_{\mr{MultiQuot}}$. Now recall that 
\begin{align*}
g_h^{\ast,\mr{MultiQuot}}(n) &= \prod_{i \in [d - 1]}\prod_{x_1 + \cdots + x_D = i} \prod_{j=1}^{\dim(W_{i,\ast})} \exp(\wt{e}_{i,j})^{z_{x,j}^\ast\cdot\frac{n^x}{x!}}  \prod_{i \in [d - 1]} \prod_{x_1 + \cdots + x_D = i} \prod_{j=1}^{\dim(W_{i,\mr{Lin}})} \exp(\wt{e}_{i,j +  \dim(W_{i,\ast})})^{\gamma_{x,j}\cdot\frac{n^x}{x!}}.
\end{align*}
We set 
\begin{align*}
g_0(n) &= \prod_{i \in [d - 1]} \prod_{x_1 + \cdots + x_D = i} \prod_{j=1}^{D_i^\ast} \exp(\wt{e}_{i,j})^{\gamma_{x,j}\cdot\frac{n^x}{x!}},\quad g_1(n) = \prod_{i \in [d - 1]} \prod_{x_1 + \cdots + x_D = i}\prod_{j=D_i^\ast+1}^{D_i^\ast+D_i^{\mr{Lin}}} \exp(\wt{e}_{i,j})^{\alpha_{i,j}\cdot\frac{n^x}{x!}}
\end{align*}
and define
\begin{align*}
g_{\mr{Final}}(h,n) &= (0,(g_0(n),g_1(n)))\cdot ((\beta_{x,j} \cdot h)_{\substack{1\le |x|\le d-1\\D_i^{\ast}<j\le D_{|x|}^\ast + D_{|x|}^{\mr{Lin}}}},(\mr{id}_{G_{\mr{MultiQuot}}},\mr{id}_{G_{\mr{Lin}}})) \\
&= (0,(g_0(n),\mr{id}_{G_{\mr{Lin}}}))\cdot(0,(\mr{id}_{G_{\mr{MultiQuot}}},g_1(n)))\cdot((\beta_{x,j} \cdot h)_{\substack{1\le |x|\le d-1\\D_{|x|}^{\ast}<j\le D_{|x|}^\ast + D_{|x|}^{\mr{Lin}}}},(\mr{id}_{G_{\mr{MultiQuot}}},\mr{id}_{G_{\mr{Lin}}})). 
\end{align*}
The crucial point is that for all $h \in H$,
\begin{align*}
g_{\mr{Final}}(h,n)\Gamma_{\mr{Multi}} &= (0,(g_0(n),g_1(n)))\cdot ((\{\beta_{i,j} \cdot h\})_{\substack{1\le |x|\le d-1\\D_{|x|}^{\ast}<j\le D_{|x|}^\ast + D_{|x|}^{\mr{Lin}}}},(\mr{id}_{G_{\mr{MultiQuot}}},\mr{id}_{G_{\mr{Lin}}}))\Gamma_{\mr{Multi}}\\
&= ((\{\beta_{x,j}h\})_{\substack{1\le |x|\le d-1\\D_{|x|}^{\ast}<j\le D_{|x|}^\ast + D_{|x|}^{\mr{Lin}}}},(\mr{id}_{G_{\mr{MultiQuot}}},\mr{id}_{G_{\mr{Lin}}}))\cdot(0,(g_{0,h}^\ast(n),g_1(n)))\Gamma_{\mr{Multi}},
\end{align*}
writing
\[g_{0,h}^\ast(n)=g_0(n)(g_1(n))^{t(h)}\]
where $t(h)=(\{\beta_{x,j} \cdot h\})_{1\le |x|\le d-1,~D_{|x|}^{\ast}<j\le D_{|x|}^\ast + D_{|x|}^{\mr{Lin}}}\in R$. This is precisely the desired sense, discussed earlier, in which we have used the group action to ``raise'' parts of $G_{\mr{Lin}}$ to $h$-fractional powers.

Therefore, for all $h\in H^\ast$ we have 
\begin{equation}\label{eq:multi-rep}
F_{\mr{Multi}}(g_{\mr{Final}}(h,n)\Gamma_{\mr{Multi}}) = F^\ast(g_h^{\mr{MultiQuot}}(n)\Gamma_{\mr{MultiQuot}}).
\end{equation}
This, in combination with (the discussion after) Lemma~\ref{lem:reduct-quot} completes the proof of Proposition~\ref{prop:degreerankextraction}.  
\end{proof}

\section{Proof of the multidimensional inverse theorem}
The purpose of this section is to deduce Theorem~\ref{thm:maininversetheorem2} assuming that the theorem holds for $(\ell, \ell' - 1)$. This will be done via a symmetry and integration argument as in \cite[Section 12]{LSS24b}
\subsection{Setup for the symmetry and integration arguments}
We begin this section with the following proposition which is essentially the analogue of \cite[Theorem 6.4]{LSS24b}.
\begin{prop}\label{prop:correlationdata}
Fix $\delta\in (0,1/2)$, $\ell, \ell', k \le K$ be parameters. Assume Theorem~\ref{thm:maininversetheorem2} for $(\ell, \ell'-1)$. Let $f\colon[N]^k\to\mb{C}$ be a $1$-bounded function such that 
\[\snorm{f}_{U([N]^k, \dots, [N]^k, e_1[N], \dots, e_\ell[N])}\ge\delta\]
with $\ell'$ copies of $[N]$.

Then the following data exists:
\begin{itemize}
    \item A subset $H\subseteq[N]^k$ of size at least $\epsilon(\delta) N^k$;
    \item A multi-degree nilcharacter $\chi(h,n) \in \mathrm{Nil}^{(1, \ell + \ell' - 1)}(M(\delta), m(\delta), k, M(\delta))$;
    \item For all $h\in H$, there exists one-bounded functions $f_{1, h}, \dots, f_{k, h}$ with $f_{i, h}$ not depending on the $i$th coordinate and $\psi_h \in \mathrm{Nil}^{\ell + \ell' - 2}(M(\delta), m(\delta), k, M(\delta))$ such that
    \[\mathbb{E}_{h \in [N]^k} \|\mathbb{E}_{n \in [N]^k}\Delta_h f(n) \otimes \ol{\chi(h,n)} \otimes \prod_{j = 1}^\ell f_{j, h}(n) \cdot \psi_h(n)\|_\infty \ge \epsilon(\delta).\]
\end{itemize}
\end{prop}
We shall deduce this via a degree-rank induction argument from the following proposition; this is essentially our analogue of \cite[Proposition 6.3]{LSS24b}. 
\begin{prop}\label{prop:correlationdegreerank}
Fix $\delta\in (0,1/2)$ and $\ell, \ell', k \le K$ be parameters. Let $f\colon[N]^k\to\mb{C}$ be a $1$-bounded function such that for each $h \in H$, a subset of $[N]^k$, we have
\[\mathbb{E}_{h \in [N]^k} \|\mathbb{E}_{n \in [N]^k} \Delta_h f(n) \otimes \chi_{h, 0}(n) \otimes \ol{\chi_0(h, n)} \otimes \prod_{j = 1}^\ell f_{j, h, 0}(n) \cdot \psi_h(n)\|_\infty \ge \epsilon(\delta)\]
for
\begin{itemize}
    \item a multi-degree nilcharacter $\chi_0(h,n) \in \mathrm{Nil}^{(1, \ell + \ell' - 1)}(M(\delta), m(\delta), k, M(\delta))$ that is a product of $O_K(1)$ $\mathbb{Z}$-multidegree nilcharacters also in $\mathrm{Nil}^{(1, \ell + \ell' - 1)}(M(\delta), m(\delta), k, M(\delta))$;
    \item a degree-rank $(\ell + \ell' - 1, r)$ nilcharacter $\chi_{h, 0}(n) \in \mathrm{Nil}^{(\ell + \ell' - 1, r)}(M(\delta), m(\delta), k, M(\delta))$ that is a product of $O_K(1)$ $\mathbb{Z}$-multidegree nilcharacters in $\mathrm{Nil}^{(\ell + \ell' - 1, r)}(M(\delta), m(\delta), k, M(\delta))$;
    \item a nilsequence $\psi_h(n) \in \mathrm{Nil}^{\ell + \ell' - 2}(M(\delta), m(\delta), k, M(\delta))$;
    \item and one bounded functions $f_{1, h, 0}, \dots, f_{k, h, 0}$ with $f_{i, h, 0}$ not depending on the $i$th coordinate.
\end{itemize}
Then the following data exists:
\begin{itemize}
    \item A subset $H\subseteq H$ of size at least $\epsilon(\delta) |H|$;
    \item A multi-degree nilcharacter $\chi(h,n) \in \mathrm{Nil}^{(1, \ell + \ell' - 1)}(M(\delta), m(\delta), k, M(\delta))$;
    \item A degree-rank $(\ell + \ell' - 1, r)$ nilcharacter $\chi_h(n) \in \mathrm{Nil}^{(\ell + \ell' - 1, r)}(M(\delta), m(\delta), k, M(\delta))$;
    \item For all $h\in H$, there exists one-bounded functions $f_{1, h}, \dots, f_{\ell, h}$ such that there exists $\psi_h \in \mathrm{Nil}^{\ell + \ell' - 2}(M(\delta), m(\delta), k, M(\delta))$ such that
    \[\mathbb{E}_{h \in [N]^k} \|\mathbb{E}_{n \in [N]^k}\Delta_h f(n) \otimes \chi_h(n) \otimes \ol{\chi(h,n)} \otimes \prod_{j = 1}^\ell f_{j, h}(n)\|_\infty \ge \epsilon(\delta).\]
\end{itemize}
\end{prop}
\begin{proof}
First, we absorb the terms in the product of the $\mathbb{Z}$-multidegree nilcharacters which are degree $0$ in the $e_1, \dots, e_\ell$ variables inside the $f_{1, h}, \dots, f_{\ell, h}$ functions. By Lemma~\ref{lem:four-fold-biased}, Proposition~\ref{prop:degreerankextraction}, and the fact that $\chi_{h, 0}(n)$ and $\chi_{h, 0}(n + n_0)$ are $(m(\delta), M(\delta))$ equivalent for fixed $n_0$, we see that there exist nilcharacters $\widetilde{\chi_h}(n) \in \mathrm{Nil}^{< (\ell + \ell' - 1, r)}(M(\delta), m(\delta), k, M(\delta))$ and $\widetilde{\chi}(h, n) \in \mathrm{Nil}^{(1, \ell + \ell' - 1)}(M(\delta), m(\delta), k, M(\delta))$ such that $\chi_{h, 0}$ and $\widetilde{\chi}_h \otimes \widetilde{\chi}$ are $(m(\delta), M(\delta))$-equivalent. The proposition then follows from Lemma~\ref{lem:equiv} applied to these two nilcharacters. 
\end{proof}
\begin{proof}[Proof of Proposition~\ref{prop:correlationdata}]
By the discussion in the beginning of Section 6, we see that there exists a family of nilcharacters $\chi_h \in \mathrm{Nil}^{\ell + \ell' - 1}(M(\delta), d(\delta), \ell, M(\delta))$ and functions $f_{1, h}, \dots, f_{\ell, h}$ such that $f_{i, h}$ does not depend on coordinate $i$, and
$$\mathbb{E}_{h \in [N]^k} \bigg\|\mathbb{E}_{n \in [N]^k} \Delta_h f(n) \chi_h(n)\prod_{i = 1}^{\ell} f_{i, h}(n)\bigg\|_\infty^2 \ge \epsilon(\delta).$$
By \cite[Lemma C.6]{LSS24b} (or rather the stronger Lemma~\ref{lem:splitdegreerank}) and Fourier expanding via Lemma~\ref{lem:nilcharacters} we may assume that $\chi_h$ is a product of $\mathbb{Z}$-multidegree nilcharacters. (We implicitly used that each vertical character $F$ of the product can be taken to be a component of a nilcharacter of the same degree-rank. This follows from a slightly modified version of \cite[Lemma B.4]{LSS24b}, adapted to degree-rank. Applying that lemma, we find a nilcharacter $\chi$ of the same $G_{(\ell + \ell' - 1, r)}$ vertical frequency. Then by taking $(F/(2\|F\|_\infty), \chi\sqrt{1 - |F/(2\|F\|_\infty)|^2})$, we find that a multiple of $F$ is a component of a nilcharacter.)

This implies that there exists at least $\epsilon(\delta)N^k$ many $h \in [N]^k$ such that
\[\bigg\|\mathbb{E}_{n \in [N]^k} \Delta_h f(n) \chi_h(n)\prod_{i = 1}^{\ell} f_{i, h}(n)\bigg\|_\infty \ge \epsilon(\delta).\]
The proposition thus follows from Proposition~\ref{prop:correlationdegreerank}, the degree-rank induction, the splitting lemma at each stage of the degree-rank induction, and Fourier expansion\footnote{Technically, the (slightly modified) splitting lemma only gives a product ``multidegree-rank $(\vec{d}, r)$ filtration", but by taking joins, one can take a degree-rank $(|\vec{d}|, r)$ filtration. The advantage of using the splitting lemma is that we may use multilinearity properties of multidimensional input nilcharacters.}.
\end{proof}
\subsection{The symmetry and integration arguments}
We are now ready to deduce Theorem~\ref{thm:maininversetheorem2}.
\begin{proof}[Proof of Theorem~\ref{thm:maininversetheorem2}]
We begin with the hypothesis of there being at least $\epsilon(\delta)|H|$ many $h$ such that
$$\bigg\|\mathbb{E}_{n \in [N]^k} \Delta_h f(n) \otimes \chi(h, n) \otimes \prod_{i = 1}^\ell f_{i, h}(n) \cdot \psi_h(n)\bigg\|_\infty \ge \epsilon(\delta).$$
By \cite[Lemma C.6]{LSS24b} (the splitting lemma) and Fourier expanding via Lemma~\ref{lem:nilcharacters}, we may assume that $\chi$ is a product of nilcharacters (with complexity, dimension, input dimension, and output dimension $M(\delta), m(\delta), k, M(\delta)$, respectively) that depend on at most $\ell + \ell'$ coordinates. Applying \cite[Lemma C.5]{LSS24b} (multilinearization) to each nilcharacter of this product, we may assume that $\chi(h, n)$ is a multilinear $(1, 1, \dots, 1)$ nilcharacter of the form $\chi(h, n, n, \dots, n)$. We may further assume that the final $\ell + \ell' - 1$ coordinates are permutation invariant. We thus wish to show that $\chi(h, n, \dots, n)$ is $(M(\delta), m(\delta))$-equivalent to $\chi(n, h, \dots, n)$. Now applying Lemma~\ref{lem:four-fold-biased}, we see that
\begin{align*}
&\mb{E}_{\substack{h_1 + h_2 = h_3 + h_4 \\ h_i\in[N]^k}}\snorm{\mb{E}_{n\in[N]^k} \wt{\chi}(h_1,n)\otimes \wt{\chi}(h_2,n + h_1-h_4) \otimes \ol{\wt{\chi}(h_3,n)}\otimes \ol{\wt{\chi}(h_4,n + h_1-h_4)} \\
& \qquad\qquad\qquad\otimes \psi_{h_1}(n)\otimes \psi_{h_2}(n + h_1 - h_4)\otimes \ol{\psi_{h_3}(n)}\otimes \ol{\psi_{h_4}(n + h_1 - h_4)}\cdot e(\Theta n)}_{\infty}\ge \epsilon(\delta)
\end{align*}
where $\wt{\chi}(h, n) = \chi(h, n)\prod_{i = 1}^\ell f_{i, h}(n).$ Now by Cauchy-Schwarzing away $f_{i, h}$, we may bound the above by the multidimensional $U(e_1[N], \dots, e_\ell[N])$ Gowers norm of
$$n \mapsto \chi(h_1,n)\otimes \chi(h_2,n + h_1-h_4) \otimes \ol{\chi(h_3,n)}\otimes \ol{\chi(h_4,n + h_1-h_4)}$$
$$\otimes \psi_{h_1}(n)\otimes \psi_{h_2}(n + h_1 - h_4)\otimes \ol{\psi_{h_3}(n)}\otimes \ol{\psi_{h_4}(n + h_1 - h_4)}\cdot e(\Theta n).$$
Expanding out the box norm, pigeonholing in some variables, and using multilinearization of $\chi$, we obtain a value of $n_0$ such that
\begin{align*}
&\mb{E}_{\substack{h_1 + h_2 = h_3 + h_4 \\ h_i\in[N]^k}}\snorm{\mb{E}_{n\in[N]^k} \chi(h_1,n + n_0)\otimes \chi(h_2,n + h_1-h_4 + n_0) \otimes \ol{\chi(h_3,n + n_0)}\otimes \ol{\chi(h_4,n + h_1-h_4 + n_0)} \\
& \qquad\qquad\qquad\otimes \psi_{h_1}'(n)\otimes \psi_{h_2}'(n + h_1 - h_4)\otimes \ol{\psi_{h_3}'(n)}\otimes \ol{\psi_{h_4}'(n + h_1 - h_4)}\cdot e(\Theta' n)}_{\infty}\ge \epsilon(\delta)
\end{align*}
for slightly different choices $\psi'_h \in \mathrm{Nil}^{\ell + \ell' - 2}(M(\delta), m(\delta), k, M(\delta))$. Noting that $\chi(h, n + n_0)$ is $(M(\delta), m(\delta))$ equivalent to $\chi(h, n)$, we are thus in the situation of the first displayed inequality \cite[Step 1]{LSS24b} and we may proceed accordingly. We eventually obtain
$$\chi(h, n, n, \dots, n) \otimes \overline{\chi(n, h, n, \dots, n)} = F^*(\epsilon(h, n, \dots, n)g^{\mathrm{Output}}(h, n, \dots, n)\gamma(h, n, \dots, n)\Gamma)$$
with
\begin{itemize}
    \item $F^*$ is a nilcharacter with frequency $\xi' = (\xi, -\xi, 0)$ where $\xi$ is the frequency of $\chi$;
    \item $g^{\mathrm{Output}}$ lies on a subgroup $H$ with $\xi'(H \cap (G \times G)_s) = 0$;
    \item $\gamma$ is a $M(\delta)$-rational sequence;
    \item and $\epsilon$ is a $(M(\delta), \vec{N})$-smooth sequence.
\end{itemize}
We now recall that
\[\mb{E}_{h\in[N]^k}\bigg\|\mb{E}_{n\in[N]^k} \Delta_{h} f(n) \otimes \chi(h, n,\ldots,n)\cdot\psi(n) \cdot\psi_{h}(n) \otimes \prod_{i = 1}^\ell f_{i, h}(n)\bigg\|_{\infty}\ge \epsilon(\delta).\]
Write $s = \ell + \ell'$. By applying Pigeonhole there exist $q,q'\in[s]^k$ such that 
\[\mb{E}_{h\in[N/s]^k}\snorm{\mb{E}_{n\in[N/s]^k} \Delta_{sh + q'} f(sn + q) \otimes \chi(sh + q', sn + q,\ldots, sn + q) \otimes \prod_{i = 1}^\ell f_{i,sh + q'}(sn + q)\cdot\psi(sn + q) \cdot\psi_{sh + q'}(sn + q)}_{\infty} \]
\[\ge \epsilon(\delta).\]

Applying multilinearity and \cite[Lemma C.6]{LSS24b} (the splitting lemma), and adjusting $\psi$ and $\psi_h$, and $f_{i, h}$, we see that
\begin{align*}
&\mb{E}_{h\in[N/s]^k}\bigg\|\mb{E}_{n\in[N/s]^k} \Delta_{sh + q'} f(sn + q) \otimes \chi(sh, sn,\ldots, sn) \otimes \prod_{i = 1}^\ell f_{i, h}\cdot\psi(n) \cdot\psi_{h}(n)\bigg\|_{\infty}\ge \epsilon(\delta).
\end{align*}
Now define
\[T(h,n) := \chi(n+h,\ldots,n+h) \otimes \ol{\chi(h,n,\ldots,n)} \otimes \ol{\chi(n,h,\ldots,n)}^{\otimes (s-1)} \otimes \ol{\chi(n,n,\ldots,n)}, b_h(n) = \prod_{i = 1}^\ell f_{i, h}(n).\]
Since $T$ is a nilcharacter, we automatically know
\begin{align*}
&\mb{E}_{h\in[N/s]^k}\snorm{\mb{E}_{n\in[N/s]^k} \Delta_{sh + q'} f(sn + q) \otimes \chi(sh, sn,\ldots, sn) \otimes b_h(n)\cdot\psi(n) \cdot\psi_{h}(n)\\
&\qquad\qquad\qquad\qquad\qquad\qquad\qquad\qquad\otimes T(h,n)^{\otimes (s)^{s-1}} \otimes \ol{T(h,n)}^{\otimes (s)^{s-1}}}_{\infty}\ge \epsilon(\delta).
\end{align*}
We define
\begin{align*}
\wt{f}_1(n) &= f(sn + q) \cdot\ol{\chi(n,\ldots, n)}^{\otimes (s)^{s-1}},\\
\wt{f}_2(n + h) &= \ol{f(s(n+h) + q+q')} \cdot\chi(n+h,\ldots,n+h)^{\otimes (s)^{s-1}},
\end{align*}
which yields
\begin{align*}
&\mb{E}_{h\in[N/s]^k}\snorm{\mb{E}_{n\in[N/s]^k} \wt{f}_1(n)\otimes\wt{f}_2(n+h) \otimes \chi(sh, sn,\ldots, sn)\cdot\psi(n) \cdot\psi_{h}(n)\\
&\qquad\qquad \otimes (\ol{\chi(h,n,\ldots,n)} \otimes \ol{\chi(n,h,n,\ldots,n)}^{\otimes (s-1)})^{\otimes (s)^{s-1}} \otimes \ol{T(h,n)}^{\otimes (s)^{s-1}} \otimes b_h(n)}_{\infty}\ge \epsilon(\delta).
\end{align*}

By multilinearity, we see that $T(h,n)$ is $(M(\delta), m(\delta))$-equivalent to 
\[\bigotimes_{k=1}^{s-1}\chi(h,h,\ldots,h,n,\ldots,n)^{\binom{s-1}{k}}\otimes\bigotimes_{k=2}^{s-1}\chi(n,h,\ldots,h,n,\ldots,n)^{\binom{s-1}{k}}.\]
(There are $k+1$ many $h$'s in the first term and $k$ many $h$'s in the second term.) Applying the splitting lemma, we may approximate each coordinate as a sum of products of multidegree $(s-1,s-2)$ and $(0,s-1)$ nilsequences in variables $(h,n)$. So, absorbing everything into $\psi(n)$ of degree $(s-1)$ and the $\psi_h(n)$ of degree $(s-2)$, we find
\begin{align*}
&\mb{E}_{h\in[N/s]^k}\norm{\mb{E}_{n\in[N/s]^k} \wt{f}_1(n)\otimes\wt{f}_2(n+h) \otimes \chi(sh, sn,\ldots, sn)\cdot\psi(n) \cdot\psi_{h}(n)\\
&\qquad \qquad\qquad\qquad\otimes (\ol{\chi(h,n,\ldots,n)} \otimes \ol{\chi(n,h,\ldots,n)}^{\otimes (s-1)})^{\otimes (s)^{s-1}} \otimes b_h(n)}_{\infty}\ge \epsilon(\delta).
\end{align*}

Applying multilinearity, we see that
\[\chi(sh, sn,\ldots, sn) \text{ and } \chi(h, n,\ldots, n)^{\otimes s^{s}}\]
are $(M(\delta),m(\delta))$-equivalent. Applying Lemma~\ref{lem:equiv} and the splitting lemma, and adjusting $\psi$ and $\psi_h$ again, we see that
\begin{align*}
&\mb{E}_{h\in[N/s]^k}\norm{\mb{E}_{n\in[N/s]^k} \wt{f}_1(n)\otimes\wt{f}_2(n+h) \otimes \chi(h, n,\ldots, n)^{\otimes s^{s}}\cdot\psi(n) \cdot\psi_{h}(n)\\
&\qquad\qquad\qquad\qquad \otimes (\ol{\chi(h,n,\ldots,n)} \otimes \ol{\chi(n,h,\ldots,n)}^{\otimes (s-1)})^{\otimes (s)^{s-1}} \otimes b_h(n)}_{\infty}\ge \epsilon(\delta).
\end{align*}
We also see that
\[\chi(h, n,\ldots, n)^{\otimes s^s}\otimes\ol{\chi(h,n,\ldots,n)}^{\otimes (s)^{s-1}} \text{ and } \chi(h,n,\ldots,n)^{\otimes (s-1) \cdot (s)^{s-1}}\]
are $(M(\delta),d(\delta))$-equivalent. Thus, applying a similar procedure as above, we see that
\begin{align*}
&\mb{E}_{h\in[N/s]^k}\norm{\mb{E}_{n\in[N/s]^k} \wt{f}_1(n)\otimes\wt{f}_2(n+h) \cdot\psi(n) \cdot\psi_{h}(n) \otimes (\chi(h,n,\ldots,n) \otimes \ol{\chi(n,h,\ldots,n)})^{\otimes (s-1)(s)^{s-1}} \otimes b_h(n)}_{\infty}\\
&\qquad\qquad\qquad\qquad\qquad\ge \epsilon(\delta).
\end{align*}

This is finally where we may apply our earlier factorization for $\chi(h,n,\ldots,n) \otimes \ol{\chi(n,h,\ldots,n)}$. Recall that
\begin{align*}
&\chi(h,n,\ldots,n) \otimes \ol{\chi(n,h,n,\ldots,n)} = F^\ast(\eps(h,n,\ldots,n) g^{\mr{Output}}(h,n,\ldots,n)\cdot\gamma(h,n,\ldots,n) (\Gamma\times\Gamma))
\end{align*}
where $\gamma$ is $M(\delta)$-periodic and $\eps$ is $(M(\delta),N)$-smooth. Let $Q$ denote the period of $\gamma$. Break $[N/s]^k$ into arithmetic progressions of length roughly $\epsilon(\delta) N$ and common difference $Q$; call these $\mc{P}_1,\ldots,\mc{P}_\ell$. Thus, approximating the smooth part, there exist a constant $\eps_{\mc{P}_{i,h}}$ with $d_{G \times G}(\eps_{\mc{P}_{i, h}}, \mathrm{Id}) \le M(\delta)$ and a $M(\delta)$-rational $\gamma_{\mc{P}_{i,h}}$ such that 
\begin{align*}
&\mb{E}_{h\in[N/s]^k}\norm{\mb{E}_{n\in[N/s]^k} \sum_{i=1}^\ell\mbm{1}_{n\in\mc{P}_i}\wt{f}_1(n)\otimes\wt{f}_2(n+h) \cdot\psi(n) \cdot\psi_{h}(n) \\
&\qquad\qquad\otimes (F^\ast(\eps_{\mc{P}_{i,h}} \gamma_{\mc{P}_{i,h}}(\gamma_{\mc{P}_{i,h}}^{-1} g^{\mr{Output}}(h,n,\ldots,n)\gamma_{\mc{P}_{i,h}}) (\Gamma\times\Gamma)))^{\otimes (s-1)(s)^{s-1}} \otimes b_h(n)}_{\infty}\ge \epsilon(\delta)
\end{align*}

By Pigeonhole, there exists an index $i$ such that 
\begin{align*}
&\mb{E}_{h\in[N/s]^k}\norm{\mb{E}_{n\in[N/s]^k} \mbm{1}_{n\in\mc{P}_i}\cdot\wt{f}_1(n)\otimes\wt{f}_2(n+h) \cdot\psi(n) \cdot\psi_{h}(n) \\
&\qquad\qquad\otimes (F^\ast(\eps_{\mc{P}_{i,h}}\gamma_{\mc{P}_{i,h}} (\gamma_{\mc{P}_{i,h}}^{-1}g^{\mr{Output}}(h,n,\ldots,n)\cdot\gamma_{\mc{P}_{i,h}}) (\Gamma\times\Gamma)))^{\otimes (s-1)(s)^{s-1}} \otimes b_h(n)}_{\infty}\ge \epsilon(\delta).
\end{align*}
By pigeonholing in $\mc{P}_{i, h}$, there is a $M(\delta)$-rational $\gamma\in\Gamma\times\Gamma$
\begin{align*}
&\mb{E}_{h\in[N/s]^k}\norm{\mb{E}_{n\in[N/s]^k}\mbm{1}_{n\in\mc{P}_i} \wt{f}_1(n)\otimes\wt{f}_2(n+h) \cdot\psi(n) \cdot\psi_{h}(n) \\
&\qquad\qquad\otimes (F^\ast(\eps_{\mc{P}_{i,h}}\gamma g^{\mr{Conj}}(h,n,\ldots,n)(\Gamma\times\Gamma)))^{\otimes(s-1)(s)^{s-1}} \otimes b_h(n)}_{\infty}\ge \epsilon(\delta),
\end{align*}
where $g^{\mr{Conj}}=\gamma^{-1}g^{\mr{Output}}\gamma$. Relabeling $\epsilon_{\mc{P}_{i, h}}$ to $\epsilon$, we see that
\begin{align*}
&\mb{E}_{h\in[N/s]^k}\norm{\mb{E}_{n\in[N/s]} \mbm{1}_{n\in\mc{P}_i}\cdot\wt{f}_1(n)\otimes\wt{f}_2(n+h) \cdot\psi(n) \cdot\psi_{h}(n) \\
&\qquad\qquad\otimes (F^\ast(\eps g^{\mr{Conj}}(h,n,\ldots,n)(\Gamma\times\Gamma)))^{\otimes (s-1)(s)^{s-1}} \otimes b_h(n)}_{\infty}\ge \epsilon(\delta)
\end{align*}

Fourier expanding the progression, there exists $\Theta_h$ such that 
\begin{align*}
&\mb{E}_{h\in[N/s]^k}\norm{\mb{E}_{n\in[N/s]^k} e(\Theta_h n)\cdot\wt{f}_1(n)\otimes\wt{f}_2(n+h) \cdot\psi(n) \cdot\psi_{h}(n) \\
&\qquad\qquad\otimes (F^\ast(\eps g^{\mr{Conj}}(h,n,\ldots,n) (\Gamma\times\Gamma))^{\otimes (s-1)(s)^{s-1}} \otimes b_h(n)}_{\infty}\ge \epsilon(\delta).
\end{align*}
As $(s-2)\ge 1$, we may absorb $e(\Theta_h n)$ into $\psi_h(n)$ and obtain 
\begin{align*}
&\mb{E}_{h\in[N/s]^k}\norm{\mb{E}_{n\in[N/s]^k} \wt{f}_1(n)\otimes\wt{f}_2(n+h) \cdot\psi(n) \cdot\psi_{h}(n) \\
&\qquad\qquad\otimes (F^\ast(\eps g^{\mr{Conj}}(h,n,\ldots,n) (\Gamma\times\Gamma)))^{\otimes (s-1)(s)^{s-1}} \otimes b_h(n)}_{\infty}\ge \epsilon(\delta).
\end{align*}
Replacing $F^\ast$ with $F^{\mr{Final}}(g) = F^\ast(\eps g\Gamma)$ and writing $g^{\mr{Final}}(h,n)=g^{\mr{Conj}}(h,n,\ldots,n)$, we have
\begin{align*}
&\mb{E}_{h\in[N/s]^k}\norm{\mb{E}_{n\in[N/s]^k} \wt{f}_1(n)\otimes \wt{f}_2(n+h) \cdot\psi(n) \cdot\psi_{h}(n) \\
&\qquad\qquad\otimes (F^{\mr{Final}}( g^{\mr{Final}}(h,n)(\Gamma\times\Gamma)))^{\otimes (s-1)(s)^{s-1}} \otimes b_h(n)}_{\infty}\ge \epsilon(\delta).
\end{align*}

Now $g^{\mr{Final}}(h,n)$ takes values in $\gamma^{-1}H\gamma$ such that $\xi'(\gamma^{-1}H\gamma\cap(G\times G)_s) = 0$. The key point is to note that $F^{\mr{Final}}$ is right-invariant under $(\gamma^{-1}H\gamma)\cap(G\times G)_s$ since it has $(G\times G)_s$-vertical frequency $\xi'$. Taking the quotient by $(\gamma^{-1}H\gamma)\cap(G\times G)_s$ gives that each coordinate of $(F^{\mr{Final}}(g^{\mr{Final}}(h,n)(\Gamma\times\Gamma)))^{\otimes (s)^{s-1}} \in \mathrm{Nil}^{s - 1}(M(\delta), m(\delta), k, M(\delta))$. 

Applying the splitting lemma to $(F^{\mr{Final}})^{\otimes (s)^{s-1}}$ and adjusting $\psi(n)$ and $\psi_h(n)$ again, we have
\begin{align*}
&\mb{E}_{h\in[N/s]^k}\snorm{\mb{E}_{n\in[N/s]^k} \wt{f}_{1}(n)\otimes\wt{f}_{2}(n+h) \cdot\psi(n) \cdot\psi_{h}(n) \otimes b_h(n)}_{\infty}\ge \epsilon(\delta)
\end{align*}
Since $\psi_h(n)$ is a nilsequence of degree $(s-2)$ and complexity $(M(\delta),m(\delta))$, by the converse of Theorem~\ref{thm:maininversetheorem2} (see Lemma~\ref{lem:converse}) we have that 
\begin{align*}
&\mb{E}_{h\in[N/s]^k}\snorm{\wt{f}_{1}(\cdot) \otimes \wt{f}_{2}(\cdot+h) \otimes\psi(\cdot)}_{U([N/s]^k, [N/s]^k, \dots, [N/s]^k, e_1[N/s], \dots, e_\ell[N/s])}\ge \epsilon(\delta).
\end{align*}
for $\ell' - 1$ copies of $[N/s]$. By the Gowers--Cauchy--Schwarz inequality, we have that 
\begin{align*}
&\snorm{\mb{E}_{n\in[N/s]}\wt{f}_{1}(n) \psi(n)}_{U([N/s]^k, [N/s]^k, \dots, [N/s]^k, e_1[N/s], \dots, e_\ell[N/s])}\ge \epsilon(\delta).
\end{align*}
Applying the induction hypothesis, there is a nilsequence $\Theta(n)$ of degree $(s-1)$ and one-bounded functions $f_1, \dots, f_\ell$ and complexity $(M(\delta),m(\delta))$ such that
\begin{align*}
\|\mb{E}_{n\in[N/s]} \wt{f}_{1}(n) \otimes \Theta(n)\cdot \psi(n)  \otimes \prod_{i = 1}^\ell f_i(n) \|_\infty\ge \epsilon(\delta).
\end{align*}
Finally, unravelling the definition of $\wt{f}_1$ and Fourier expanding $1_{[N/s]}$, we are done.
\end{proof}

\section{Constructing the structured extension of the cubic measure}
In this section, we prove Theorem~\ref{thm:ergodicinversetheorem}.
\begin{proof}[Proof of Theorem~\ref{thm:ergodicinversetheorem}.]
\textbf{Step 1: Preliminary maneuvers.} By construction of Host-Kra factors and applying Theorem~\ref{thm:austinergodictheorem}, we see that $L^2(\mathbf{Z})$ generated by dual functions of the form
$$\mathcal{D}f = \lim_{N \to \infty} \mathcal{D}_N f = \lim_{N \to \infty} \mathbb{E}_{h_1, \dots, h_\ell \in [N]} \mathbb{E}_{k_1, \dots, k_{\ell' + 1} \in [N]^k} \prod_{\omega \in \{0, 1\}^{\ell + \ell'} \setminus \{0\}} \vec{T}^{\omega \cdot (e_1h_1, \dots, e_\ell h_\ell, k_1, \dots, k_{\ell' + 1})} f$$
with $f$ one-bounded. By separability\footnote{Recall that we work with the standard assumption that our measure space is \emph{regular} in that it is isomorphic to the set of all Borel sets on a compact metric space.}, we may restrict to a countable sequence $f_n$ such that $\mathcal{D}f_n$ is dense in $L^2(\mathbf{Z})$. To obtain such a factor, we find an appropriate space $\mathbf{Y} = (Y, \mathcal{Y}, \nu)$, measure-preserving transformations $S_1, \dots, S_k$ on $\mathbf{Y}$, and a unital homomorphism $\pi: L^\infty(\mathbf{Z}) \to L^\infty(\mathbf{Y})$ (each equipped with the $L^2$ topology) with $\pi \circ \vec{T} = \vec{S} \circ \pi$ and for which $\pi$ pushes forward the measure of $Y$ to $X$; this would entail finding functions $\widetilde{f}_n$ such that for any multivariate polynomial $P$ and shifts $h_1, \dots, h_n$,
\begin{equation}\label{eq:intertwine}
\int P(\vec{T}^{h_1}\mathcal{D}f_{k_1}, \dots, \vec{T}^{h_n}\mathcal{D}f_{k_n}) d\mu = \int P(\vec{S}^{h_1}\tilde{f}_{k_1}, \dots, \vec{S}^{h_n}\tilde{f}_{k_n}) d\nu.    
\end{equation}
If this exists, the we obtain that the sigma algebra of $\mathbf{Y}$ is a sub-sigma algebra of $\mathbf{Z}$ and hence by classical ergodic theoretic results (e.g., \cite[Lemma 5.25]{EW11}), we obtain a factor map from $L^\infty(\mathbf{Z})$ and $L^\infty(\mathbf{Y})$. The remainder of the proof is devoted to finding such functions $\widetilde{f}_n$ and $(\mathbf{Y}, \vec{S})$.

\textbf{Step 2: Passing to finitary dual functions.} We first normalize $f_n$ so that it has absolute value at most $1$. For any $n, m$, we may find a scale $N_{n, m} \ge m$ such that
$$\|\mathcal{D} f_n - \mathcal{D}_{N_{n, m}} f_n\|_{L^2(\mathbf{X})} \le 2^{-10(m + n)}.$$
We next define
$$E_{n, m} := \{x: |\mathcal{D}f(x) - \mathcal{D}_{N_{n, m}}f(x)| \ge 2^{-(m + n)}\}.$$
This set has measure at most $2^{-99(m + n)} \le 2^{-10(m + n)}$, so the function
$$f = \sum_{n, m} 2^{9(m + n)}1_{E_{m, n}}$$
lies in $L^1(X)$. By the maximal ergodic theorem applied to $f$, we see that for each $\lambda > 0$,
$$f^*(x) = \sup_{N \ge 1} \mathbb{E}_{i \in [\pm N]^k} f \circ \vec{T}^i(x)$$
$$\lambda \mu\{f^* > \lambda \} \le \|f\|_{L^1(\mathbf{X})} := C.$$
Hence, for almost all $x$, there exists some $C_x$ such that $f^*(x) \le C_x$. Note however that for any $m$ and $n$, $|f^*(x)| \ge 2^{9(m + n)}\sup_{H} \frac{\{h \in [\pm H]^k: \vec{T}^h x \in E_{n, m}\}}{(2H)^k}$ and hence, it follows that for almost all $x$, there exists some $C_x$ such that
$$\sup_{H} \frac{|\{h \in [\pm H]^k: \vec{T}^h x \in E_{n, m}\}|}{(2H)^k} \le C_x 2^{-9(m + n)}.$$

\textbf{Step 3: Invoking the Hardy-Littlewood maximal inequality.} We observe from the Cauchy-Schwarz-Gowers inequality that
$$\mathbb{E}_{h \in [\pm N_{n, m}]^2} \mathcal{D}_{N_{n, m}}f_n(\vec{T}^hx)g(h) \ll \|g\|_{U([\pm N_{n, m}]^k, \dots, [\pm N_{n, m}]^k, e_1[\pm N_{n, m}], \dots, e_\ell[\pm N_{n, m}])}$$
for any function $g$ (where there are $\ell' + 1$ copies of $[\pm N_{n, m}]^k$). We now apply (a shifted version of) Theorem~\ref{thm:maininversetheorem2} and Lemma~\ref{lem:strongregularity} with $H = L^2[\pm N_{n, m}]^k$ and $S$ are the set of functions of the form $f_1(h)f_2(h) \cdots f_j(h) \chi(h)$ where $f_i, \chi$ are one-bounded with $f_i$ independent of coordinate $i$ and $\chi$ is a degree $\ell + \ell'$ nilsequence of complexity at most $O_\epsilon(1)$. By splitting up $\chi$ into a sum of $O_\epsilon(1)$ terms, we may assume that $\chi$ is also one-bounded.

Thus, for almost all $x$, there exists a $O_{q, n}(1) = D_{q, n}$-combination of elements in $S$ of the form
$$\chi_{m, n, q, x}(h) = \sum_{i = 1}^{D_{n, q}} (f_1)^i_{m, n, q, x}(h)\cdots (f_j)^i_{m, n, q, x}(h)F^i_{m, n, q, x}(g_{m, n, q, x}(h)\Gamma)$$
such that $F^i_{m, n, q, x}(g_{m, n, q, x}(h)\Gamma)$ is a nilsequence of Lipschitz norm at most $D_{q, n}$ and (suppressing notation)
$$\mathcal{D}_{N_{n, m}}f_n(\vec{T}^hx) = \chi_{m, n, q}(h) + f_{n, \mathrm{sml}}(h) + f_{n, \mathrm{psd}}(h)$$
where $\|f_{n, \mathrm{sml}}\|_{L^2[\pm N_{n, m}]^k} \le 2^{-1000(n + q)}$ and $\|f_{n, \mathrm{psd}}\|_{U(e_1[\pm N_{n, m}], \dots, e_j[\pm N_{n, m}], [\pm N_{n, m}]^k, \dots, [\pm N_{n, m}]^k)} \le \frac{1}{\mathcal{F}(D_{q, n})}$ and $\mathcal{F}$ is a growth function depending on $q$ and $n$ to be chosen later. It follows that
\begin{align*}
\|\mathcal{D}_{N_{n, m}}f_n(\vec{T}^hx) - \chi_{m, n, q}(h)\|_{L^2[\pm N_{n, m}]^k}^2 &\le \|f_{n, \mathrm{psd}} + f_{n, \mathrm{sml}}\|_{L^2[\pm N_{n, m}]^k}^2 \\
&\le 4\|f_{n, \mathrm{psd}}\|_{L^2[\pm N_{n, m}]^2}^2 + 2^{-1000(n + q)} \\
&\le 4(|\langle f_{n, \mathrm{psd}}, \mathcal{D}_{N_{n, m}}f_n \rangle| + |\langle f_{n, \mathrm{psd}}, f_{n, \mathrm{sml}}\rangle| \\
&+ |\langle f_{n, \mathrm{psd}}, \chi_{m, n, q}\rangle| + 2^{-1000(n + q)}) \\
&\le \frac{D_{k, n}^2}{\mathcal{F}(D_{k, n})} + 4 \cdot 2^{-1000(n + q)}
\le 2^{-100(n + q)}    
\end{align*}
if $\mathcal{F}$ is chosen appropriately. Since $L^2[\pm N_{n, m}]^k$ is compact, we may assume that there are finitely many such functions $\chi_{m, n, q}$. By ordering these functions and choosing which of these finitary functions leads to the smallest $L^2[\pm N]^k$ norm (and using the ordering to break ties), we may assume that the selection $x \mapsto \chi_{m, n, q, x}$ is measurable. We wish to make these structured finitary model functions give good $L^2$ approximations on all fixed scales, and not just growing scales $N_{n, m}$. This procedure involves a use of the discrete Hardy-Littlewood maximal inequality.

Let $E'_{m, n, q}$ denote the set of all $x$ such that for every choice of $\chi_{m, n, q}(h)$ with the aforementioned complexity bounds,
$$\sup_{1 \le H} \|\mathcal{D}_{N_{n, m}}f_n(\vec{T}^hx) - \chi_{m, n, q}(h)\|_{L^2[\pm H]^k}^2 \ge 2^{-10(n + q)}.$$
Now take $\chi_{m, n, q}$ be as we have constructed via the regularity procedure. Consider the discrete function $M$ defined via $h \mapsto |\mathcal{D}_{N_{n, m}}f_n(\vec{T}^hx) - \chi_{m, n, q}(h)|^2$ which is supported on $[\pm N_{n, m}]^k$ and extended to be zero outside its domain. By the discrete Hardy-Littlewood maximal theorem, the maximal function $M^*$ satisfies
$$\sup_{t > 0} t \lambda\{M^* > t\} \ll 2^{-100(n + q)}$$
where $\lambda$ is Lebesgue measure. This tells us that the proportion of $g$ such that
$$\sup_{1 \le H}\|\mathcal{D}_{N_{n, m}}f(\vec{T}^{h + g}x) - \chi_{m, n, q}(h + g)\|_{L^2[\pm H]^k}^2$$
is larger than $\gg 2^{-20(n + q)}$ is $\ll 2^{-80(n + q)}$ (so in particular $\ll 2^{-10(n + q)}$). (We note that these implicit constants which we have suppressed depend on $k$.) Noting that
$$\mathbb{E}_{g \in [\pm N_{n, m}]^k} 1(\vec{T}^g x \in E_{m, n, q}') \ll \mathbb{E}_g 1_{G_x}(g) \ll 2^{-10(q + n)},$$
we see that
$$2^{-10(q + n)} \gg \int \mathbb{E}_{g \in [\pm N_{n, m}]^k} 1(\vec{T}^g x \in E_{m, n, q}') dx$$
$$= \mathbb{E}_{g \in [\pm N_{n, m}]^k} \int 1(\vec{T}^g x \in E_{m, n, q}') dx = \mu(E_{m, n, q}').$$
Hence, the functions
$$\mathbb{E}_{m \in [M]} \sum_{q, n} 2^{5(q + n)}1_{E_{m, n, q}'}$$
are uniformly bounded in $L^1$, and hence by Fatou's lemma,
$$\liminf_{M \to \infty} \mathbb{E}_{m \in [M]} \sum_{q, n} 2^{5(q + n)} 1_{E_{m, n, q}'}$$
is in $L^1$. In particular, for almost every $x$, there exists some $C_x'$ such that
$$\liminf_{M \to \infty} \mathbb{E}_{m \in [M]} \sum_{q, n} 2^{5(n + q)}1_{E_{m, n, q}'} \le C'_x$$
and hence for an infinite sequence of $m$'s depending on $x$, we have
$$\sum_{q, n} 2^{5(n + q)} 1_{E_{m, n, q}'} \le C_x'.$$
In particular, there exists some $L_x$ such that for all $q \ge L_x$, all $m$ in that sequence, and all $n$, we have $x \not\in E_{m, n, q}'$. Hence,
$$\sup_{1 \le H} \|\mathcal{D}_{N_{n, m}}f(\vec{T}^h x) - \chi_{m, n, q}(h)\|_{L^2[\pm H]^k}^2 \le 2^{-10(n + q)}$$
for all $q \ge L_x$. Hence, by our earlier observation, we see that
$$\sup_{1 \le H} \|\mathcal{D} \vec{T}^h f_n - \chi_{m, n, q}(h)\|_{L^2[\pm H]^k}^2 \ll 2^{-10(n + q)} + C_x 2^{-9(m + n)}.$$
Hence, for all $H$, we have
$$\limsup_{m \to \infty} \|\mathcal{D} \vec{T}^h f_n - \chi_{m, n, q}(h)\|_{L^2[\pm H]^k}^2 \ll 2^{-10(n + q)}.$$
Now fix some $x$ that is generic and that satisfies all of the above. Thus, we have found ``structured finitary model functions" $\chi_{m, n, q}$ for $\mathcal{D}f_n$ which $L^2$ approximate $\mathcal{D}f_n$ on each scale; the rest of the proof involve ``taking limits to lift these finitary model functions to infinitary model functions." The nilfactor part will be simple enough as the space of all nilsystems is in some sense relatively compact. We follow \cite{Tao15}. To model the additional functions $(f_j)^i_{m, n, q}$, we use the Furstenberg correspondence principle.

\textbf{Step 4: Constructing the pro-nilsystem.} We now construct the inverse limit of the nilsystem. By taking products, we may assume that the nilsequence portions of $\chi_{m, n, q}$ lie on a fixed nilmanifold $G_{m, n, q}/\Gamma_{m, n, q}$. Since there are finitely many nilmanifolds of a given dimension and complexity up to isomorphism (since an isomorphism class of a nilmanifold is specified by its Mal'cev bases data $\mathcal{X}$ so relations between commutators of elements of $\mathcal{X}$, of which there are only finitely many of for a given dimension and complexity, specify the nilmanifold), we may assume that $G_{m, n, q}/\Gamma_{m, n, q}$ is independent of $m$. We now consider the polynomial sequence. Since we passed to a fixed nilmanifold $G_{n, q}/\Gamma_{n, q}$ for each element of the linear combination, we may assume that $g^i_{m, n, q}(\cdot) = g_{m, n, q}(\cdot)$ is independent of $i$. By Lemma~\ref{lem:liftlinear}, we may lift $g_{m, n, q}$ to a linear polynomial sequence $\widetilde{g}_{m, n, q}$ on $\widetilde{G}_{n, q}/\widetilde{\Gamma}_{n, q}$ with $\widetilde{g}_{m, n, q}(0) = \mathrm{id}_{\widetilde{G}_{n, q}}$. By \cite[Lemma 2.1]{Len23b}, we may assume that the coefficients of $\widetilde{g}_{m, n, q}$ (in particular $\widetilde{g}_{m, n, q}(e_i)$) are bounded.

By passing to a subsequence in $m$, we may assume that its coefficients converge, and so $\widetilde{g}_{m, n, q}$ converges to a polynomial sequence $g_{n, q}$ in $\widetilde{G}_{n, q}$. Hence, we pass to a nilsystem $(\widetilde{G}_{n, q}/\widetilde{\Gamma}_{n, q}, \vec{T}_{g_{n, q}})$. Also, by taking products, we may assume that $\widetilde{G}_{n, q}$ is increasing in $q$ and in $n$. By Arzela-Ascoli, we may also assume that each of the Lipschitz functions $F^i_{m, n, q}$ converges to a continuous function independent of $m$. We next pass to the closure of the image of the identity and by the equidistribution theory of nilsequences (e.g., \cite{Lei04, GT12}), it is a subnilmanifold with the action $\vec{T}_{g_{n, q}}$ ergodic. The system $(\widetilde{G}_{n, q}/\Gamma_{n, q}, \vec{T}_{g_{n, q}}, 0)$ admits an ergodic inverse limit $(Z, \vec{U}, 0)$. We write $\widetilde{F}^i_{n, q}$ to be the lift of $F^i_{n, q}$ in $Z$. 

\textbf{Step 5: Invoking the Furstenberg correspondence principle.} We now construct the remaining factors. We note that given $n$, $q$, and $x$ (and fixing $H$), there are $O_{n, q}$ many functions that appear in the collection $((f_{i'})^i_{m, n, q, x})_{i' \in [j], i \in \mathcal{O}}$. By passing to a subsequence in $m$ (and using a diagonalization argument), we may assume that these functions converge pointwise to functions $(f_{i'})^i_{q, n}$ that are independent of $m$. Hence, we see that $\chi_{m, n, q}$ converges pointwise to $\chi_{n, q}$. 

Next, let $\mathcal{F}_{i'}$ to be the set of all $(f_{i'})^i_{q, n}$ and their conjugates, and $\mathcal{F}$ to be the set of all $F^i_{n, q}$ and their conjugates; for $\phi \in \mathcal{F}$, we let $\vec{T}_{g_\phi}$ be the transformation corresponding to $\phi$ and $\widetilde{\phi}$ denote the lift of $\phi$ to $Z$. Let $\widetilde{\mathcal{F}_{i'}}$, $\widetilde{\mathcal{F}}$ be multi-sets where we include each element of $\mathcal{F}_{i'}$, $\mathcal{F}$ $n$ many times for each $n \in \mathbb{N}$, respectively. 


Let $\mathbb{D}$ be the closed unit disk and consider $X_0 = \mathbb{D}^{\mathbb{Z}^{k - 1}}$, and $Y = X_0^{\mathcal{F}_1} \times \cdots \times X_0^{\mathcal{F}_j} \times Z$. Let $\sigma_{\mathcal{F}_{i'}}$ be the action of $\vec{T}^h$ on $X_0$ by omitting the $i'$th coordinate of $h$. Define $\vec{S}^h$ on $Y$ to be $\sigma_{\mathcal{F}_{i'}}^{h}$ on each of the $\mathcal{F}_{i'}$ components, and $\vec{U}^h$ on the $Z$ component. Now define $x_0 = ((\alpha_1)_{\alpha_1 \in \mathcal{F}_1}, \dots, (\alpha_j)_{\alpha_j \in \mathcal{F}_j}, 0)$, for each $\alpha_{i'} \in \mathcal{F}_{i'}$, we define $\widetilde{\alpha_{i'}} \in C(Y)$ via $\widetilde{\alpha_{i'}}(x) = (x_0)_{\alpha_{i'}}(0)$ and for each $\phi \in \mathcal{F}$, $\widetilde{\phi}$ naturally embeds in $C(Y)$. Consider the sequence of measures
$$\nu_p := \mathbb{E}_{h \in [\pm H_p]^k} \delta_{\vec{S}^{h}x_0}.$$
By the Banach-Alaoglu theorem, a subsequence of $\nu_q$ converges weak star to a measure $\nu$. Hence, passing to such a subsequence, we may assume that for any polynomial $P$ with rational coefficients and shifts $((h_{\alpha_i})_{\alpha_i \in \mathcal{F}_i}, (h_\phi)_{\phi \in \mathcal{F}})$ with all but finitely many zero,
$$\lim_{p \to \infty} \mathbb{E}_{h \in [\pm H_p]^k} P((\vec{T}^{h_{\alpha_1}}\alpha_1(h))_{\alpha_1 \in \widetilde{\mathcal{F}_1}}, \dots, (\vec{T}^{h_{\alpha_j}}\alpha_j(h))_{\alpha_j \in \widetilde{\mathcal{F}_j}}, (\vec{T}_{g_\phi}^{h_\phi}\phi(h))_{\phi \in \widetilde{\mathcal{F}}})$$
$$= \int P((\vec{S}^{h_{\alpha_1}}\widetilde{\alpha_1})_{\alpha_1 \in \widetilde{\mathcal{F}_1}}, \dots, (\vec{S}^{h_{\alpha_j}}\widetilde{\alpha_j})_{\alpha_j \in \widetilde{\mathcal{F}_j}}, (\vec{S}_{g_\phi}^{h_\phi}\widetilde{\phi})_{\phi \in \widetilde{\mathcal{F}}}) d\nu.$$
Now let $\mathcal{Y}$ be the Borel sigma algebra on $Y$ and $\mathbf{Y} = (Y, \mathcal{Y}, \nu)$.

\textbf{Step 6: Completing the proof.} The functions $\chi_{n, q}$ lift to $\widetilde{\chi}_{n, q}$ on $L^\infty(\mathbf{Y})$. Taking a limit as $m$ goes to infinity, we have
\begin{equation}\label{eq:finitarymodelapproximation}
\|\mathcal{D} \vec{T}^h f_n - \chi_{n, q}(h)\|_{L^2[\pm H_p]^k}^2 \ll 2^{-10(n + q)}.
\end{equation}
Thus, we have for $q' \ge q$,
\[\|\chi_{n, q'}(h) - \chi_{n, q}(h)\|_{L^2[\pm H_p]^k}^2 \ll 2^{-10(n + q)}. \]
Taking a limit as $p$ goes to infinity, we have
$$\int |\widetilde{\chi}_{n, q'} - \widetilde{\chi}_{n, q}|^2 d\nu \ll 2^{-10(n + q)}.$$
Thus, $\widetilde{\chi}_{n, q}$ forms a Cauchy sequence in $q$; hence it converges to $\widetilde{\chi}_n \in L^\infty(Y, \nu, \vec{S})$. Let $\widetilde{f}_n = \widetilde{\chi}_n$. We see by applying the Cauchy-Schwarz inequality, the triangle inequality, taking limit as $p$ goes to infinity, and then taking a limit as $q$ goes to infinity in \eqref{eq:finitarymodelapproximation} that
\[\int_Y \widetilde{f}_n d\nu = \int_X \mathcal{D} f_n d\mu.\]
We observe that \eqref{eq:intertwine} holds by observing that the proof of \eqref{eq:finitarymodelapproximation} holds for $\mathcal{D}f_n$ replaced with 
$$P(\vec{T}^{h_1}\mathcal{D}f_{m_1}, \dots, \vec{T}^{h_n}\mathcal{D} f_{m_n})$$ 
and $\chi_{n}(h)$ replaced with $P(\vec{S}^{h_1}\chi_{m_1}, \dots, \vec{S}^{h_n}\chi_{m_n})(h)$ for any polynomial $P$ and any shifts $h_1, \dots, h_n$. Finally, we must ensure ergodicity of the system. By taking a disintegration $\nu = \int \nu_\omega d\omega$ with respect to $I(\vec{S})$ (which is possible since $\mathbf{Y}$ is regular), we see via a pushforward map that this gives a disintegration of $(X, \vec{T}, \mu, \mathcal{Z})$. It follows by ergodicity that almost every measure of the disintegration agrees with $\mu$ and hence we may choose one of them as a substitute for $\nu$. This completes the proof.
\end{proof}

\section{Obtaining the nil + null decomposition}
For the purposes of using Theorem~\ref{thm:ergodicinversetheorem}, we require the following lemma.
\begin{lemma}\label{lem:joiningconstruction}
Let $\mathbf{X}$ be a probability space with factors $\mathbf{Z}_1, \dots, \mathbf{Z}_k$ and let $\mathbf{E}_i$ be extensions of $\mathbf{Z}_i$ for $1 \le i \le k$. Then there exists a measure $\nu$ on the underlying space of $\mathbf{E}_1 \times \mathbf{X}\times \cdots \times \mathbf{E}_k$ such that given $f_0 \in L^\infty(\mathbf{X})$, $f_i \in L^\infty(\mathbf{E}_i)$,
$$\int_{\mathbf{X} \times \mathbf{E}_1 \times \cdots \times \mathbf{E}_k} f_0 \otimes f_1 \otimes f_2 \otimes \cdots \otimes f_k d\nu = \int f_0 \cdot \mathbb{E}(f_1|\mathbf{Z}_1) \cdots \mathbb{E}(f_k |\mathbf{Z}_k) d\mu_{\mathbf{X}}.$$
\end{lemma}
\begin{proof}
We give an inductive construction. We define $\nu^1 = \mu_{\mathbf{E}_1} \times_{\mathbf{Z}_1} \mu_{\mathbf{X}}$ and for $j > 1$, $\nu^j = \nu^{j - 1} \times_{\mathbf{X}} (\mu_{\mathbf{E}_j} \times_{\mathbf{Z}_j} \mu_{\mathbf{X}}).$ This gives a measure on the underlying space of $(\mathbf{E}_1 \times \mathbf{X}) \times \prod_{j = 2}^k (\mathbf{E}_j \times \mathbf{X})$. We now define $\nu$ to be the projection of $\nu^k$ to $(\mathbf{E}_1 \times \mathbf{X} \times \prod_{j = 2}^k \mathbf{E}_j)$. Let $\mu_i = \mu_{\mathbf{E}_i} \times_{\mathbf{Z}_i} \mu_{\mathbf{X}}$. Then for $h \in L^\infty(\mathbf{X})$,
$$\mathbb{E}_{\mu_i}(f_i \otimes h| \mathbf{X}) = \mathbb{E}(f_i|\mathbf{Z}_i) h$$
since if $g \in L^\infty(\mathbf{X})$, we see that
$$\int (f_i \otimes h) \cdot(1 \otimes g) d\mu_i = \int \mathbb{E}(f_i|\mathbf{Z}_i) \cdot \mathbb{E}( gh|\mathbf{Z}_i) d\mu_{\mathbf{Z}_i} =\int \mathbb{E}(f_i|\mathbf{Z}_i) \cdot h \cdot g d\mu_{\mathbf{X}}.$$
From this, we claim by induction that
$$\mathbb{E}_{\nu^{\ell}}(f_1 \otimes f_0 \otimes f_2 \otimes 1 \otimes f_3 \otimes 1 \otimes \cdots \otimes f_{\ell} \otimes 1|\mathbf{X}) = f_0 \cdot \mathbb{E}(f_1 |\mathbf{Z}_1) \cdots \mathbb{E}(f_\ell |\mathbf{Z}_\ell).$$
Let $F = f_1 \otimes f_0 \otimes f_2 \otimes 1 \otimes f_3 \otimes 1 \otimes \cdots \otimes f_{\ell} \otimes 1$. Assuming that this claim holds for $\ell - 1$,
\begin{align*}
\int F \cdot (1 \otimes \cdots \otimes g) d\nu^\ell &= \int  f_1 \otimes f_0 \otimes f_2 \otimes 1 \otimes f_3 \otimes 1 \otimes \cdots \otimes f_{\ell} \otimes g d\nu^\ell \\
&= \int \mathbb{E}_{\nu^{\ell - 1}}(f_1 \otimes f_0 \otimes f_2 \cdots \otimes f_{\ell - 1} \otimes 1|\mathbf{X}) \cdot \mathbb{E}_{\mu_\ell}(f_\ell \otimes g|\mathbf{X}) d\mu_{\mathbf{X}} \\
&= \int \mathbb{E}_{\nu^{\ell - 1}}(1 \otimes \cdots \otimes f_0 \cdot\mathbb{E}(f_1 |\mathbf{Z}_1) \cdots \mathbb{E}(f_{\ell - 1} |\mathbf{Z}_{\ell - 1})|\mathbf{X}) \mathbb{E}_{\mu_\ell}(1 \otimes \mathbb{E}(f_\ell|\mathbf{Z}_\ell) g|\mathbf{X}) d\mu_\mathbf{X} \\
&= \int f_0 \cdot \mathbb{E}(f_1|\mathbf{Z}_1)\cdots \mathbb{E}(f_{\ell - 1}|\mathbf{Z}_{\ell - 1}) \cdot \mathbb{E}(f_\ell|\mathbf{Z}_\ell)g d\mu_\mathbf{X}.
\end{align*}
(Note that the projection of $\nu^\ell$ and $\mu_\ell$ to their final factors agree with $\mu_{\mathbf{X}}$.) This gives the desired claim. Finally, the lemma follows from the last line of the above calculation for $g \equiv 1$, $\ell = k$.
\end{proof}

We also require a twisted generalized von Neumann-type theorem.
\begin{prop}\label{prop:twistedgeneralizedvonneumann}
Let $\mathbf{Z} = (Z, \nu, T_1, \dots, T_k, \mathcal{Z})$ be an ergodic $\mathbb{Z}^k$ system with $k$-step nilfactors $\mathbf{\Xi_1}, \dots, \mathbf{\Xi_k}$. Let $\chi_i\in L^\infty(\mathbf{\Xi_i})$ be smooth vertical characters, $\epsilon > 0$, and $f_0, f_1, f_2, \dots, f_j \in L^\infty(\mathbf{Z})$. Then there exists a constant $C_{\chi_1, \dots, \chi_k, \epsilon} > 0$ such that
\begin{equation}\label{eq:gvn}
\lim_{N \to \infty} \mathbb{E}_{n \in [\pm N]} \left|\int f_0\cdot T_1^n\chi_1 \cdots T_k^n \chi_k \cdot T_1^n f_1 \cdots T_j^n f_j d\nu\right|^2 \le \min_{i \ge 1} C_{\chi_1, \dots, \chi_k, \epsilon}\|f_i\|_{\vec{T}, \dots, \vec{T}, T_i, T_1T_i^{-1}, \dots, T_jT_i^{-1}}^2 + \epsilon    
\end{equation}
where there are $k - j + 1$ copies of $\vec{T}$ and $T_1T_i^{-1}, \dots, T_jT_i^{-1}$ ranges over all $\ell$ for $\ell \neq i$.
\end{prop}
\begin{remark}
The limit on the left hand side exists by \cite{Tao08, A09}.
\end{remark}
\begin{proof}
Let $\Lambda$ denote the left hand side of \eqref{eq:gvn}. It suffices to show seminorm control for $f_j$. We also apply induction on $j$. By the mean ergodic theorem, we may write
$$\int f_0\cdot T_1^n\chi_1 \cdots T_k^n \chi_k \cdot T_1^n f_1 \cdots T_j^n f_j d\nu = \lim_{M \to \infty} \mathbb{E}_{m \in [\pm M]^k} \vec{T}^m(f_0\cdot T_1^n\chi_1 \cdots T_k^n \chi_k \cdot T_1^n f_1 \cdots T_j^n f_j)$$
where we are taking an $L^2$ limit. Thus,
$$\Lambda \le \lim_{N \to \infty} \lim_{M \to \infty} \mathbb{E}_{n \in [\pm N]} \mathbb{E}_{m \in [\pm M]^k} \int \Delta_{\vec{T}^m} (f_0\cdot T_1^n\chi_1 \cdots T_k^n \chi_k \cdot T_1^n f_1 \cdots T_j^n f_j) d\nu.$$
Since $\chi_i$ are vertical characters, $\Delta_{\vec{T}^m}\chi_i$ can be realized as smooth functions on degree $k - 1$ nilsystems by Lemma~\ref{lem:differentiation}. While these nilsystems may not necessarily have complexity bounded uniformly in $m$, Lemma~\ref{lem:differentiation} implies that these nilsystems all are subnilsystems of a common nilsystem and crucially, the underlying smooth functions of $\Delta_{\vec{T}^m}\chi_i$ on the degree $k -  2$ nilsystem are in fact restrictions of functions on the common degree $k - 1$ nilmanifold with uniformly bounded (in $m$) Lipschitz norm. Thus, by Lemma~\ref{lem:nilcharacters}, we may Fourier approximate each $\Delta_{\vec{T}^m}\chi_i$ up to $\epsilon$ error into degree $k - 1$ nilcharacters with respect to the common nilmanifold, and pigeonholing in one of the Fourier coefficients $\chi_{i, m}$ and suppressing $\epsilon$, we have
$$\Lambda \ll \lim_{N \to \infty} \lim_{M \to \infty} \mathbb{E}_{n \in [\pm N]}\mathbb{E}_{m \in [\pm M]^k} \int \Delta_{\vec{T}^m}(f_0\cdot T_1^n f_1 \cdots T_j^n f_j) \cdot T_1^n \chi_{1, m} \cdots T_k^n\chi_{k, m} d\nu.$$
Applying the Cauchy-Schwarz (with one of the functions being $\Delta_{\vec{T}^m}f_0$) and van der Corput inequlities, we obtain
$$\Lambda^2 \ll \limsup_{N \to \infty} \limsup_{M \to \infty} \mathbb{E}_{n \in [\pm N]} \mathbb{E}_{m \in [\pm M]^k} \mathbb{E}_{n' \in [\pm N]} \mu_{N'}(n') \int  T_1^n\Delta_{T_1^{n'}, \vec{T}^m}(f_1) \cdots T_k^n\Delta_{T_j^{n'}, \vec{T}^m}(f_j)$$
$$T_1^n\Delta_{T_1^{n'}}(\chi_{1, m}) \cdots (T_k)^n\Delta_{T_k^{n'}}\chi_{k, m} d\nu$$
where $\mu_N(n) = 2_{[\pm (2N - 1)]}(n) \cdot \left(1 - \frac{|n|}{2N - 1}\right)$ is the Fej\'er kernel. Now applying $T_1^{-n}$, we obtain
$$\Lambda^2 \ll \limsup_{N \to \infty} \limsup_{M \to \infty} \mathbb{E}_{n \in [\pm N]} \mathbb{E}_{m \in [\pm M]^k} \mathbb{E}_{n' \in [\pm N]} \mu_{N}(n') \int  \Delta_{T_1^{n'}, \vec{T}^m}(f_1) \cdots (T_kT_1^{-1})^n\Delta_{T_j^{n'}, \vec{T}^m}(f_j)$$
$$\Delta_{T_1^n}(\chi_{1, m}) \cdots (T_kT_1^{-1})^n\Delta_{T_k^n}(\chi_{k, m}) d\nu.$$
We now Fourier expand $\Delta_{T_i^{n'}}\chi_{1, m}$ and pigeonholing in a Fourier coefficient to a degree $k - 2$ nilsequence $\chi_{1, m, n'}$ (with complexity and Lipschitz norm in an ambient nilmanifold uniform in $m$ and $n'$), we obtain a bound on the left hand side of \eqref{eq:gvn} by
$$\Lambda^2 \ll \lim_{N \to \infty} \limsup_{M \to \infty} \mathbb{E}_{n \in [\pm N]} \mathbb{E}_{m \in [\pm M]^k} \mathbb{E}_{n' \in [\pm N]} \mu_{N}(n') \int  \Delta_{T_1^{n'}, \vec{T}^m}(f_1) \cdots (T_kT_1^{-1})^n$$
$$\Delta_{T_j^{n'}, \vec{T}^m}(f_j) \cdot \chi_{1, m, n'} \cdots (T_kT_1^{-1})^n\chi_{k, m, n'} d\nu.$$
Now, applying this procedure $j - 1$ more times (and suppressing $\epsilon$ each time), we eventually obtain
$$\Lambda^{2^j} \ll \limsup_{N \to \infty} \limsup_{M \to \infty} \mathbb{E}_{m \in [\pm M]^k} \mathbb{E}_{n \in [\pm N]} \mathbb{E}_{h_1, \dots, h_{j} \in [\pm N]} \mu_{N}(\vec{h}) \int (T_{j + 1}T_j^{-1})^n\chi_{j + 1, m, \vec{h}}\cdots (T_kT_j^{-1})^n\chi_{k, m, \vec{h}} $$
$$\Delta_{\vec{T}^m, T_j^{h_1}, (T_jT_1^{-1})^{h_2}, \dots, (T_jT_{j - 1}^{-1})^{h_{j}}}(f_j) d\nu.$$
Now Fourier expanding $\chi_{i, m, \vec{h}}$ with respect to the vertical torus of the common ambient nilmanifold of bounded dimension and complexity and applying the Cauchy-Schwarz inequality $k - j - 1$ times while suppressing $\epsilon$ again (and also using the mean ergodic theorem for $\vec{T}$), we eventually obtain (see also the proof of Lemma~\ref{lem:converse})
$$\Lambda^{2^{k - 1}} \ll \limsup_{N \to \infty} \limsup_{M, M_1, \dots, M_{k - j - 1} \to \infty} \mathbb{E}_{n, h_1, \dots, h_j \in [\pm N]} \mathbb{E}_{m \in [\pm M], k_i \in [\pm M_i]} \mu_N(\vec{h}) \int \Delta_{\vec{T}^{k_{k - j - 1}}} \chi_{n, m, \vec{h}, k_1, \dots, k_{k - j - 1}}$$
$$\Delta_{\vec{T}^m, \dots \vec{T}^{k_{k - j - 1}}, T_j^{h_1}, (T_jT_1^{-1})^{h_2}, \dots, (T_jT_{j - 1}^{-1})^{h_{j}}}(f_j) d\nu$$
where
$$\chi_{n, m, \vec{h}, k_1, \dots, k_{k - j - 1}} = (T_{j + 1}T_j^{-1})^n\chi_{j + 1, m, \vec{h}, k_1, \dots, k_{k - j - 1}}\cdots (T_kT_j^{-1})^n\chi_{k, m, \vec{h}, k_1, \dots, k_{k - j - 1}}.$$
Note that $\chi_{n, m, \vec{h}, k_1, \dots, k_{k - j - 1}}$ is constant, so applying the Cauchy-Schwarz inequality and using the mean ergodic theorem for $\vec{T}$ once more, we obtain a bound by
$$\Lambda^{2^{k}} \ll \limsup_{N \to \infty} \mathbb{E}_{h_1, h_2, \dots, h_j \in [\pm N]} \mu_N(\vec{h}) \|\Delta_{T_j^{h_1}, (T_jT_1^{-1})^{h_2}, \dots, (T_jT_{j - 1}^{-1})^{h_{j - 1}}}(f_j)\|_{\vec{T}, \dots, \vec{T}}^{2^{k - j + 1}}$$
$$\ll \limsup_{N \to \infty} \mathbb{E}_{h_1, h_2, \dots, h_j \in [\pm N]}\|\Delta_{T_j^{h_1}, (T_jT_1^{-1})^{h_2}, \dots, (T_jT_{j - 1}^{-1})^{h_{j - 1}}}(f_j)\|_{\vec{T}, \dots, \vec{T}}^{2^{k - j}}.$$
By Theorem~\ref{thm:austinergodictheorem} (or rather the remark after that theorem), the above limit is equal to 
$$\|f_j\|_{T_1T_j^{-1}, T_2T_j^{-1}, \dots, T_{j - 1}T_j^{-1}, \vec{T}, \dots, \vec{T}}^{2^{k + 1}}.$$
The proposition follows from taking a $2^{k}$-th root.
\end{proof}

We are now ready to prove the main theorem.
\begin{proof}[Proof of Theorem~\ref{thm:maintheorem}]
We give an inductive argument. We show that for any probability space $\mathbf{X} = (X, \mathcal{X}, \mu)$ and commuting measure preserving transformations $T_1, \dots, T_k$ on $\mathbf{X}$, and $\chi_1, \dots, \chi_k$ that are lifts of measureable functions on nilfactors of degree at most $k$, the following may be written as a sum of a degree $k$ nilsequence and a null-sequence.
\begin{equation}\label{eq:multicorrelation}
\int f_0\cdot T_1^n(\chi_1 f_1) \cdots T_j^n (f_j \chi_j) \cdots T_k^n\chi_k d\mu
\end{equation}
We do not lose any generality by assuming that $\chi_1 = \cdots = \chi_j \equiv 1$. The base case of $j = 0$ follows from Lemma~\ref{lem:Leibmanresult2}. We now prove this for $j$ assuming $j - 1$. Theorem~\ref{thm:maintheorem} follows from this by setting $j = k$. By Proposition~\ref{prop:Leibmanresult}, the ergodic decomposition theorem, and the fact that every ergodic component of a nilfactor is still a nilsystem (which follows from \cite{Lei04}), we may assume that $\vec{T}$ is ergodic. By smooth approximation and permitting an error of a uniformly bounded sequence of arbitrarily small $\ell^\infty(\mathbb{Z})$ norm, we may assume that $\chi_i$ are smooth.

By Proposition~\ref{prop:twistedgeneralizedvonneumann}, we see that \eqref{eq:multicorrelation} is $\le C_{\epsilon, \chi_1, \dots, \chi_k} \min_i \|f_i\|_{T_i, (T_1T_i^{-1}), \cdots (T_j T_i^{-1}), \vec{T}, \dots, \vec{T}}^2 + \epsilon$ for every $\epsilon > 0$. Thus, letting $\mathbf{E}_i$ be the extension in Theorem~\ref{thm:ergodicinversetheorem} of $\mathbf{Z}_i = \mathbf{Z}_{T_i, (T_1T_i^{-1}), \cdots (T_j T_i^{-1}), \vec{T}, \dots, \vec{T}}$, we see that at a cost of a null-sequence, we may replace $f_i$ with its conditional expectation on $\mathbf{Z}_{i}$, and by passing to an extension, realize $f_i \in L^\infty(\mathbf{E}_i)$. Next, we use Lemma~\ref{lem:joiningconstruction} to obtain a measure $\nu$ on the underlying space of $\mathbf{X} \times \mathbf{E}_1 \times \cdots \times \mathbf{E}_j$ such that
$$\int b_0 \otimes b_1 \otimes \cdots \otimes b_j d\nu = \int b_0\cdot \mathbb{E}(b_1|\mathbf{Z}_1) \cdots \mathbb{E}(b_j|\mathbf{Z}_j) d\mu$$
for $b_i \in L^\infty(\mathbf{E}_i)$ and $b_0 \in L^\infty(\mathbf{X})$. In addition, for $0 \le i \le j$ we replace $f_i$ with $1 \otimes 1 \cdots \otimes f_i \otimes \cdots \otimes 1$, and for each $i'$, $\chi_{i'}$ by $\chi_{i'} \otimes 1 \otimes \cdots \otimes 1$. Note that we have the invariant action of $\mathbb{Z}^k$ on $\nu$ via the diagonal action $T_i^\Delta$. By abuse of notation, we shall shorthand $T_i^\Delta$ as $T_i$.

By Lemma~\ref{lem:inverselimitapproximation} and Lemma~\ref{lem:inversetheoremjoining} (and noticing that any element in $L^2(\mathbf{X})$ and in $L^2(\mathbf{E}_i)$ bounded by $\epsilon$ is also bounded by $\epsilon$ in $L^2(\nu)$), up to a uniformly bounded sequence by an arbitrarily small parameter, we may replace $f_i$ with a product
$$g_i = \chi_i'\cdot \prod_{i' \neq i} b_{T_{i'}T_i^{-1}}$$
with $\chi_i'$ an extension of a smooth function on a degree $k$ nilmanifold and for a transformation $S$, we denote $b_S$ a generic $1$-bounded function invariant under $S$. By Fourier expanding $\chi_i$ via Lemma~\ref{lem:nilcharacters}, again at the cost of a uniformly bounded sequence by an arbitrarily small parameter, we may assume that $\chi_i'$ are lifts of vertical characters on the nilmanifold. We show that
\begin{equation}\label{eq:appliedinversetheorem}
\int f_0\cdot T_1^n( g_1) \cdots T_j^n (g_j) \cdots T_k^n\chi_k d\nu
\end{equation}
is the sum of a degree $k$ nilsequence, a null-sequence, and a sequence uniformly bounded by an arbitrarily small parameter. Invoking Lemma~\ref{lem:standard}, we obtain a decomposition of \eqref{eq:appliedinversetheorem} into a nilsequence and a null-sequence, and invoking Lemma~\ref{lem:standard} again (since we got a uniformly bounded error when we passed to smooth $\chi_i$), we obtain a decomposition of \eqref{eq:multicorrelation} into a nilsequence and a null-sequence as well. Now write $T_j^n g_j = T_j^n \chi_{j}' \prod_{i' \neq j} T_j^n b_{T_{i'}T_j^{-1}} = T_j^n \chi_j' \cdot \prod_{i' \neq j} T_{i'}^n b_{T_{i'}T_j^{-1}}$. Thus, letting $h_i = g_i\cdot  b_{T_i T_j^{-1}}$, we have
\begin{equation}
\int f_0\cdot T_1^n(h_1) \cdots T_{j - 1}^n (h_{j - 1}) \cdot T_j^n(\chi_j \chi_j') \cdots T_k^n\chi_k d\nu.
\end{equation}
The theorem follows via induction.
\end{proof}

\appendix

\section{Auxiliary Lemmas}
The first two lemmas we prove are lemmas involving nilsequences. As many of the lemmas are stated with quantitative bounds but we only need qualitative versions of them in Sections 9-10, the reader interested only in those sections can safely ignore these quantitative bounds on first reading. We set the following convention.
\begin{notation}
Let $G$ be a group. For $g_1, \dots, g_t \in G$ and $n_1, \dots, n_t \in \mathbb{Z}$, we denote $\vec{g}^{\vec{n}} = g_1^{n_1}g_2^{n_2} \cdots g_t^{n_t}$.
\end{notation}
We require the following lemma for the proof of Proposition~\ref{prop:twistedgeneralizedvonneumann}.
\begin{lemma}\label{lem:differentiation}
Let $G/\Gamma$ be a degree $k$ nilmanifold of complexity $M$, dimension $m$ and $F: G/\Gamma \to \mathbb{C}$ be a $G_k$-vertical character with Lipschitz norm $K$. Let $G'/\Gamma'$ be a subnilmanifold of $G/\Gamma$. Consider the nilsystem $\mathbf{G}' = (G'/\Gamma', \mu, T_{g_1}, \dots, T_{g_\ell})$ where $T_g$ is translation by $g$ with $g_i$ all commuting with each other. Let $\vec{g}^{\vec{t}} = g_1^{t_1}g_2^{t_2} \cdots g_\ell^{t_\ell}$. Then for any $\vec{h} \in \mathbb{Z}^\ell$ and $x \in G$, the function $F(\vec{g}^{\vec{h}}x\Gamma)\overline{F(x\Gamma)}$ lifts from a function $\widetilde{F_{\vec{h}}}$ on a degree $k - 1$ (\textbf{$h$-dependent}) nilfactor of $\mathbf{G}'$ whose underlying nilmanifold is a subnilmanifold of an \textbf{$h$-independent} $\overline{G^\square}/\overline{\Gamma^\square}$ of complexity $M^{O(m^{O(1)})}$, dimension $O(m)$, and such that the lift from $\widetilde{F_h}$ to $\overline{G^\square}/\overline{\Gamma^\square}$ has Lipschitz constant at most $(MK)^{O(m^{O(1)})}$. (Note that we are not making any assumptions on the dimension or complexity of $G'/\Gamma'$.)
\end{lemma}
\begin{remark}
The hypotheses and conclusions of this lemma may seem rather contrived. The lemma is stated as such purely because of its application to Proposition~\ref{prop:twistedgeneralizedvonneumann}.
\end{remark}
\begin{proof}
Following the notation in \cite{Len23b}, we write for an element $g \in G$ as $g = \{g\}[g]$ where all of the Mal'cev coordinates of $[g]$ are integers and all of the Mal'cev coordinates of $\{g\}$ lie between $-1/2$ and $1/2$. We write 
\begin{align*}
\vec{g}^{\vec{n} + \vec{h}}x\Gamma &= \vec{g}^{\vec{n}} \vec{g}^{\vec{h}}x\Gamma = \vec{g}^{\vec{n}} \vec{g}^{\vec{h}} x [\vec{g}^{\vec{h}}]^{-1}\Gamma = \vec{g}^{\vec{n}}\{\vec{g}^{\vec{h}}\}[\vec{g}^{\vec{h}}] x[\vec{g}^{\vec{h}}]^{-1}\Gamma \\
&= \{\vec{g}^{\vec{h}}\} \{\vec{g}^{\vec{h}}\}^{-1}\vec{g}^{\vec{n}}\{\vec{g}^{\vec{h}}\} [\vec{g}^{\vec{h}}]x[\vec{g}^{\vec{h}}]^{-1}\Gamma = \{\vec{g}^{\vec{h}}\}(\{\vec{g}^{\vec{h}}\}^{-1}\vec{g}\{\vec{g}^{\vec{h}}\})^{\vec{n}} [\vec{g}^{\vec{h}}] x[\vec{g}^{\vec{h}}]^{-1}\Gamma.
\end{align*}
Define $G^\square = G \times_{G_2} G = \{(g, g'): g'g^{-1} \in G_2\}$. Let
$$F_{\vec{h}}(x, y) = F(x)\overline{F(\{g^{\vec{h}}\}y)}, g_{\vec{h}}(n) = (g^nx, (\{\vec{g}^{\vec{h}}\}^{-1}\vec{g}\{\vec{g}^{\vec{h}}\})^{\vec{n}} [\vec{g}^{\vec{h}}] x[\vec{g}^{\vec{h}}]^{-1})$$
By \cite[Lemma A.3]{Len23b}, $F_{\vec{h}}$ descends to $\overline{G^\square} := G^\square/G_k^\triangle$ where $G_k^\triangle$ is the diagonal group in $G^\square$, and $\overline{G^\square}$ is degree $k - 1$ and has the desired dimension and complexity with $F_{\vec{h}}$ having the desired Lipschitz norm by that same lemma. We now define the map $\phi_{\vec{h}}: G/\Gamma \to G^\square/\Gamma^\square$ induced by $x \mapsto \widetilde{x} = (x, [g^{\vec{h}}]x[g^{\vec{h}}]^{-1})$ and equip $G^\square$ with the translations by $\widetilde{g_i} = (g_i, \{\vec{g}^{\vec{h}}\}^{-1}g_i\{\vec{g}^{\vec{h}}\})$. The image of this mapping is a translation on a subnilmanifold of $G^\square/\Gamma^\square$. By our computation earlier, we see that the functions $x \mapsto F(\vec{g}^{\vec{n} + \vec{h}}x\Gamma)\overline{F(\vec{g}^{\vec{n}}x\Gamma)}$ and $F_{\vec{h}}(\widetilde{\vec{g}}^{\vec{n}}\widetilde{x})$ are equal. Now defining $\pi_{\vec{h}}: G/\Gamma \to \overline{G^\square}/\overline{\Gamma^\square}$ via $\pi_{\vec{h}} = \overline{\phi_{\vec{h}}}$ and restricting to $G'/\Gamma'$, we see that $G'/\Gamma'$ maps to a sub-nilmanifold $H/\Lambda$ of $\overline{G^\square}/\overline{\Gamma^\square}$. This completes the proof. 
\end{proof}
We require the statement that all nilsequences can be lifted to linear nilsequences of a slightly larger nilmanifold. This is essentially \cite[Proposition C.2]{GTZ12} and our proof follows their proof closely.

\begin{lemma}\label{lem:liftlinear}
Let $F(g(\cdot)\Gamma)$ be a $\mathbb{Z}^k$ nilsequence of complexity and Lipschitz norm $\le M$, degree $\ell$, and dimension $m$ with $G/\Gamma$ the underlying nilmanifold. Then there exists a nilmanifold $\widetilde{G}/\widetilde{\Gamma}$ with complexity $M^{m^{O_{k, \ell}(1)}}$ and dimension $m^{O_{k, \ell}(1)}$, a Lipschitz function $\widetilde{F}: \widetilde{G}/\widetilde{\Gamma}$ with norm at most $M^{m^{O_{k, \ell}(1)}}$, and elements $\widetilde{g}_1, \dots, \widetilde{g}_k$ such that $F(g(n)\Gamma) = \widetilde{F}(\vec{\widetilde{g}}^{\vec{n}}\widetilde{\Gamma})$.
\end{lemma}
\begin{proof}
Consider the polynomial sequence $\mathrm{poly}(\mathbb{R}^k, G)$ and the action $\mathbb{R}^k$ on the space by left translation. Using this action, we define the semidirect product $\widetilde{H} = \mathrm{poly}(\mathbb{Z}^k, G) \ltimes \mathbb{R}^k$ and $\widetilde{\Lambda} = \mathrm{poly}(\mathbb{Z}^k, \Gamma) \ltimes \Gamma$. (Note that $\widetilde{H}$ is $\ell + 1$\emph{-step}.) Next, define $\widetilde{h}_i \in \widetilde{H}$ via $\widetilde{x} = (g, \vec{0})$, $\widetilde{h}_i = (\mathrm{id}, e_i)$. Here, $e_i$ denotes the standard vectors in $\mathbb{R}^k$. Let $\widetilde{G} = \mathrm{poly}(\mathbb{Z}^k, G^{+1}) \ltimes \mathbb{R}^k$ where $G^{+1}$ denotes the shifted filtration of $G_i^{+1} = G_{i + 1}$. We observe that 
$$\vec{\widetilde{h}}^{\vec{n}} \widetilde{x} = \widetilde{x} (\widetilde{x}^{-1} \vec{\widetilde{h}} \widetilde{x})^{\vec{n}}$$
and hence the orbit lies on a coset of $\widetilde{G}$, which is \emph{$\ell$-step}. Defining $\widetilde{g}_i = \widetilde{x}^{-1} h_i \widetilde{x}$, and $\widetilde{F}((p, \vec{v})\widetilde{\Gamma}) = F(g(0) \cdot p(0)\Gamma)$, we see that the desired property of
$$F(g(n)\Gamma) = \widetilde{F}(\vec{\widetilde{g}}^{\vec{n}}\widetilde{\Gamma})$$
holds. The quantitative bounds follow via standard computations (e.g., \cite[Proposition C.2]{LSS24}).
\end{proof}
We next require the following factorization result, which is essentially a multidimensional version of \cite[Lemma B.2]{LSS24b} (which in turn is a generalization of \cite[Lemma A.1]{Len23b}). 
\begin{lemma}\label{lem:factor}
Let $s, r, D' \le K$ be positive integers. Consider a nilmanifold $G/\Gamma$ of degree-rank $(s,r)$ of dimension $d$ and complexity $M$. Consider a polynomial sequence $g: \mathbb{Z}^{D'} \to G$ such that $g(0) = \mr{id}_G$ and consider a set of horizontal characters $\psi_{i,j}$ for $1\le j\le\ell_i$ and where $\psi_{i,\cdot}$ is an $i$-th horizontal character of height at most $H$. Furthermore suppose that for all $\vec{i} \in \mathbb{Z}^{D'},j$, 
\[\on{dist}(\psi_{|\vec{i}|,j}(\on{Taylor}_{\vec{i}}(g)), \mb{Z})\le H\cdot N^{-|\vec{i}|}.\]
Then one may factor
\[g = \eps \cdot g' \cdot\gamma\]
where:
\begin{itemize}
    \item $\eps(0) = g'(0) = \gamma(0) = \mr{id}_G$;
    \item $\psi_{|\vec{i}|,j}(\on{Taylor}_{\vec{i}}(g')) = 0$;
    \item $\gamma$ is $(MH)^{O_K(d^{O_K(1)})}$-rational;
    \item $\eps$ is $((MH)^{O_K(d^{O_K(1)})}, \vec{N})$-smooth.
\end{itemize}
\end{lemma}
\begin{proof}
By Lemma~\ref{lem:taylorexansion}, we may write
\[g(n) = \exp\Big(\sum_{|\vec{k}| \le s} \binom{n}{\vec{k}}g_{\vec{k}}\Big)\]
for some $g_{\vec{k}}\in\log(G_{(|\vec{k}|,0)}) = \log(G_{(|\vec{k}|,1)})$ with \[\on{Taylor}_{\vec{k}}(g) = \exp(g_{\vec{k}})\imod G_{(|\vec{k}|,2)}.\]
As $\log$ is a homorphism, we may treat $\psi_{|\vec{k}|, j}$ as a linear map on $\log(G_{(|\vec{k}|, 1)})$. 

Since $\on{dist}(\psi_{|\vec{k}|,j}(\on{Taylor}_{\vec{k}}(g)), \mb{Z})\le H\cdot N^{-|\vec{k}|}$ by assumption, by \cite[Lemma A.1]{LSS24b}, we may write $g_{\vec{k}} = g_{\vec{k},\mr{small}} + g_{\vec{k},\mr{rat}} + (g_i - g_{\vec{k},\mr{small}} - g_{\vec{k},\mr{rat}})$ such that $g_{\vec{k},\mr{rat}}$ is an $H^{O_K(d^{O_K(1)})}$-rational combination of elements in $\mc{X}\cap\log(G_{(|\vec{k}|,1)})$, such that $\snorm{g_{\vec{k},\mr{small}}}_{\infty}\le (MH)^{O_K(d^{O_K(1)})}\cdot N^{-|\vec{k}|}$, and such that $\psi_{|\vec{k}|,j}(g_{\vec{k}} - g_{\vec{k},\mr{small}} - g_{\vec{k},\mr{rat}}) = 0$. Defining
\[\gamma := \exp\Big(\sum_{|\vec{k}| \le s} \binom{n}{\vec{k}}g_{\vec{k},\mr{rat}}\Big),\quad\eps := \exp\Big(\sum_{|\vec{k}| \le s} \binom{n}{\vec{k}}g_{\vec{k},\mr{small}}\Big),\]
and $g' := \eps^{-1} g \gamma^{-1}$, we note that $\gamma$ and $\eps$ having the appropriate smooth and rational properties follow from standard Mal'cev bases computations (e.g., \cite[Lemma B.1, B.3, B.14]{Len23b}.)
\end{proof}
As the ergodic portion of this paper uses Lemma~\ref{lem:nilcharacters}, a result more standard in the finitary regime, we restate it for the convenience of the reader.
\begin{lemma}[Expansion into functions with vertical frequency]\label{lem:nilcharacters}
Let $G/\Gamma$ be a nilmanifold with dimension $d$, complexity $M$, degree $k$, and step $s$. Let $0 < \delta < 1/100$, and $F\colon G/\Gamma \to \mathbb{C}$ be $L$-Lipschitz. Given a $Q$-rational subgroup $H \subseteq Z(G)$, we have an approximation
$$\sup_{x \in G} \left|F(x\Gamma) - \sum_{|\xi| \le (\delta/(QL))^{-O_k(d^{O(1)})}} F_\xi(x\Gamma)\right| \le \delta$$
where $F_\xi$ is an $H$-vertical character of frequency $\xi$ and Lipschitz norm at most $(\delta/(QML))^{-O_k(d)^{O(1)}}$.
\end{lemma}

Below is a standard lemma (e.g., \cite{BHK05, Lei15, Mo20}) regarding decompositions of sequences into nilsequences and null-sequences.
\begin{lemma}\label{lem:standard}
Suppose given $\epsilon > 0$, a sequence $a(n)$ can be written as a sum of a degree $k$ nil-sequence $a_{\mathrm{nil}, \epsilon}(n)$, a null-sequence $a_{\mathrm{null}, \epsilon}(n)$, and a sequence $a_\epsilon(n)$ which is uniformly bounded by $\epsilon$. Then $a(n)$ is the sum of a generalized nilsequence and a null-sequence.
\end{lemma}
\begin{proof}
Inputting $\epsilon = 1/(2m)$ we obtain decompositions $a(n) = c_{\mathrm{nil}, m}(n) + c_{\mathrm{null}, m}(n) + c_{\mathrm{small}, m}(n)$ with
$$c_{\mathrm{nil}, m} = a_{\mathrm{nil}, \epsilon}, c_{\mathrm{null}, m} = a_{\mathrm{null}, \epsilon}, c_{\mathrm{small}, m} = a_\epsilon.$$
For $m_1 \neq m_2$, we have via the triangle inequality that
\begin{equation}\label{eq:nilnulltriangle}
|c_{\mathrm{nil}, m_1}(n) - c_{\mathrm{nil}, m_2}(n)| \le |c_{\mathrm{null}, m_1}(n) - c_{\mathrm{null}, m_2}(n)| + |c_{\mathrm{small}, m_1}(n) - c_{\mathrm{small}, m_2}(n)|    
\end{equation}
Let $A$ be the set of all $n \in \mathbb{Z}$ such that $|c_{\mathrm{nil}, m_1}(n) - c_{\mathrm{nil}, m_2}(n)| > \frac{1}{m_1} + \frac{1}{m_2}$. Since $|c_{\mathrm{small}, m_1}(n)| \le \frac{1}{2m_1}$ and $|c_{\mathrm{small}, m_2}| \le \frac{1}{2m_2}$, we see that for $n \in A$ and using \eqref{eq:nilnulltriangle}, $|c_{\mathrm{null}, m_1}(n) - c_{\mathrm{null}, m_2}(n)| > \frac{1}{2m_1} + \frac{1}{2m_2}$. Hence,
$$\limsup_{N \to \infty} \mathbb{E}_{n \in [\pm N]} 1_A(n) \le \limsup_{N \to \infty} \left(\frac{1}{2m_1} + \frac{1}{2m_2}\right)^{-1}\mathbb{E}_{n \in [\pm N]} |c_{\mathrm{null}, m_1}(n) - c_{\mathrm{null}, m_2}(n)| = 0.$$
Consequently, on $A^c$, $|c_{\mathrm{nil}, m_1}(n) - c_{\mathrm{nil}, m_2}(n)| \le \frac{1}{m_1} + \frac{1}{m_2}$ and $1_A$ has upper density zero. 

We claim that $|c_{\mathrm{nil}, m_1}(n) - c_{\mathrm{nil}, m_2}(n)| \le \frac{1}{m_1} + \frac{1}{m_2}$ for all $n \in \mathbb{Z}$. Suppose there exists some $n_0 \in \mathbb{Z}$ such that $|c_{\mathrm{nil}, m_1}(n_0) - c_{\mathrm{nil}, m_2}(n_0)| > \frac{1}{m_1} + \frac{1}{m_2}$. Given any nilsystem, the set of return times to any neighborhood on the underlying nilmanifold is syndetic. Since $c_{\mathrm{nil}, m_1} - c_{\mathrm{nil}, m_2}$ is a nilsequence, the set of all $n$ such that $|c_{\mathrm{nil}, m_1}(n) - c_{\mathrm{nil}, m_2}(n)| > \frac{1}{m_1} + \frac{1}{m_2}$ must be syndetic since it is non-empty, but it must also be a subset of $A$, which is not syndetic since it has upper density zero. This is a contradiction. 

It follows that $c_{\mathrm{nil},m}$ is a Cauchy sequence and thus by converges to a generalized nilsequence. It also follows that $c_{\mathrm{null}, m}$ is a Cauchy sequence that converges to a nullsequence and $c_{\mathrm{small}, m}$ converges uniformly to zero. The lemma follows.
\end{proof}

We next require the following result which is essentially \cite[Lemma 2.2]{Lei15}.
\begin{lemma}\label{lem:Leibmanresult2}
Let $\mathcal{N}$ be the set of all generalized nilsequences and $\mathcal{Z}$ the set of all nullsequences and $\mathcal{M} = \mathcal{N} + \mathcal{Z}$. Then $\mathcal{M}$ is closed in $\ell^\infty(\mathbb{Z})$ and each element of $\mathcal{M}$ is represented uniquely as the sum of an element of $\mathcal{N}$ and $\mathcal{Z}$. Furthermore, the projections of $\mathcal{M}$ to $\mathcal{N}$ and $\mathcal{Z}$ are continuous in $\ell^\infty(\mathbb{Z})$.
\end{lemma}
We next need a result that the integral combination of generalization nilsequences is still a generalized nilsequence, which is recorded as \cite[Theorem 0.5]{Lei15}. First, we must define what we mean by an integral combination of generalized nilsequences.
\begin{definition}
An \textbf{integral combination of (resp. generalized) nilsequences} is a sequence of the form
$$n \mapsto \int_\Omega a_\omega(n) d\nu(\omega)$$
where $(\Omega, \nu)$ is a probability space and $\omega \mapsto a_\omega(n)$ is absolutely integrable and for almost every $\omega$, $n \mapsto a_\omega(n)$ is a (resp. generalized) nilsequence.
\end{definition}
We now state the relevant result.
\begin{prop}\label{prop:Leibmanresult}
Any integral combination of generalized nilsequences can be written as a sum of a generalized nilsequence and a null-sequence.
\end{prop}
We also require facts about inverse limits of dynamical systems, which occupy the next two lemmas.
\begin{lemma}[Inverse limits respect joinings]\label{lem:inversetheoremjoining}
Suppose that $\mathbf{X} = (X, \mathcal{X}, \mu)$ is a probability space and for each $1 \le i \le k$, we have an increasing sequence of $\Sigma$ algebras $\Sigma_i^1 \subseteq \Sigma_i^2 \subseteq \Sigma_i^3 \subseteq \Sigma_{\mathrm{limit}}^i$ where $$\Sigma_{\mathrm{limit}}^i = \bigvee_{j = 1}^\infty \Sigma_j^i.$$
Then
$$\bigvee_{j = 1}^\infty (\Sigma_i^1 \vee \Sigma_i^2 \vee \cdots \vee \Sigma_i^k) = \Sigma_{\mathrm{limit}}^1 \vee \cdots \vee \Sigma_{\mathrm{limit}}^k.$$
\end{lemma}

\begin{lemma}[Approximation of inverse limits]\label{lem:inverselimitapproximation}
Let $\Sigma_1 \subseteq \Sigma_2 \subseteq \Sigma_3 \subseteq \cdots$ be a increasing sequence of $\sigma$-algebras on a measure space $\mathbf{X} = (X, \mathcal{X}, \mu)$ with inverse limit $\Sigma$. Let $f \in L^2(\Sigma)$ and $\epsilon > 0$. Then there exists $M_0$ such that for $m \ge M_0$, $\|f - \mathbb{E}_\mu(f|\Sigma_m)\|_{L^2(\mu)} < \epsilon$.
\end{lemma}

\section{Auxiliary additive combinatorics lemmas}
\begin{lemma}\label{lem:converse}
Fix $\delta\in(0,1/2)$, positive integers $\ell, \ell', k \le K$, $f_1, \dots, f_\ell:[N]^k \to \mathbb{C}$ be one-bounded, and let $F(g(n)\Gamma) \in \mathrm{Nil}^{\ell + \ell'}(M, m, k, 1)$. If $f\colon[N]^k\to\mb{C}$ is a $1$-bounded function such that
\[\bigg|\mb{E}_{n\in[N]^k}f(n) F(g(n)\Gamma) \cdot \prod_{i = 1}^\ell f_i(n)\bigg|\ge\delta,\]
then 
\[\snorm{f}_{U([N]^k, \dots, [N]^k, e_1[N], \dots, e_\ell[N])}\ge (\delta/(Mm))^{m^{O_K(1)}}.\]
\end{lemma}
\begin{proof}
Let $s = \ell + \ell'$. In the degenerate case when $s = 0$, we take a degree $s$ nilsequence of complexity $M$ to be a constant function $\psi$ bounded by $M$. This implies that 
\[|\mb{E}_{n\in[N]^k}f(n)|\ge\eps/M\]
and by Cauchy--Schwarz we have 
\[\mb{E}_{n,n'\in[N]^k}f(n)\ol{f(n')}\ge (\eps/M)^2.\]
By unwinding definitions this implies the case $s = 0$.

For larger $s$, by applying \cite[Lemma~A.6]{Len23b} we may assume that 
\[\bigg|\mb{E}_{n\in[N]^k}f(n) \ol{F_{\xi}(g(n)\Gamma)} \prod_{i = 1}^\ell f_{i}(n)\bigg|\ge (\eps/M)^{O_K(d^{O_K(1)})}\]
where $F_{\xi}$ is a $(M/\eps)^{O_K(d^{O_K(1)})}$-Lipschitz function with $G_s$-vertical frequency $\xi$ bounded in height by $(M/\eps)^{O_K(d^{O_K(1)})}$, after Pigeonhole. First, we assume that $\ell \ge 1$. Cauchy--Schwarz in all but the $n_\ell$ variable implies that
\begin{equation}\label{eq:cseell}
\mb{E}_{n\in[N]^k,h\in[\pm N]}\Delta_{e_\ell h} f(n) F_{\xi}(g(n + he_\ell)\Gamma)\ol{F_{\xi}(g(n)\Gamma)} \prod_{i = 1}^{\ell - 1} f_{i}(n)\ge (\eps/M)^{O_K(d^{O_K(1)})},
\end{equation}
where we extend $f$ by $0$ in the usual manner. We define 
\[G^{\Box} = \{(g,g')\colon g,g'\in G,g^{-1}g'\in G_2\}\]
and note that this has a filtration $(G^{\Box})_i = \{(g,g')\colon g,g'\in G_i,g^{-1}g'\in G_{i+1}\}$ by \cite[Lemma~A.3]{Len23b} (with $G^{\Box} = (G^{\Box})_{1}$). Let $\Gamma^{\Box} = (\Gamma\times\Gamma)\cap G^{\Box}$ and note that 
\[\wt{F}_{\xi}((x,y) (\Gamma\times\Gamma)) := F_{\xi}(x \Gamma) \ol{F_{\xi}(y\Gamma)}\]
is invariant under $G^{\triangle}_{s}$ and thus descends to $\overline{G^\square} = G^\square/G_s^\triangle$. We have the appropriate complexity and Lipschitz bounds by \cite[Lemma A.3]{Len23b}.

Let 
\[(g(0),g(he_\ell)) = \{(g(0),g(he_\ell))\} \cdot [(g(0),g(he_\ell))]\]
with $\|\psi(\{(g(0),g(he_\ell))\})\|_\infty \le 1$ and $[(g(0),g(he_\ell))]\in\Gamma\times\Gamma$. Define
\[g_h'(n) = \{(g(0),g(he_\ell))\}^{-1}(g(n),g(n+he_\ell))[(g(0),g(he_\ell))]^{-1};\]
this is easily seen to be a polynomial sequence with respect to $G^{\Box}$. Thus, absorbing $\{(g(0),g(he_\ell))\}$ in $\wt{F}_\xi$, we have
\[\mb{E}_{n\in[N]^k,h\in[\pm N]}\Delta_{he_\ell}f(n)\ol{\wt{F}_{\xi}(g_h'(n) \overline{\Gamma^\square})} \prod_{i = 1}^{\ell - 1} f_{i}(n)\ge (\eps/M)^{O_K(d^{O_K(1)})}.\]
Applying induction, and deduce that 
\[\mb{E}_{h\in[\pm N]}\snorm{\Delta_{he_\ell} f}_{U([N]^k, \dots, [N]^k, e_1[N], \dots, e_{\ell - 1}[N])}\ge (\eps/M)^{O_K(d^{O_K(1)})}.\]
The desired result follows. We now assume that $\ell = 0$. We instead square instead of Cauchy-Schwarz so instead of \eqref{eq:cseell}, we obtain
\[\mb{E}_{n\in[N]^k,h\in[\pm N]^k}\Delta_{h} f(n) F_{\xi}(g(n + h)\Gamma)\ol{F_{\xi}(g(n)\Gamma)}\ge (\eps/M)^{O_K(d^{O_K(1)})}.\]
We now instead define
\[(g(0),g(h)) = \{(g(0),g(h))\} \cdot [(g(0),g(h))]\]
with $\|\psi(\{(g(0),g(h))\})\|_\infty \le 1$, and $[(g(0),g(h))]\in\Gamma\times\Gamma$, and define
\[g_h'(n) = \{(g(0),g(h))\}^{-1}(g(n),g(n+h))[(g(0),g(h))]^{-1}.\]
Again, absorbing $\{g(0), g(h)\}$ in $\wt{F}_\xi$, we have
\[\mb{E}_{n\in[N]^k,h\in[\pm N]^k}\Delta_{h}f(n)\ol{\wt{F}_{\xi}(g_h'(n) \overline{\Gamma^\square})}\ge (\eps/M)^{O_K(d^{O_K(1)})}.\]
By induction, this implies that
$$\mathbb{E}_{h \in [\pm N]^k} \snorm{\Delta_{h}f}_{U([N]^k, \dots, [N]^k)} \ge (\eps/M)^{O_K(d^{O_K(1)})}$$
which implies the desired result.
\end{proof}

We next require the following lemma about approximate homomorphisms. The proof is almost identical to \cite[Lemma A.1]{LSS24b} so we omit it. (The only change is that our proper progression is replaced with $P + H$ with $H$ a subgroup; the point is that $H$ has rank bounded by the dimension of the subgroup, which in this case is $D'$. To ensure linear independence of $(\{\xi \cdot v_j\})_{\xi \in S}$ as in \cite[Lemma 10.4]{GT08b}, we complete $\xi$ to be a set of generators of $H$.)

\begin{lemma}\label{lem:approximate}
Fix $\delta\in (0,1/2)$, $D' \in \mathbb{N}$, let $H_1,H_2,H_3,H_4\subseteq[N]^{D'}$ and let functions $f_i\colon H_i\to\mb{R}^d$ be such that there are at least $\delta N^{3D'}$ additive tuples $h_1 + h_2 = h_3 + h_4$ with
\[\snorm{(f_1(h_1) + f_2(h_2) - f_3(h_3) - f_4(h_4))_j}_{\mb{R}/\mb{Z}}\le\eps_j\]
for all $1\le j\le d$. Then there exists $H_1'\subseteq H_1$ with $|H_1'|\ge\exp(-(d\log(1/\delta))^{O_{D'}(1)})N$ such that
\[\norm{\bigg(f_1(h) - \sum_{i=1}^{d'}a_i \{\alpha_i \cdot h\} - b\bigg)_j}_{\mb{R}/\mb{Z}} \le \eps_j\]
for all $h\in H_1'$, for appropriate choices of $d'\le (d\log(1/\delta))^{O_{D'}(1)}$, $a_i,b\in\mb{R}^d$, and $\alpha_i\in(1/N')\mb{Z}^{D'}$ where $N'$ is a prime between $100N$ and $200N$.
\end{lemma}

We finally require a slightly stronger version of the splitting lemma.
\begin{lemma}\label{lem:splitdegreerank}
Let $J$ and $J'$ be finite downsets in $\mb{N}^k$ and fix $\eps \in (0,1/2)$. Suppose that $\beta(h_1,\ldots,h_k)$ is a nilsequence of multidegree-rank $(J\cup J', r)$ nilsequence with complexity and dimension $(M,d)$. Then there exists $1\le L\le (M/\eps)^{O_{J,J'}(d^{O_{J,J'}(1)})}$ such that 
\[\norm{\beta(h_1,\ldots,h_k) - \sum_{j=1}^{L}\beta_j(h_1,\ldots,h_k)\beta_j'(h_1,\ldots,h_k)}_{L^{\infty}(\mb{Z}^k)}\le\eps\]
with the $\beta_j$ being nilsequences of multidegree-rank $(J, r)$, the $\beta_j'$ being  nilsequences of multidegree-rank $(J', r)$, and $\beta_j,\beta_j'$ having complexity $((M/\eps)^{O_{J,J'}(d^{O_{J,J'}(1)})},d^{O_{J,J'}(1)})$.
\end{lemma}
\begin{proof}
As the proof in \cite[Lemma C.6]{LSS24b} is very lengthy and the changes required from \cite[Lemma C.6]{LSS24b} minimal, we only indicate the required changes and do not replicate the proof. 
We let 
\[\beta(h_1,\ldots,h_k) = F(g(h_1,\ldots,h_k)\Gamma)\]
where the underlying nilmanifold is $G/\Gamma$. As is standard, we may assume that $g(0,\ldots,0)=\mr{id}_G$ up to the insignificant change of adjusting $M$ to $M^{O_{J,J'}(d^{O_{J,J'}(1)})}$. Furthermore let the adapted Mal'cev basis for $G$ be $\mc{X}$. Instead of defining $G_{t}^{\vec{j}}$ in \cite[Lemma C.6]{LSS24b}, we define 
$$G_{(t, r')}^{\vec{j}} = \bigvee_{|\vec{i}| = t} G_{(\vec{i} + \vec{j}, r')}.$$
Note that this gives a degree-rank filtration for $G^{\vec{j}}_{(0, 0)}$. As these subgroups are all $M$-rational with respect to $\mc{X}$, there exists a Mal'cev basis $\mc{X}^{\vec{j}}$ adapted to this filtration of complexity $M^{O_{J,J'}(d^{O_{J,J'}(1)})}$ where each element is an $M^{O_{J,J', r}(d^{O_{J,J'}(1)})}$-rational combination of elements in $\mc{X}$ by \cite[Lemma~B.11]{Len23b}.

Using a variant of \cite[Lemma 10.2]{LSS24b}, adapted to multidegree filtrations, we may write
\[g(h_1,\ldots,h_k) = \prod_{\vec{j}\neq \vec{0}}\prod_{X_{\vec{j},i}\in\mc{X}^{\vec{j}}}\exp(X_{\vec{j},i})^{\alpha_{\vec{j},i}\prod_{\ell=1}^{k}(h_{\ell}^{j_{\ell}}/j_{\ell}!)}.\]
The product here is taken in $\vec{j}$ is increasing $|\vec{j}|$ and then lexicographic order and $X_{\vec{j},i}$ taken in increasing order of $i$.

We now lift to the universal nilmanifold. We define the universal nilmanifold $\wt{G}$ to be generated by generators $\exp(e_{\vec{j},i})^{t_{\vec{j},i}}$ for $\vec{j}\neq \vec{0}$, $1\le i\le\dim(G_{\vec{j}})$, and $t_{\vec{j},i}\in\mb{R}$. We now make a final change to the proof in \cite[Lemma C.6]{LSS24b}. We make the $r^*$-fold commutator $\exp(e_{\vec{j_1}, i_1}), \dots, \exp(e_{\vec{j_{r^*}, i_1}})$ vanish if
\begin{itemize}
    \item $j_1 + \cdots + j_{r^*} \not\in J \cup J'$ or
    \item $j_1 + \cdots + j_{r^*}$ is maximal in $J \cup J'$ and $r^* > r$.
\end{itemize}
Now fix $(d_1, \dots, d_k) \in I$. By defining $\widetilde{G}_{(d_1, \dots, d_k, r')}$ to be generated by the iterated commutators $\exp(e_{\vec{j_1}, i_1}), \dots, \exp(e_{\vec{j_{r^*}}, i_{r^*}})$ with $\vec{j_1} + \cdots + \vec{j_{r^*}}$ larger than $(d_1, \dots, d_k)$ in each component or if $\vec{j_1} + \cdots + \vec{j_{r^*}} = (d_1, \dots, d_k)$ and $r^* \ge r'$, we have a valid multidegree-rank filtration on $\widetilde{G}$ (see \cite[Lemma 10.4]{LSS24b}). The rest of the proof proceeds as in \cite[Lemma C.6]{LSS24b}. Note that even if these changes are implemented, $\widetilde{G}_{>J}$ with is the group generated by $\widetilde{G}_{(i, 0)}$ for all $i \in J' \setminus J$ and $\widetilde{G}_{>J'}$ are still normal and still generate $\widetilde{G}$ (because $\widetilde{G}_{> J} \cap \widetilde{G}_{>J'} = \mathrm{Id}_{\widetilde{G}}$) and $\widetilde{G}/\widetilde{G}_{> J}$ can be given a multidegree-rank $(J', r)$ filtration and $\widetilde{G}/\widetilde{G}_{> J'}$ can be given a multidegree-rank $(J, r)$ filtration.
\end{proof}

\section{From structured extensions to pleasant extensions}\label{sec:structuredtopleasant}
In the following section, we note how given a structured extension type result as in Theorem~\ref{thm:ergodicinversetheorem}, we can deduce a pleasant extensions result as in \cite{A09}.

We use the notion of idempotent classes from \cite{AusThe}. We now detail a construction which deduces a pleasant extensions result from a structured extension result. Let $C^1, \dots, C^n$ and $(D_i^j)_{i \in [m], j \in [n]}$ be order preserving and continuous idempotent classes such that $D_i^j\mathbf{X} \subseteq C^j\mathbf{X}$ for any system $\mathbf{X}$, i.e., $D_i^j\mathbf{X}$ is a factor of $C^j\mathbf{X}$.

\begin{theorem}\label{thm:structuredtopleasant}
Suppose $\mathbf{X}$ is a $\Gamma$-system such that there exists a $\Gamma$-system $\mathbf{Y}^j$ with $C^j\mathbf{X} = \bigvee_{i = 1}^m D_i^j\mathbf{Y}^j$. Then there exists an ergodic extension $\widetilde{\mathbf{X}}$ of $\mathbf{X}$ with $C^j\widetilde{\mathbf{X}} = \bigvee_{i = 1}^m D_i^j\mathbf{\widetilde{X}}$ for all $j \in [n]$.
\end{theorem}
\begin{proof}
Suppose we have established the theorem for $C^1, \dots, C^{\ell - 1}$. Note that the claims of
$$C^j\widetilde{\mathbf{X}} = \bigvee_{i = 1}^m D_i^j\mathbf{\widetilde{X}}$$
are closed under taking inverse limits of systems. (This is because taking inverse limits commutes with taking joinings.) We shall say that a system $\mathbf{Z}$ has property $P(\ell - 1)$ if
$$C^j\mathbf{Z} = \bigvee_{i = 1}^m D_i^j\mathbf{Z}$$
for all $i \in [\ell - 1]$.

We shall first construct an extension $\mathbf{X}_\infty$ which satisfies the property $Q$ of
$$C^\ell\mathbf{X}_\infty = \bigvee_{i = 1}^m D_i^j\mathbf{X}_\infty.$$
Consider the following diagram. 
\begin{center}
\begin{tikzcd}
\mathbf{X_1} \arrow[dd] &                                                  & \mathbf{X}_2 \arrow[ll] \arrow[ld] \arrow[dd] &                                                  & \cdots \arrow[ld] \arrow[ll] \\
                        & \mathbf{Y_1} \arrow[ld] \arrow[d]                &                                               & \mathbf{Y}_2 \arrow[ll] \arrow[d] \arrow[ld]     & \cdots \arrow[l]             \\
C^\ell\mathbf{X_1}      & \bigvee_{j = 1}^m D_j^\ell\mathbf{Y}_1 \arrow[l] & C^\ell\mathbf{X}_2 \arrow[l]                  & \bigvee_{j = 1}^m D_j^\ell\mathbf{Y}_2 \arrow[l] & \cdots \arrow[l]            
\end{tikzcd}
\end{center}
Here, $\mathbf{X}_1 = \mathbf{X}$, $\mathbf{Y}_i = \mathbf{Y}^\ell(\mathbf{X}_i)$ and $\mathbf{X}_{i + 1}$ is an ergodic relatively independent joining of $\mathbf{X}_i$ and $\mathbf{Y}_i$. Note that the arrow from $C^\ell\mathbf{X}_{i + 1}$ to $\bigvee_{i = 1}^m D_i^\ell\mathbf{Y}_i$ exists because the arrow from $C^\ell(\mathbf{X}_{i + 1})$ to $\mathbf{Y}_i$ exists and $C^\ell(\mathbf{Y}_i) \supseteq \bigvee_{j = 1}^m D_j^\ell\mathbf{Y}_i$. Let $\mathbf{X}_\infty = \lim_{\leftarrow} \mathbf{X}_i$ and $\mathbf{Z}_\infty = \lim_{\leftarrow} C^\ell \mathbf{X}_i$. Note that since $C^\ell$ preserves inverse limits, $C^\ell \mathbf{X}_\infty = \mathbf{Z}_\infty$. In addition, since $C^\ell\mathbf{X}_i \subseteq \bigvee_{i = 1}^m D^\ell_i \mathbf{X}_{i + 1} \subseteq C^\ell\mathbf{X}_{i + 1}$ by our hypotheses, it follows that under limits, we see that $C^\ell\mathbf{X}_\infty = \bigvee_{i = 1}^m D^\ell_i \mathbf{X}_{\infty}$ as desired. 

Now given a system $\mathbf{X} = \mathbf{X}_1$, we iteratively construct extensions $\widetilde{\mathbf{X}}_i$ such that
\begin{itemize}
    \item If $i$ is odd, $\widetilde{\mathbf{X}}_i$ satisfies property $P(\ell - 1)$
    \item If $i$ is even, $\widetilde{\mathbf{X}}_i$ satisfies property $Q$.
\end{itemize}
Taking an inverse limit, we obtain $\widetilde{\mathbf{X}}_\infty$ and letting $\widetilde{\mathbf{X}} = \widetilde{\mathbf{X}}_\infty$, we obtain a system which satisfies both $P(\ell - 1)$ and $Q$, and hence $P(\ell)$.
\end{proof}
Using this result, we deduce the following variant of the main inverse theorem:
\begin{theorem}\label{thm:ergodicinversetheorem2}
Fix integers $k, j, j'$ with $0 < j + 1, j' + 1, k$ and $j \le k$. Let $(X, \mathcal{X}, \mu, \vec{T})$ be an ergodic $\mathbb{Z}^k$ system and let $\mathbf{Z} = \mathbf{Z}_{T_1, T_2, \dots, T_j, \vec{T}, \dots, \vec{T}}$ be the Host-Kra factor for the seminorm $\|\cdot \|_{T_1, T_2, \dots, T_j, \vec{T}, \dots, \vec{T}}$ with $j' + 1$ copies of $\vec{T}$. Then $\mathbf{X}$ admits an extension to an ergodic system $(\widetilde{X}, \widetilde{\mathcal{X}}, \widetilde{\mu}, \widetilde{\vec{T}})$ with
$$\mathbf{Z}_{\mathbf{\widetilde{X}}, \widetilde{T_1}, \dots, \widetilde{T_j}, \widetilde{\vec{T}}, \dots, \widetilde{\vec{T}}} = I(\widetilde{T_1}) \vee I(\widetilde{T_2}) \vee \cdots \vee I(\widetilde{T_j}) \vee \Xi_{j + j', \mathrm{pronil}}$$
where $\Xi_{j + j', \mathrm{pronil}}$ is the largest inverse limit of $j + j'$-step $\mathbb{Z}^k$-nilfactors and if $S \in \mathrm{Aut}(\mathbf{X})$, $I(S)$ denotes the sigma algebra of $S$-invariant sets in $\wt{\mathcal{X}}$.
\end{theorem}

\begin{proof}
First, note that all multidimensional Host-Kra factors are order preserving and continuous idempotent classes. Order preserving and continuity can be seen since Host-Kra factors as being generated by dynamical dual functions. To show \emph{idempotentness}, we use that the function's Host-Kra-Gowers multidimensional norm vanishes if and only if its projection to the multidimensional Host-Kra factor is zero. Idempotentness then follows from the fact that a multidimensional Host-Kra factor-measurable function's multidimensional Host-Kra-Gowers norm is equal to the multidimensional Host-Kra-Gowers norm of the original system, which follows from \cite[Lemma 4]{Ho09} and the multidimensional Host-Kra-Gowers inequality.

Since the norm $\|\cdot\|_{T_1, \dots, T_j, \vec{T}, \dots, \vec{T}}$ controls both the norms $\|\cdot\|_{T_i}$ and $\|\cdot\|_{\vec{T}, \dots, \vec{T}}$ (with $j' + 1 + j$ copies of $\vec{T}$ (since $I(\vec{T}^{[k]}) \subseteq I(T_i^{[k]})$ on any multidimensional Host-Kra cubic measure space), it follows that the idempotent class of $\mathbf{Z}_{T_1, \dots, T_j, \vec{T}, \dots, \vec{T}}$ contains $I(T_i)$ and $\mathbf{Z}_{\vec{T}, \dots, \vec{T}}$. The theorem them follows from Theorem~\ref{thm:structuredtopleasant} and Theorem~\ref{thm:ergodicinversetheorem}.
\end{proof}

\bibliographystyle{amsplain0.bst}
\bibliography{main.bib}

\end{document}